\documentclass[reqno]{amsart}

\usepackage{amsmath,amssymb}
\usepackage{amsfonts, amscd, epsfig, amsmath, amssymb,enumerate}
\usepackage{amstext}
\usepackage{graphicx}
\usepackage[usenames,dvipsnames,svgnames,table]{xcolor}
\usepackage{mathrsfs}

\usepackage[normalem]{ulem}
\usepackage[charter]{mathdesign}
\usepackage{cancel}

\usepackage{tikz}
\usetikzlibrary{backgrounds}
\usetikzlibrary{patterns,fadings}
\usetikzlibrary{arrows,decorations.pathmorphing}
\usetikzlibrary{calc}
\definecolor{light-gray}{gray}{0.95}
\usepackage{float}
\usepackage[colorlinks=true,linkcolor=blue,citecolor=magenta]{hyperref}

\allowdisplaybreaks

\newcommand{\pfrac}[2]{\genfrac{}{}{}{1}{#1}{#2}}

\newtheorem{theorem}{Theorem}[section]
\newtheorem{lemma}[theorem]{Lemma}
\newtheorem{proposition}[theorem]{Proposition}
\newtheorem{corollary}[theorem]{Corollary}
\newtheorem{remark}[theorem]{Remark}
\newtheorem{definition}[theorem]{Definition}

\renewcommand\vec[1]{\overrightarrow{#1}}
\newcommand\vecleft[1]{\overleftarrow{#1}}
\renewcommand{\bar}{\overline}
\renewcommand{\leq}{\leqslant}
\renewcommand{\le}{\leqslant}
\renewcommand{\geq}{\geqslant}
\renewcommand{\ge}{\geqslant}
\renewcommand{\tilde}{\widetilde}

\numberwithin{equation}{section}

\newcommand{\mc}[1]{{\mathcal #1}}
\newcommand{\mf}[1]{{\mathfrak #1}}
\newcommand{\mb}[1]{{\mathbf #1}}
\newcommand{\bb}[1]{{\mathbb #1}}
\newcommand{\mr}[1]{{\mathrm #1}}

\renewcommand{\epsilon}{\varepsilon}

\newcommand{\RR}{\mathbb R}
\newcommand{\NN}{\mathbb N}

\newcommand{\tmop}[1]{\ensuremath{\operatorname{#1}}}

\newcommand{\Ctwo}{\ensuremath{\mc S_{\tmop{Dir}}}}
\newcommand{\SD}{\ensuremath{\mc S_{\tmop{Dir}}}}
\newcommand{\SN}{\ensuremath{\mc S_{\tmop{Neu}}}}

\def\centerarc[#1](#2)(#3:#4:#5){\draw[#1] ($(#2)+({#5*cos(#3)},{#5*sin(#3)})$) arc (#3:#4:#5);}

\let\oldtocsection=\tocsection
\let\oldtocsubsection=\tocsubsection
\let\oldtocsubsubsection=\tocsubsubsection
\renewcommand{\tocsection}[2]{\hspace{0em}\oldtocsection{#1}{#2}}
\renewcommand{\tocsubsection}[2]{\hspace{1em}\oldtocsubsection{#1}{#2}}
\renewcommand{\tocsubsubsection}[2]{\hspace{2em}\oldtocsubsubsection{#1}{#2}}
\DeclareRobustCommand{\SkipTocEntry}[5]{}

\begin{document}

\title[Stochastic Burgers equation with Dirichlet boundary conditions]{Derivation of the stochastic Burgers equation with Dirichlet boundary conditions from the WASEP}

\author{Patr\' icia  Gon\c calves}

\address{Center for Mathematical Analysis,  Geometry and Dynamical Systems,
Instituto Superior T\'ecnico, Universidade de Lisboa,
Av. Rovisco Pais, 1049-001 Lisboa, Portugal and   Institut  Henri
Poincar\'e, UMS 839 (CNRS/UPMC), 11 rue Pierre et Marie Curie, 75231 Paris Cedex 05, France.}
\email{{\tt patricia.goncalves@math.tecnico.ulisboa.pt}}

\author{Nicolas Perkowski}
\address{Humbolt-Universit\"at zu Berlin, Institut f\"ur Mathematik, Unter den Linden 6, 10099 Berlin, Germany.}
\email{}

\author{Marielle Simon}
\address{\noindent INRIA Lille Nord Europe, \'equipe-projet MEPHYSTO, 40 avenue du Halley, 59650 Villeneuve d'Ascq, France. }
\email{marielle.simon@inria.fr}


\begin{abstract}
We consider the weakly asymmetric simple exclusion process on the discrete space $\{1,...,n-1\}\ (n\in\bb N)$, in contact with stochastic  reservoirs, both with density $\rho\in{(0,1)}$ at the extremity points, and starting from the invariant state, namely the Bernoulli product measure of parameter $\rho$. Under time diffusive scaling $tn^2$ and for $\rho=\frac12$, when the asymmetry parameter is taken of order $1/ \sqrt n$, we prove that the density fluctuations at stationarity  are macroscopically governed by the  energy solution of the stochastic Burgers equation with Dirichlet boundary conditions, which is shown to be  unique and different from the Cole-Hopf solution.
\end{abstract}

\maketitle


\section{Introduction}\label{s1}

\subsection{Context}
A vast amount of physical phenomena  were first described
at the macroscopic scale, in terms of the classical partial differential equations (PDEs) of mathematical physics. Over the last decades the scientific community has tried to give a precise understanding of their derivation from first principles at the microscopic level in order to identify the limits of their validity. Typically, the microscopic systems are composed of a huge number of atoms
and one  looks at a very large time scale with respect to the typical
frequency of atom vibrations.
Mathematically, this corresponds to a space-time scaling limit procedure.

The macroscopic laws that can arise from microscopic systems can either be partial differential equations (PDEs) or stochastic PDEs (SPDEs) depending on whether one is looking at the convergence to the mean or at the fluctuations around that mean. Among the classical SPDEs is the Kardar-Parisi-Zhang (KPZ) equation which   has been first introduced more than thirty years ago in  \cite{KPZ} as the \emph{universal} law describing the fluctuations of randomly growing interfaces of one-dimensional
stochastic dynamics close to a stationary state (as for example, models of bacterial growth, or fire propagation). Since then, it has generated an intense research activity among the physics and mathematics community. In particular, the  \emph{weak KPZ universality conjecture} \cite{Corwin12, QuaS, Sp16} states that  the  fluctuations of a large class
of one-dimensional microscopic interface growth models
are ruled at the macroscopic scale by  solutions of the \emph{KPZ equation}, which reads as follows: consider a time variable $t$ and a one-dimensional space variable $u$, then the evolution of the height $\mc Z(t,u)$ of the randomly growing interface can be described by
\begin{equation}
d\mc Z(t,u) = A \Delta \mc Z(t,u)\; dt + B \big( \nabla \mc Z(t,u)\big)^2 \; dt + \sqrt C d \mc W_t, \label{eq:KPZintro}
\end{equation}
where $A,B,C$ are thermodynamic constants depending on the model and $\partial_t\mc W$ is a space-time white noise. Note that the non-linearity $( \nabla \mc Z(t,u))^2$
makes the KPZ equation \eqref{eq:KPZintro} ill-posed, essentially because the trajectory of the solution lacks space regularity  (due to the  presence of the white noise), and therefore the square of its distributional derivative is not defined. 

One possible way to
solve this equation is to consider its \emph{Cole-Hopf transformation} $\Phi$, which solves a stochastic heat equation with a multiplicative noise (SHE) and is related to $\mc Z$ through a $\log$--transformation. Since the SHE is linear, its solution can easily
be constructed and it is unique, and the solution to the KPZ equation can then simply be defined as the inverse Cole-Hopf transformation of $\Phi$. However, the solution to the SHE is too irregular to allow for a change of variable formula, and a priori there is no meaningful equation associated to its inverse Cole-Hopf transformation.
 Only recently Hairer has developed
a meaningful notion of solution for the KPZ equation and proved existence and uniqueness of such solutions with periodic boundary conditions, see \cite{Hairer2013}. His approach uses rough path integrals to construct the nonlinearity $(\nabla \mc Z(t,u))^2$, and it has inspired the development of new technologies (regularity structures~\cite{Hairer2014} and paracontrolled distributions~\cite{Gubinelli2015Paracontrolled}) for the so-called \emph{singular} stochastic partial differential equations (SPDEs).

The first breakthrough towards the  weak KPZ universality conjecture is due to Bertini and Giacomin: in their seminal paper \cite{BerG}, they show that the Cole-Hopf solution can be obtained as a scaling limit of the \emph{weakly asymmetric exclusion process} (WASEP) (which will be defined ahead). 
Their approach consists in performing the Cole-Hopf transformation at the microscopic level, following \cite{GAR}, and then showing that this microscopic Cole-Hopf transformation solves a linear equation (similarly to what happens at the macroscopic level). Since then, this strategy has been used in more sophisticated models, see \cite{Corwin2016,CST,corwin2017,Dembo2016},  however the applicability of the microscopic Cole-Hopf transformation is  limited to a very specific class of particle systems.

Another way to look at the KPZ equation is via the \emph{stochastic Burgers equation} (SBE), which is obtained from \eqref{eq:KPZintro} by taking its derivative: if $\mc Y_t = \nabla \mc Z_t$, then $\mc Y_t$ satisfies
\begin{equation}
\label{eq:SBEintro}
d\mc Y(t,u) = A \Delta \mc Y(t,u)\; dt + B\nabla \big( \mc Y^2(t,u)\big) \; dt + \sqrt C d\nabla\mc W_t,
\end{equation}
which has of course the same regularity issues as the KPZ equation. Nevertheless, this formulation is well adapted to derive KPZ behavior from  microscopic models. Indeed, the work initiated by Gon\c{c}alves and Jara in \cite{gj2014} has introduced a new tool, called \emph
{second order Boltzmann-Gibbs principle} (BGP), which makes the non-linear term $\nabla ( \mc Y^2(t,u))$ of the SBE  naturally emerge from
the underlying microscopic dynamics. The authors have first proved the BGP for general weakly asymmetric simple  exclusion processes, and 
shortly  thereafter it has been extended to a wider class of microscopic systems, such as
 zero-range models \cite{GJS}, integrable Hamiltonian one-dimensional chains \cite{GJSi}, non-degenerate kinetically constrained
models \cite{blondel2016}, exclusion processes with defects \cite{fgsimon2015}, non-simple exclusion processes \cite{GJLR}, semilinear SPDEs \cite{Gubinelli2016Hairer}, or interacting diffusions \cite{Diehl2017}.
From the BGP, it comes naturally that some suitably rescaled microscopic quantity, called \emph{density fluctuation field} (see below for a precise meaning)  subsequentially converges, as the size of the microscopic system goes to infinity, to random
fields which are solutions $\mc Y$ of a generalized martingale problem for \eqref{eq:SBEintro}, where
the singular non-linear drift  $\nabla ( \mc Y^2(t,u))$ is a well-defined space-time distributional random field. Gon\c{c}alves and Jara in \cite{gj2014} (see also \cite{Gubinelli2013}) called them \emph{energy solutions}. Recently, Gubinelli and Perkowski \cite{Gubinelli2015Energy} proved uniqueness of  energy solutions to \eqref{eq:SBEintro} and as a significant consequence, the proof of the weak KPZ universality conjecture could be concluded for all the models mentioned above. We note that
none of these models admits a microscopic Cole-Hopf transformation, which prevents the use of the methods of \cite{BerG}.

\subsection{Purposes of this work}

 Our goal in this article is to go beyond the weak KPZ universality conjecture and to derive a new SPDE, namely,  the KPZ equation with boundary conditions, from an interacting particle system in contact with stochastic reservoirs.
Indeed, the presence of boundary conditions in evolution equations  often lacks understanding from a physical and mathematical point of view. Here we intend to legitimate the choice done at the macroscopic level for the KPZ/SBE equation from the microscopic description of the system. 
For that purpose, we first prove two main theorems:\begin{itemize}
\item[(Theorem \ref{def:energy})] We extend the notion of energy solutions to the stochastic Burgers equation \eqref{eq:SBEintro} by adding Dirichlet boundary conditions: we set up a rigorous definition and prove existence and uniqueness of such solutions.
\item[(Theorem \ref{theo:convergence})] We construct a microscopic model (inspired by the WASEP mentioned above) from which the energy solution naturally emerges as the macroscopic limit of its stationary density fluctuations. 
\end{itemize}

This gives a physical justification for the Dirichlet boundary conditions in \eqref{eq:SBEintro}. We also introduce the notion of energy solutions to two related SPDEs: the KPZ equation with Neumann boundary conditions and the SHE with Robin boundary conditions; we prove their existence and uniqueness, and we rigorously establish the formal links between the equations discussed above. This is more subtle than expected, because the boundary conditions do not behave canonically: a formal computation suggests that the Cole-Hopf transform of the KPZ equation with Neumann boundary conditions should also satisfy Neumann boundary conditions, but we show that instead it satisfies Robin boundary conditions.

We also  associate an interface growth model to our microscopic model, roughly speaking by integrating it in the space variable, and show that it converges to the energy solution of the KPZ equation, thereby giving a physical justification of the Neumann boundary conditions.

In the remaining lines we go into further details and explain these results.

\subsubsection{WASEP with reservoirs} At the microscopic level, the model that we consider is the following: let a discrete system of particles  evolve on the finite set $\Lambda_n=\{1,\dots,n-1\}$ of size $n\in\bb N$. For any site $x\in\Lambda_n$ there is at most one particle, and we denote by $\eta(x)\in\{0,1\}$ its occupation variable, namely: $\eta(x)=1$ if there is a particle at site $x$, and 0 otherwise. We then define a continuous-time Markov process $\eta_t=\{\eta_t(x)\; ; \; x\in\Lambda_n\}$ on the state space $\{0,1\}^{\Lambda_n}$ using the following possible moves: for any $x \notin\{1,n-1\}$, a particle at site $x$ can attempt to jump to its neighbouring sites $x-1$ or $x+1$, provided that they are empty. Similarly, a particle at site $1$ can jump to its right neighbour $2$ or it can leave the system, and the particle at site $n-1$ can jump to its left neighbour $n-2$ or it can leave the system. Moreover, attached to the extremities of $\Lambda_n$ there are two reservoirs of particles: one at site $0$, which can send a particle to site $1$ (if this site is empty), and the other one at site $n$, which can send a particle to site $n-1$ (if this site is empty). Another interpretation of the boundary dynamics could be given as follows: particles  can either be created (resp. annihilated)  at the sites $x=1$ or $x=n-1$ if the site is empty (resp. occupied). 

All possible moves are endowed with random Poissonian clocks, independently of each other. Along the time evolution, we launch the jump whose clock rings first, and after the jump all clocks are reset. Namely, we are given a family of independent Poisson processes, indexed by all the possibles moves $x \curvearrowleft y $, with intensities $\lambda(x,y)$, on the time line $[0,\infty)$.  The intensities $\lambda(x,y)$ depend on the occupation variables $\eta(x),\eta(y)$ and on some small parameter $\varepsilon_n >0$ as follows:

\begin{itemize}
\item if $x \in \{2,\dots,n-2\}$, then we assume \begin{align}
\lambda(x,x+1) & = (1+\varepsilon_n)\eta(x)\big(1-\eta(x+1)\big), \label{eq:jumpright}\\
\lambda(x,x-1) & = \eta(x)\big(1-\eta(x-1)\big), \label{eq:jumpleft}
\end{align} 
\item while at the boundaries, for some fixed $\rho \in(0,1)$ we set
\begin{align}
\lambda(0,1) & = \rho(1+\varepsilon_n) \big(1-\eta(1)\big), \qquad \lambda(n-1,n) = (1-\rho)(1+\varepsilon_n)\eta(n-1), \label{eq:rightboundary}\\
\lambda(1,0)& = (1-\rho)\eta(1), \qquad   \qquad \quad\;\;\lambda(n,n-1) = \rho\big(1-\eta(n-1)\big).\label{eq:leftboundary}
\end{align}
\end{itemize}
In \eqref{eq:jumpright} and \eqref{eq:rightboundary}, the factor $1+\varepsilon_n$ breaks the symmetry of the jumps: there is a non-trivial drift towards the right. But since $\varepsilon_n \to 0$, we are in the weakly asymmetric setting. Moreover,  note that the product $\eta(x)(1-\eta(x+1))$ in \eqref{eq:jumpright} (for instance) corresponds to the \emph{exclusion rule} explained above: for the jump $x\curvearrowleft x+1$ to have a non-zero intensity, there has to be a particle at site $x$, and no particle at site $x+1$.

An invariant measure for these dynamics is the Bernoulli product measure on $\{0,1\}^{\Lambda_n}$ of parameter $\rho$, which we denote by  $\nu_\rho^n$, and whose marginals  satisfy: $$\nu_\rho^n\big\{\eta(x)=1\big\}=1-\nu_\rho^n\big\{\eta(x)=0\big\}=\rho.$$
We start the Markov process $\eta_t$ under the invariant measure $\nu_\rho^n$ and we look at the evolution on the diffusive time scale $tn^2$, where $t\ge 0$ is the macroscopic time. The \emph{microscopic density fluctuation field} is then defined as 
\begin{equation}
\label{eq:fieldintro}
\mc Y_t^n(\cdot) = \frac{1}{\sqrt n}\sum_{x=1}^{n-1} \big(\eta_{tn^2}(x) - \rho \big) \delta_{x /n}(\cdot),
\end{equation}
where $\delta_{x/n}$ is the Delta dirac distribution, and therefore $(\cdot)$ is meant to be a test function. We prove in Theorem \ref{theo:convergence} two main results on the large $n$ limit of that field, assuming that the initial density is $\rho=\frac12$ (this assumption, which aims at removing a transport phenomenon inside the system, will be explained in detail in Remark \ref{rem:rho}):
\begin{itemize} 
\item if $\varepsilon_n = E/\sqrt n$ (for some $E>0$), the sequence of processes $\mc Y^n$ converges, in a suitable space of test functions, towards the unique energy solution $\mc Y$ to \eqref{eq:SBEintro} with Dirichlet boundary conditions, with parameters $A=1$, $B=E$ and $C=\frac12$;
\medskip
\item whenever $\sqrt n\varepsilon_n \to 0$ as $n\to \infty$,  the sequence $\mc Y^n$ converges towards the unique Ornstein-Uhlenbeck  (OU) process with Dirichlet boundary conditions, which is equivalent to the unique energy solution to \eqref{eq:SBEintro} with parameters $A=1$, $B=0$ and $C=\frac12$.
\end{itemize}

\subsubsection{The second order Boltzmann-Gibbs principle} The main ingredient that we use at the microscopic level is the BGP, that we have to reprove completely in our particular setting. Indeed, this is the first time that the BGP is established in a space with boundaries.

This tool, which was first introduced in \cite{gj2014}, permits to investigate the space-time fluctuations of the microscopic density fluctuation field, and to prove that, when properly recentered, the latter is close in the macroscopic  limit to a quadratic functional of the conservative field.
It was originally proved for general weakly asymmetric simple exclusion processes by a multi-scale argument (also given in \cite{G}),  which consists in replacing, at each step,  a local function of the dynamics by another function whose
support increases  and whose variance decreases, and its proof used a key spectral gap inequality which is uniform in the density of particles, and is not available for many models. Later in \cite{GJS} it is assumed that the models satisfy a spectral gap bound which does not
need to be uniform in the density of particles and allows more general  jump rates. More recently in \cite{fgsimon2015,GJSi}, and then in \cite{blondel2016}, a new proof of the BGP has permitted
to extend the previous results to models which do not need to satisfy a spectral gap bound,
as, for example, symmetric exclusion processes with a slow bond, and microscopic dynamics that
allow degenerate exchange rates. In this paper, we adapt that strategy, which turns out to be quite robust, to our finite model with stochastic boundary reservoirs. This is the goal of Theorem \ref{theo:BG}.

\subsubsection{Boundary behavior at the microscopic scale}

As we already mentioned, the KPZ equation and SBE are closely related, and this relation can be seen also at the microscopic level. There is a natural \emph{height function} $h(x)$ associated to the occupation variable $\eta(x)$, and in particular its increments satisfy: 
\[h(x+1)-h(x) = \eta(x) -\rho.\] Because of the presence of the reservoirs in our dynamics, the microscopic height process $\{h_t(x)\; ; \; t \geq 0, x=1,\dots,n\}$, which is derived from the Markov process $\eta_t$ defined above, has to be carefully defined. We refer the reader to Section \ref{ssec:height} for a rigorous definition. Similarly to \eqref{eq:fieldintro}, we are interested in the macroscopic limit of height fluctuations  starting the evolution from $\nu_\rho^n$ and looking in the time scale $tn^2$. In this case the averaged local height at site $x \in \{1,\dots,n\}$ and microscopic time $tn^2$ is equal to $c_n t$, where $c_n$ is related to the initial density $\rho$ and the strength of the asymmetry $\varepsilon_n$, as follows: $c_n=n^2\varepsilon_n\rho(1-\rho)t$. Therefore, the \emph{microscopic height fluctuation field} is given by 
\begin{equation}
\label{eq:height-field-intro}
\mc Z_t^n(\cdot) = \frac{1}{n^{3/2}} \sum_{x=1}^n \big(h_{tn^2}(x)-c_n t\big)\delta_{x/n}(\cdot),
\end{equation}
which means, formally, that $\mc Y_t^n = -\nabla \mc Z_t^n$.  In the same spirit as Theorem \ref{theo:convergence}, for $\varepsilon_n=E/\sqrt n$ and $\rho=\frac12$, we prove in Theorem \ref{theo:convheight} that the sequence of processes $\mc Z^n$ converges, in a suitable space of distributions, towards the unique energy solution $\mc Z$ to \eqref{eq:KPZintro} with Neumann boundary conditions, with the same parameters $A=1$, $B=E$ and $C=\frac12$.

A closely related convergence result was recently shown by Corwin and Shen~\cite{Corwin2016}, who consider the height function associated to a variant of the WASEP in contact with reservoirs, apply G\"artner's microscopic Cole-Hopf transform, and show that in the scaling limit the process converges to a solution of the SHE with Robin boundary conditions. However, the model we study here has parameters that are not admissible in~\cite{Corwin2016}, because in their formulation the parameters $A$ and $B$ have to be positive (see e.g. Proposition~2.7 and p.14 in~\cite{Corwin2016}), while we would get $A = B = -E^2/8$ (see Proposition~\ref{prop:link} below and also Remark~\ref{rmk:CS}). Apart from that, our methods are very different from the ones used in~\cite{Corwin2016} and we study the microscopic density fluctuation field without relying on the Cole-Hopf transform.

\subsubsection{Uniqueness of energy solutions and boundary behavior at the macroscopic scale}

The convergence proofs described above show relative compactness of the microscopic density fluctuations field (resp. the microscopic height fluctuation field) under rescaling, and that any limiting point is an energy solution to the stochastic Burgers equation with Dirichlet boundary conditions (resp. the KPZ equation with Neumann boundary conditions). To conclude the convergence, it remains to prove the uniqueness of energy solutions. We achieve this following the same strategy used in~\cite{Gubinelli2015Energy} for proving the uniqueness of energy solutions on the real line: We mollify an energy solution $\mc Y$ to the SBE, $\mc Y^n$, find a suitable anti-derivative $\mc Z^n$ of $\mc Y^n$, and let $\Phi^n = e^{\mc Z^n}$. Now we take the mollification away and show that $\Phi = \lim_n \Phi^n$ solves (a version of) the SHE. Since uniqueness holds for solutions to the SHE, $\mc Y = \nabla \log \Phi$ must be unique. But while the strategy is the same as in~\cite{Gubinelli2015Energy}, the technical details are considerably more involved because our space is no longer translation invariant and many of the tools of~\cite{Gubinelli2015Energy} break down. Moreover, the dynamics of $\Phi^n$ contain a singular term that converges to Dirac deltas at the boundaries, a new effect which can be interpreted as a change of boundary conditions: Formally we would expect that $\nabla \Phi_t(0) = e^{\mc Z_t(0)} \nabla \mc Z_t(0) = 0$ because $\nabla \mc Z_t(0) = 0$ by the Neumann boundary conditions for $\mc Z$. However, the singular term in the dynamics changes the boundary conditions to Robin's type, $\nabla \Phi_t(0) = -D \Phi_t(0)$, $\nabla \Phi_t(1) = D \Phi_t(1)$ for a constant $D \in \bb R$ which depends on the parameters $A,B,C$.

In that sense $\mc Z$ can be interpreted as a Cole-Hopf solution to the KPZ equation with inhomogeneous Neumann boundary conditions, and then the question arises which of the two formulations (inhomogeneous or homogeneous Neumann conditions) accurately describes the behavior of our stochastic process at the boundary. While of course the main difficulty with the KPZ equation is that its solutions are non-differentiable and in particular we cannot evaluate $\nabla \mc Z_t(0)$ pointwise, we show in Proposition~\ref{prop:boundary} that after averaging a bit in time there are canonical interpretations for $\nabla \mc Z_t(0)$ and $\nabla \mc Z_t(1)$, and both indeed vanish. We also show in Proposition~\ref{prop:CH-boundary} that the formal change of the boundary conditions from Neumann to Robin in the exponential transformation $\Phi = e^{\mc Z}$ is reflected in the ``pointwise''  boundary behavior of $\Phi$ (again after averaging in time).

This should be compared with the recent work of Gerencs\'er and Hairer~\cite{Gerencser2017} who show that the classical solution $\mc Z^\varepsilon$ to the KPZ equation,
\begin{equation}\label{eq:GH-KPZ-intro}
   d\mc Z^\varepsilon_t = \Delta \mc Z^\varepsilon_t dt + \big((\nabla\mc Z^\varepsilon_t)^2 - c^\varepsilon\big) dt +  d \mc W^\varepsilon_t \;,
\end{equation}
with Neumann boundary condition $\nabla \mc Z^\varepsilon_t(0) = \nabla \mc Z^\varepsilon_t(1) = 0$ and where $\mc W^\varepsilon$ is a mollification of $\mc W$ and $(c^\varepsilon)_{\varepsilon>0}$ is a sequence of diverging constants, may converge to different limits satisfying different boundary conditions as $\varepsilon \to 0$, depending on which mollifier was used for $\mc W^\varepsilon$. But if the noise is only mollified in space and white in time, then the limit is always the same and it agrees with the Cole-Hopf solution with Neumann boundary conditions. Since our results show that the Cole-Hopf solution exhibits ``non-physical'' boundary behavior, this suggests that in order to obtain the ``correct'' limit it is not only necessary to subtract a large diverging constant $c^\varepsilon$ in~\eqref{eq:GH-KPZ-intro}, but additionally one should perform a boundary renormalization. Indeed, under boundary conditions the solution $\mc Z^\varepsilon$ to~\eqref{eq:GH-KPZ-intro} is not spatially homogeneous, and therefore there is no reason to renormalize it by subtracting a constant $c^\varepsilon$.

\subsection{Outline of the paper:}
In Section \ref{s2} we give the precise definition of our microscopic dynamics and its invariant measures, we also introduce all the spaces of test functions where the microscopic fluctuation fields, namely the density and the height, will be defined and we give the proper definition of these fields. Section  \ref{s3} contains all the rigorous definitions of solutions to the SPDEs that we obtain, namely the  OU/SBE with Dirichlet boundary conditions, the KPZ equation with Neumann boundary conditions  and the SBE with linear Robin boundary conditions, and we explain how these equations are linked.  In Section \ref{s4} we prove the convergence of the microscopic fields to the  solutions of the respective SPDEs, namely  Theorems \ref{theo:convergence} and  \ref{theo:convheight}.  In Section \ref{s5} we prove  the  second order Boltzmann-Gibbs principle, which is the main technical result that we need at the microscopic level in order to recognize the macroscopic limit of the density fluctuation field as an energy solution to the SBE. Finally, in Section \ref{sec:energy-pr} we give the proof of the uniqueness of solutions to the aforementioned SPDEs.  The appendices contain some important aside  results that are needed along the paper, but to facilitate the reading flow we removed them from the main body of the text. In particular, Appendix \ref{app:micro-CH} sketches how one could prove that the microscopic Cole-Hopf transformation of the microscopic density fluctuation field converges to the SHE, and in particular we show that already at the microscopic level the Cole-Hopf transformation changes the boundary conditions from Neumann to Robin's type.

\section{Model and definitions}
\label{s2}

\subsection{The microscopic dynamics: WASEP in contact with stochastic reservoirs} 
For $n$ a positive integer, we define $\Lambda_n = \{1,\ldots, n-1\}$ and   $\Omega_n=\{0,1\}^{\Lambda_n}$ as the state space of a  Markov process  $\{\eta_t^n\,; \, t \geq 0\}$, whose dynamics is entirely encoded into its infinitesimal generator, denoted below by $\mc L_n$. More precisely, our process belongs to the family of well-known \emph{weakly asymmetric simple exclusion processes}. Here we consider that the strength of asymmetry is ruled by a parameter $\gamma \geq \frac12$, and  we put the system of particles in contact with two stochastic reservoirs, whose rates of injection/removal of particles from the bulk depend on a parameter $\rho\in{(0,1)}$, which is fixed. To keep notation simple we omit the dependence on $\rho$ in all the quantities that we define ahead and in order to  facilitate future computations, we write the generator as
\begin{equation*}
\mathcal L_n  =  \mc L_n^{\textrm{bulk}} + \mc L_n^{\mathrm{bnd}}\,,
\end{equation*}
where $\mc L_n^{\textrm{bulk}}$ and  $\mc L_n^{\mathrm{bnd}}$  given below encode respectively the dynamics on the bulk, and on the left/right boundaries. For any $E>0$ and $\gamma\geq \frac12$  we define $\mc L_n^{\textrm{bulk}}$ and  $\mc L_n^{\mathrm{bnd}}$  acting on  functions $f: \Omega_n \to \bb R$ as follows: first,
\begin{equation*}
\big(\mc L_n^{\textrm{bulk}} f\big)(\eta) = 
\sum_{x=1}^{n-2} \bigg\{\Big(1+\frac{E}{n^\gamma} \Big) \eta(x)\big(1-\eta(x+1)\big)+\eta(x+1)\big(1-\eta(x)\big)\bigg\} \big(\nabla_{x,x+1}f\big)(\eta),
\end{equation*}
where $\big(\nabla_{x,x+1}f\big)(\eta)=f(\sigma^{x,x+1}\eta)
-f(\eta)$ and $\sigma^{x,x+1} \eta$ is the configuration obtained from $\eta$ after exchanging $\eta(x)$ with $\eta(x+1)$, namely
\[
(\sigma^{x,x+1} \eta)(y)=
\left\{
\begin{array}{ll}
\eta(x)\,,& \textrm{ if } y=x+1,\\
\eta(x+1)\,,& \textrm{ if } y=x,\\
\eta(y)\,,& \textrm{ if } y\notin \{x,x+1\}.\\
\end{array}
\right.
\]
Second, the generator of the dynamics at the boundaries  is given by
\begin{equation*}
\begin{split}
\big(\mc L_n^{\mathrm{bnd}} f\big)(\eta) & =  
 \Big\{\Big(1+\frac{E}{n^\gamma} \Big) \rho(1-\eta(1))+\eta(1)\big(1-\rho\big)\Big\}\big(f(\sigma^{1}\eta)
-f(\eta)\big)\\
&\; +
 \Big\{\Big(1+\frac{E}{n^\gamma} \Big) \eta(n-1)(1-\rho)+\rho\big(1-\eta(n-1)\big)\Big\}\big(f(\sigma^{n-1}\eta)
-f(\eta)\big)\\
\end{split}
\end{equation*}
where  
\[
(\sigma^{x} \eta)(y)=
\left\{
\begin{array}{ll}
1-\eta(y)\,,& \textrm{ if } y=x\,,\\
\eta(y)\,,&  \textrm{ if } y\neq x.\\
\end{array}
\right.
\] 
From now on to simplify notation we denote the rates that appear above 
by  $r_{x,x+1}(\eta)$: precisely, at any $x\in\{0,\dots,n-1\}$, and assuming the convention $\eta(0)=\eta(n)=\rho$, they are given by 
\begin{equation}\label{eq:rate}
r_{x,x+1}(\eta)=\Big(1+\frac{E}{n^\gamma} \Big) \eta(x)\big(1-\eta(x+1)\big)+\eta(x+1)\big(1-\eta(x)\big).
\end{equation} 
We refer to Figure \ref{fig1} for an illustration of the dynamics.  

\begin{figure}[H]
\centering
\begin{tikzpicture}
\centerarc[thick,<-](0.5,0.3)(10:170:0.45);
\centerarc[thick,->](0.5,-0.3)(-10:-170:0.45);
\centerarc[thick,->](4.5,-0.3)(-10:-170:0.45);
\centerarc[thick,<-](4.5,0.3)(10:170:0.45);
\centerarc[thick,->](8.5,-0.3)(-10:-170:0.45);
\centerarc[thick,<-](8.5,0.3)(10:170:0.45);
\draw (1,0) -- (8,0);

\shade[ball color=black](1,0) circle (0.25);
\shade[ball color=black](3,0) circle (0.25);
\shade[ball color=black](5,0) circle (0.25);
\shade[ball color=black](6,0) circle (0.25);
\shade[ball color=black](8,0) circle (0.25);

\filldraw[fill=white, draw=black]
(2,0) circle (.25)
(4,0) circle (.25)
(7,0) circle (.25);

\draw (1.3,-0.05) node[anchor=north] {\small  $\bf 1$}
(2.3,-0.05) node[anchor=north] {\small $\bf 2 $}
(8.5,-0.05) node[anchor=north] {\small $\bf n\!-\!1$};
\draw (0.5,0.8) node[anchor=south]{$\rho (1+\frac{E}{n^\gamma})$};
\draw (0.5,-0.8) node[anchor=north]{$1-\rho$};
\draw (4.5,0.8) node[anchor=south]{$1+\frac{E}{n^\gamma}$};
\draw (8.5,-0.8) node[anchor=north]{$\rho$};
\draw (8.5,0.8) node[anchor=south]{$(1-\rho) (1+\frac{E}{n^\gamma})$};
\draw (4.5,-0.8) node[anchor=north]{$1$};
\end{tikzpicture}
\caption{Illustration of the  jump rates. The leftmost and rightmost rates are the entrance/exiting rates.}\label{fig1}
\end{figure}
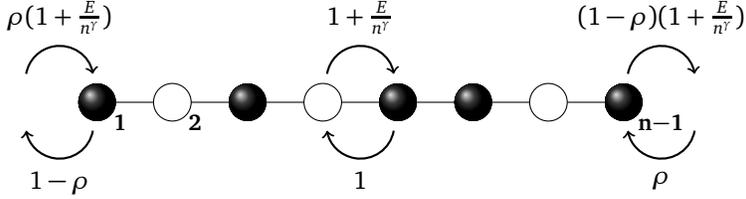

Throughout the paper, a time horizon line $T>0$ is fixed. We are interested in the evolution of this exclusion process in the diffusive time scale $tn^2$, thus we denote by $\{\eta^n_{tn^2}\,; \, t\in [0,T]\}$ the Markov process on $\Omega_n$ associated to the accelerated generator $n^2 \mathcal{L}_n$. The path space of c\`adl\`ag trajectories with values in $\Omega_n$ is denoted by $\mathcal{D}([0,T],\Omega_n)$. For any initial probability measure $\mu$ on $\Omega_n$, we denote by $\bb P_{\mu}$ the probability measure on  $\mathcal{D}([0,T],\Omega_n)$ induced by $\mu$ and the Markov process $\{\eta^n_{tn^2}\; ;\; t \in [0,T]\}$. The corresponding expectation is denoted by $\bb E_{\mu}$. 

\medskip

\textbf{Notations:} Throughout the paper, for any measurable space $(U,\nu$) we denote by $L^2(\nu)$ or $L^2(U)$ the usual $L^2$--space with norm $\|\cdot\|_{L^2(\nu)}$ and scalar product $\langle \cdot,\cdot\rangle_{L^2(\nu)}$. Whenever the integration variable $u$ may be not clear to the reader, we enlighten it by denoting $L_u^2(\nu)$ or $L^2_u(U)$. We also write $f\lesssim g$ or $g \gtrsim f$ if there exists a constant $C>0$, independent of the parameters involved in $f$ and $g$, such that $f \leq Cg$. If $f \lesssim g$ and $f \gtrsim g$, we also write $f \simeq g$. Finally we denote by $\bb N=\{1,2,\dots\}$ the set of positive integers and by $\bb N_0=\{0\}\cup\bb N$ the set of non-negative integers.

\subsection{Invariant measures and Dirichlet form}
For $\lambda \in (0,1)$, let $\nu^n_\lambda$ be the  \emph{Bernoulli product measure} on $\Omega_n$ with density $\lambda$: its marginals satisfy
$$
\nu^n_\lambda\big\{\eta: \eta(x)=1\big\} = 1-\nu^n_\lambda\big\{\eta:\eta(x)=0\big\} =\lambda, \quad  \textrm {for each } x\in \Lambda_n.
$$
When $\lambda=\rho$, the  measure $\nu^n_\rho$ is invariant but not reversible with respect to the evolution of $\{\eta_t^n\,; \,t \geq 0\}$.
To ensure the invariance, it is enough to check that
$
\int \mathcal L_n f(\eta)\, \nu^n_\rho(d\eta) =0\,,$
for any  function $f:\Omega_n\to \bb R$. This is a long albeit elementary computation, which is omitted here.

 In what follows we consider that the initial measure of the Markov process $\{\eta^n_{tn^2}\,; \, t\in [0,T]\}$ is the invariant measure $\nu^n_\rho$. For short, we denote the probability measure $\bb P_{\nu^n_\rho}$ by $\bb P_\rho$ and the corresponding expectation by $\bb E_\rho$.

\medskip

 The \emph{Dirichlet form} $\mf D_n$ is the  functional acting on functions $f:\Omega_n\to \bb R$ which belong to $L^2(\nu^n_\rho)$ as:
\begin{equation}\label{eq:dir_form}
\mf D_n(f) =\int f(\eta)\,(-\mc L_nf)(\eta)\, \nu^n_\rho(d\eta)\,.
\end{equation}
Invoking a general result \cite[Appendix 1, Prop. 10.1]{kl} we can rewrite $\mf D_n$ as
\begin{equation}\mf D_n  (f)=\mf D_n^{\textrm{bulk}} (f)+\mf D_n^{{\textrm{bnd}}}  (f),\label{eq:direxpl}\end{equation} where 
\begin{align*}
\mf D_n^{\textrm{bulk}} (f)&=\sum_{x=1}^{n-2}\int \Big(1+\frac{E}{n^\gamma} \Big){\big(\eta(x)-\eta(x+1)\big)^2}\big(\nabla_{x,x+1} f(\eta)\big)^2\nu^n_\rho(d\eta)\\
\mf D_n^{\textrm{bnd}} (f)&=
\int r_{0,1}(\eta)\big(f(\sigma^{1}\eta)
-f(\eta)\big)^2\nu^n_\rho(d\eta) \\ & \quad +
\int r_{n-1,n}(\eta)\big(f(\sigma^{n-1}\eta)
-f(\eta)\big)^2\nu^n_\rho(d\eta).
\end{align*}

\subsection{The spaces of test functions}
\label{ssec:test_functions}

Let $\mc C^\infty ([0, 1])$ be the space of real valued functions $\varphi : [0, 1] \to \bb R$ such that $\varphi$ is
continuous in $[0,1]$ as well as all its derivatives. We denote by ${\rm d}^{k}\varphi$ the $k$--th derivative of $\varphi$, and for $k=1$ (resp. $k=2$) we simply denote it by $\nabla\varphi$ (resp. $\Delta\varphi$).

Before defining the \emph{fluctuation fields} associated to our process, we first need to introduce the suitable space for test functions for each one of the fields that we will analyze. First of all, let 
$\mc S_{\tmop{Dir}}$ and $\mc S_{\tmop{Neu}}$  be the following spaces of  
functions:
\begin{align}
\mc S_{\tmop{Dir}} =&\Big\{ \varphi \in \mc C^\infty ([0, 1]) \; ; \;  {\rm d}^{2k}\varphi (0) ={\rm d}^{2k} \varphi
   (1) = 0, \text{ for any }k\in\mathbb{N}_0\Big\} \label{eq:space_test_funcDir}
\\
   \mc S_{\rm Neu} =& \Big\{\varphi \in \mc C^\infty([0,1]) \; ; \; {\rm d}^{2k+1} \varphi(0) ={\rm d}^{2k+1}  \varphi(1) = 0, \text{ for any } k\in\mathbb{N}_0 \Big\},\label{eq:space_test_funcNeu}
   \end{align}
both equipped with the family of seminorms $\big\{\sup_{u \in [0,1]} |{\rm d}^k \varphi(u)|\big\}_{k \in \bb N_0}$. Then $\SD$ and $ \mc S_{\rm Neu}$ are Fr\'echet spaces, and we write $\SD'$ and $ \mc S_{\rm Neu}'$ for their topological duals.

  Now, let $-\Delta$ be the closure of the Laplacian operator $-\Delta \colon {\mc S_{\tmop{Dir}}} \to L^2([0,1]) $ as an unbounded operator in the Hilbert space $L^2([0,1])$. It is a positive and self-adjoint operator whose eigenvalues and  eigenfunctions are given for any integer $m\geq 1$, respectively by  $\lambda_m=(m\pi)^2$ and $e_m(u)=\sqrt{2}\sin(m\pi u)$. We note that  the set
$\{e_m\; ; \;m\ge 1\}$ forms an orthonormal basis of $L^2([0,1])$. 
We denote by  $\{ T_t^{\rm Dir} \; ; \; t \ge 0\}$ the semi-group associated to $\Delta$.

For $k\in\mathbb{N}$, let us define 
${\mathcal{H}_{\tmop{Dir}}^k}=\mathrm{Dom}((-\Delta)^{k/2})$, endowed with the inner product
$$(\varphi,\psi)_k=\int_0^1(-\Delta)^{k/2}\;\varphi(u)\;(-\Delta)^{k/2}\;\psi(u)\, du.$$
By the spectral
theorem for self-adjoint operators, ${\mathcal{H}_{\tmop{Dir}}^k}$ equals
\begin{equation*}
{\mathcal{H}_{\tmop{Dir}}^k} = \Big\{\varphi\in L^2([0,1])\; ;\;
\|\varphi\|_k^2<\infty\Big\} ,
\end{equation*}
where $\|\varphi\|_k^2=(\varphi,\varphi)_k$
and for $\varphi,\psi\in  L^2([0,1])$ we also have 
\begin{equation*}
(\varphi,\psi)_k = \sum_{m=1}^{+\infty}\bigg\{(m\pi)^{2k}\langle\varphi,e_m\rangle_{L^2([0,1])}\langle\psi,e_m\rangle_{L^2([0,1])}\bigg\}.
\end{equation*}
Moreover, if ${\mathcal{H}_{\tmop{Dir}}^{-k}}$ denotes the topological dual space of ${\mathcal{H}_{\tmop{Dir}}^k}$, then
\[
{\mathcal{H}_{\tmop{Dir}}^{-k}}= \Big\{\mc Y \in \SD' \; ;\;\|\mc Y\|_{-k}^2<\infty\Big\},\] 
where  $\|\mc Y\|_{-k}^2=(\mc Y,\mc Y)_{-k}$
and the inner product $(\cdot,\cdot)_{-k}$ is defined as
\begin{equation}(\mc Y,\mc Z)_{-k} = \sum_{m=1}^{+\infty}\Big\{(m\pi)^{-2k} \mc Y(e_m)\mc Z(e_m)\Big\}, \label{eq:norm-k}
\end{equation}
with  $ \mc Y(e_m)$ denoting the action of the distribution $\mc Y$ on $e_m$. 

\medskip

Finally, we define $\Delta_{\rm Neu}$ (resp.~$\Delta_{\rm Dir}$) as the Neumann (resp.~Dirichlet) Laplacian acting on $\varphi\in\SN$ and $\mc Y\in\SN'$ (resp.~$\varphi \in \SD$ and $\mc Y \in \SD'$) as follows: 
\[
   \Delta_{\rm Neu} \mc Y(\varphi) = \mc Y(\Delta \varphi), \qquad \text{resp. } \quad \Delta_{\rm Dir} \mc Y(\varphi) = \mc Y(\Delta \varphi). 
\]
Let  ${\mathcal{H}_{\tmop{Neu}}^{-k}}$ be the topological dual space of $\mc H_{\rm Neu}^k$, both defined similarly to $\mc H_{\rm Dir}^k$
 and ${\mathcal{H}_{\tmop{Dir}}^{-k}}$ but replacing the basis $e_m(u)=\sqrt{2}\sin(m\pi u)$ by  ${\tilde{e}_m(u)}=\sqrt{2}\cos(m\pi u)$. 
In the next sections we will also need to consider one last operator, denoted by $\nabla_{\rm Dir}$ and defined as follows: given $k > 0$,  $\mc Y \in\SN'$  and $\varphi \in\SD$ we set
\begin{equation}\label{eq:nabladir}
   \nabla_{\rm Dir} \mc Y(\varphi) = - \mc Y(\nabla \varphi).
\end{equation}

\subsection{Fluctuation fields}\label{sec:fields}
Now we introduce all the microscopic fluctuation fields for which we will prove convergence to some infinite dimensional stochastic partial differential equations (SPDEs).   In the following, for any space $\mc S$ of distributions, we denote by $\mathcal{D}([0,T], \mc S)$ (resp.~$\mathcal C([0,T], \mc S$) the set of c\`adl\`ag (resp.~continuous) trajectories   taking values in $\mc S$ and  endowed
with the Skorohod topology.

\subsubsection{Density fluctuation field}\label{ssec:density} 

Since the particle system is stationary, the microscopic average $\bb E_\rho[\eta^n_{tn^2}(x)]$ is constantly equal to $\rho$. We are therefore looking at the fluctuations of the microscopic configurations around their mean, namely: 

\begin{definition} For any $t>0$, let $\mathcal{Y}^n_t$ be the \emph{density fluctuation
field} which acts on functions $\varphi\in \mc S_{\tmop{Dir}}$ as
\begin{equation*}
\mathcal Y^n_t (\varphi) = \frac 1{\sqrt n} \sum_{x=1}^{n-1} \varphi\big(\tfrac x n\big)
\big( \eta^n_{tn^2}(x) - \rho\big). 
\end{equation*}
\end{definition}
In the following, we will see that the right space one has to consider in order to prove the convergence of the field $\mc Y_t^n$ is the space $\mc H_{\rm Dir}^{-k}$, with $k>\frac52$. Note that, for any $t>0$ and any $k > \frac{5}{2}$, $\mc Y_t^n$ is indeed an element of $\mathcal{H}_{\tmop{Dir}}^{-k}$ and $\mc  Y^n_\cdot \in \mathcal{D}([0,T], \mathcal{H}_{\tmop{Dir}}^{-k})$.

\medskip

\subsubsection{Height function} \label{ssec:height}

Alternatively, instead of working with the density fluctuation field of the particle system, we may also consider the height function associated to it. Roughly speaking, the height function integrates the density fluctuation field in the space variable, i.e.~it describes a curve $h$ from $\{1, \dots, n\}$ to $\RR$ which satisfies $h(x+1) = h(x) + (\eta(x)-\rho)$. This suggests the definition $h(x) = \sum_{y=1}^{x-1} (\eta(y) - \rho)$, which however has the disadvantage that if a new particle enters at
site $1$, then $h(x)$ increases by $1$ for all $x$. Therefore, we  set
\begin{equation}\label{eq:htn}
   h^n_t (x) = h^n_t (1) + \sum_{y = 1}^{x - 1} (\eta^n_{t} (y) - \rho), \qquad \text{for any }x \in \{ 1, \ldots, n \},
\end{equation}
where, by definition, $h_0^n (1) = 0$ and  \begin{itemize}
\item  $h_t^n (1)$ increases by $1$ whenever a particle at site
$1$ leaves the system to the left, 
\item $h_t^n (1)$ decreases by $1$ whenever a
new particle enters the system at the site $1$.
\end{itemize} In other words $h_t^n (1)$ is
exactly the net number of particles removed from the system at the left
boundary until time $tn^2$. 

The weak asymmetry in the particle system causes the height function to slowly decrease because $E>0$  (it would grow if $E<0$), and at the first order the decrease is of order {$n^{-2} c_n t$} for
\begin{equation}
   c_n = - \rho (1 - \rho) n^{2 - \gamma} E. \label{eq:cn}
\end{equation}
For instance, with our stationary dynamics, one can easily see (see  \eqref{eq:height-boundary-pr1} for the case $x=1$ but note that for the other values of $x$  we also have a similar expression) that 
\[ 
\bb E_\rho\big[h_t^n(x)\big] = \bb E_\rho\big[h_0^n(x)\big] - \frac{E}{n^\gamma} \rho(1-\rho)t = {n^{-2} c_n} t, \qquad \text{for any } x\in\{1,\dots,n\}.
\]
In the case $\gamma = \frac12$ (where we will see KPZ behavior) note that $c_n = - n^{3/2} \rho(1-\rho)$. So while on the microscopic scale the decrease is negligible, it becomes important on time scales of order $n^2$. We want to investigate the fluctuations around the flat interface $c_n t$, namely: 

\begin{definition} For any $t >0$, let $\mc Z_t^n$ be the \emph{height fluctuation field} which acts on functions $\varphi \in \SN$ as
\[
   \mc Z^n_t(\varphi) = \frac{1}{n^{3/2}} \sum_{x = 1}^n \varphi\big(\tfrac{x}{n}\big) (h^n_{tn^2}(x) - c_n t).
\]
\end{definition} 

\begin{remark}
We will see, as expected, that $\mc Z_t^n$ and $\mc Y_t^n$ are related. Note that: if $\varphi \in\SD$, then $\nabla\varphi \in \SN$ and, a simple computation, based on a summation by parts, shows that $\mc Z_t^n(\tilde \nabla_n \varphi)$ can be written as $-\mc Y_t^n(\varphi)$ plus the  boundary terms: 
\[ 
-\tfrac{1}{\sqrt n}\varphi\big(\tfrac1n\big)\big(h_{tn^2}^n(1)-c_nt\big) +  \tfrac{1}{\sqrt n}\varphi\big(\tfrac n n\big)\big(h_{tn^2}^n(n)-c_nt\big).
\] 
Above $\tilde \nabla_n \varphi$ is essentially a discrete gradient, see \eqref{eq:relation} below.
Note that since $\varphi \in\SD$, the last expression vanishes in  $ L^2(\bb P_\rho)$ as $n\to\infty$, as a consequence of  Lemma \ref{lem:height-boundary} given ahead.
\end{remark}

Below, we will prove that the convergence of the field $\mc Z_t^n$ is in the space $\mc H_{\rm Neu}^{-k}$, with $k>\frac52$ and we  note that $\mc  Z^n_\cdot \in \mathcal{D}([0,T], \mathcal{H}_{\tmop{Neu}}^{-k})$.

\section{Solutions to non-linear SPDEs and statement of the results}
\label{s3}

In this section, we first define properly the notion of solutions for four stochastic partial differential equations, all with boundary conditions (Sections \ref{ssec:0U}, \ref{ssec:SBE}, \ref{ssec:kpz} and \ref{app:she}), and we connect them to each other, especially through their boundary conditions (Section \ref{ssec:boundary}). These SPDEs are going to describe the macroscopic behavior of the fluctuation fields of our system, for different values of the parameter $\gamma$ ruling the strength of the asymmetry: the precise statements are then given in Section \ref{ssec:results}. For the sake of clarity, the proofs will be postponed to further sections.

\subsection{Ornstein-Uhlenbeck process with Dirichlet boundary condition}
\label{ssec:0U}
Let us start with the \emph{generalized Ornstein-Uhlenbeck process} described by the formal SPDE
\begin{equation} \label{eq:OU}
d\mathcal Y_t = A\Delta_{\rm Dir} \;\mathcal Y_t \;dt + \sqrt{D}\; \nabla_{\rm Dir} \;d\mc W_t,
\end{equation}
with $A,D > 0$ and where $\{\mc W_t \; ; \;  t\in [0,T]\}$  is a standard  $\SN'$--valued Brownian motion, with covariance
\begin{equation}\label{eq:Bm-cov}
   \bb E[\mc W_s(\varphi) \mc W_t(\psi)] = (s \wedge t) \langle \varphi,\psi \rangle_{L^2([0,1])}.
\end{equation}
Since  $\{\mc W_t\; ; \; t\geq 0\}$ is  $\SN'$--valued, then $\nabla_{\rm Dir}\mc W_t$ is a well defined object in $\SD'$.

\medskip

The following proposition aims at defining in a unique way the stochastic process solution to \eqref{eq:OU}:

\begin{proposition}\label{prop:OU}
There exists a unique (in law) stochastic process $\{\mc Y_t \; ;\; t\in[0,T]\}$, with trajectories in $\mc C([0,T],\SD')$,  called \emph{Ornstein-Uhlenbeck process solution of} \eqref{eq:OU}, such that: 
\begin{enumerate}

\item  the process $\{\mc Y_t \; ; \; t \in [0,T]\}$  is \emph{``stationary''}, in the following sense: for any $t\in[0,T]$, the random variable $\mc Y_t$ is a $\SD'$--valued space white noise with covariance given on $\varphi,\psi \in \SD$ by 
\[ \bb E \big[ \mc Y_t(\varphi) \mc Y_t(\psi)\big] =\frac{D}{2A} \langle \varphi, \psi \rangle_{L^2([0,1])}\; ;  \]  
\item for any $\varphi \in \mc S_{\tmop{Dir}}$, 
\begin{align*}
\mc M_t(\varphi)&= \mc Y_t(\varphi)-\mc Y_0(\varphi) - A\int_0^t \mc Y_s(\Delta \varphi)\; ds
\end{align*}
is a continuous martingale with respect to the natural filtration associated to $\mc Y_\cdot$, namely
\begin{equation}\label{eq:naturalfiltration} \mc F_t=\sigma\big(\mc Y_s(\varphi)\; ; \; s \leq t, \; \varphi \in \mc S_{\tmop{Dir}} \big),\end{equation}
 of quadratic variation \[ \big\langle \mc M (\varphi) \big\rangle_t = t D \big\|\nabla \varphi\big\|^2_{L^2([0,1])}.\] 
\end{enumerate} 
Moreover, for every function $\varphi \in \mc S_{\tmop{Dir}}$, the stochastic process $\{\mc Y_t(\varphi)\; ; \; t \in [0,T]\}$ is Gaussian, and the distribution of $\mc Y_t(\varphi)$ conditionally to $\{\mc F_u \; ; \; u < s\}$, is normal of mean $\mc Y_s(T_{{A(t-s)}}^{\rm Dir}\; \varphi)$ and variance \[ {D} \int_0^{t-s}\big\|\nabla (T_{{A}r}^{\rm Dir} \; \varphi) \big\|^2_{L^2([0,1])} \; dr.\]
\end{proposition}

\begin{proof}[Proof of Proposition \ref{prop:OU}]
The proof of this fact is standard, and we refer the interested reader to \cite{kl,LMO}. 
\end{proof}

\subsection{Stochastic Burgers equation with Dirichlet boundary condition} \label{ssec:SBE}
 Now we define the notion of \emph{stationary energy solutions} for the \emph{stochastic Burgers equation with Dirichlet boundary condition}, written as
\begin{equation}\label{eq:SBE}
d\mathcal Y_t = A \; \Delta_{\rm Dir}\; \mathcal Y_t\; dt + \bar E\;\nabla_{\rm Dir}\;  \big(\mathcal{Y}_t^2 \big) \;dt + \sqrt{D}\;  \nabla_{\rm Dir}\; d\mc W_t,
\end{equation}
with $A,D > 0$ and $\bar E \in \bb R$, and where $\{\mc W_t\; ; \; t\in[0,T]\}$ is a   $\SN'$--valued Brownian motion with covariance~\eqref{eq:Bm-cov}.
Note that $\nabla_{\rm Dir}\;  \big(\mathcal{Y}_t^2 \big)$ so far has not a precise meaning but this will be rigorously given below in Theorem \ref{def:energy}. More precisely, we are going to adapt to our case (i.e.~adding boundary conditions) the  notion of energy solutions as given for the first time in \cite{gj2014, Gubinelli2013}. 

\begin{definition} \label{def:approx}
For $u\in[0,1]$ and $\varepsilon>0$, let $\iota_\varepsilon(u): [0,1]\to \bb R$ be the approximation of the identity defined as 
\[
\iota_\varepsilon(u)(v):= \begin{cases} \varepsilon^{-1}  \; \mathbf{1}_{]u,u+\varepsilon]}(v)  & \text{ if } u \in [0, 1 - 2
       \varepsilon), \\
\varepsilon^{-1}\; \mathbf{1}_{[u-\varepsilon,u[}(v) & \text{ if } u \in [1 - 2
       \varepsilon, 1]. \end{cases}
\] 
\end{definition}

The following theorem, which we prove in this paper, aims at defining uniquely the stochastic process solution to \eqref{eq:SBE}: 
\begin{theorem} \label{def:energy}
There exists a unique (in law) stochastic process $\{\mc Y_t \; ; \; t \in [0,T]\}$ with trajectories in $\mc C([0,T],\SD')$,  called \emph{stationary energy solution} of \eqref{eq:SBE}, such that: 
\begin{enumerate}
\item the process $\{\mc Y_t \; ; \; t \in [0,T]\}$  is \emph{``stationary''}, in the following sense: for any $t\in[0,T]$, the random variable $\mc Y_t$ is a $\SD'$--valued space white noise with covariance given on $\varphi,\psi \in \SD$ by 
\[ \bb E \big[ \mc Y_t(\varphi) \mc Y_t(\psi)\big] =\frac{D}{2A} \langle \varphi, \psi \rangle_{L^2([0,1])}\; ;  \] 

\item  the process $\{\mc Y_t \; ; \; t \in [0,T]\}$  satisfies the following \emph{energy estimate}: there exists $\kappa >0$ such that 
for any $\varphi \in \mc S_{\tmop{Dir}}$, any $\delta,\varepsilon \in (0,1)$ with $\delta < \varepsilon$, and any $s,t \in [0,T]$ with $s<t$, 
\begin{equation}
\label{eq:lipschitz}
\bb E\Big[ \big(\mc A_{s,t}^\varepsilon (\varphi) - \mc A_{s,t}^\delta(\varphi)\big)^2\Big] \leq \kappa (t-s) \varepsilon \; \big\| \nabla \varphi\big\|^2_{L^2([0,1])},
\end{equation}
where 
\begin{equation}\mc A_{s,t}^\varepsilon(\varphi)= - \int_s^t \int_0^1  \Big(\mc Y_r\big(\iota_\varepsilon(u)\big)\Big)^2 \; \nabla \varphi(u)\; du dr; \label{eq:A}\end{equation}

\item \label{item}for any $\varphi \in \mc S_{\tmop{Dir}}$ and any $t\in[0,T]$, the process
\[
   \mc M_t (\varphi) = \mc Y_t(\varphi)-\mc Y_0(\varphi) - A \int_0^t \mc Y_s(\Delta \varphi)\; ds - \bar E \mc A_t(\varphi)
\]
is a continuous martingale with respect to the natural filtration \eqref{eq:naturalfiltration} associated to $\mc Y_\cdot$, of quadratic variation \[ \big\langle \mc M (\varphi) \big\rangle_t = t D \big\|\nabla \varphi\big\|^2_{L^2([0,1])}\; ,\] where the process $\{\mc A_t\; ; \; t\in[0,T]\}$ is obtained through the following limit, which holds in the $L^2$-sense:
\begin{equation} \mc A_t(\varphi)= \lim_{\varepsilon \to 0} \mc A_{0,t}^\varepsilon(\varphi)\; ;\label{eq:B}\end{equation}

\item the reversed process $\{\mc Y_{T-t}\; ; \; t \in [0,T]\}$ satisfies item \eqref{item} with $\bar E$ replaced by $-\bar E$.
\end{enumerate}
\end{theorem}
\begin{proof}[Proof of Theorem \ref{def:energy}] The proof is somewhat lengthy and we give it in Section~\ref{sec:energy-pr}.
\end{proof}

\begin{remark}
Note that there is a way to make sense of $\mc Y_s(\iota_\varepsilon(u))$ (which is \emph{a priori} not obvious since $\iota_\varepsilon(u)$ is not a test function), as explained in \cite[Remark 4]{gj2014}.
\end{remark}

\begin{remark}
Note that the definition of the Ornstein-Uhlenbeck process $\mc Y$ given in Proposition \ref{prop:OU} and of the energy solution to the stochastic Burgers equation \eqref{eq:SBE} when $\bar E=0$, are equivalent. Indeed, the only part which is not obvious is that the Ornstein-Uhlenbeck process $\mc Y$ satisfies conditions~(2) and~(4) in Theorem~\ref{def:energy}. But both of these will follow from our convergence result Theorem~\ref{theo:convergence} with $\bar E=0$.
\end{remark}

\subsection{KPZ equation with Neumann boundary condition}
\label{ssec:kpz}
Here we define the notion of solution for  \emph{the KPZ equation with Neumann boundary condition}, which is formally given by 
\begin{equation}\label{eq:KPZ}
d \mc Z_t = A \;\Delta_{\rm Neu}\; \mc Z_t dt + \bar{E}\;  (\nabla \mc Z_t )^{\diamond 2} \;dt + \sqrt{D} \; d\mc W_t,
\end{equation}
with $A,D > 0$ and $\bar{E} \in \bb R$, and where $\{\mc W_t\; ; \; t\in[0,T]\}$ is a standard $\SN'$--valued Brownian motion with covariance~\eqref{eq:Bm-cov} and $(\nabla \mc Z_t )^{\diamond 2}$ denotes a \emph{renormalized square} which will have a precise meaning in Theorem \ref{def:energy-KPZ} below.
\begin{definition}

For $\varepsilon>0$  and $\mc Z\in\mc C([0,1])$, let us define
 \begin{equation}\label{eq:nablah} \nabla_{\varepsilon} \mc Z (u) =  \begin{cases}
       \varepsilon^{-1} (\mc Z (u + \varepsilon) - \mc Z (u)), & \text{ if } u \in [0, 1 - 2
       \varepsilon), \vphantom{\Big(}\\
       \varepsilon^{-1}  (\mc Z(u) - \mc Z(u - \varepsilon)), & \text{ if } u \in [1 - 2
       \varepsilon, 1]\vphantom{\Big(} .
     \end{cases}  \end{equation}  
\end{definition}

The following theorem aims at defining uniquely the stochastic process solution to \eqref{eq:KPZ}:
\begin{theorem} \label{def:energy-KPZ}
Let $Z$ be a random variable with values in $\mc C([0,1])$, such that $\nabla_{\rm Dir} Z$ is a white noise with variance $D/(2A)$ and such that $\sup_{u \in [0,1]} \bb E[e^{2Z(u)}] < \infty$.

\medskip

 There exists a unique (in law) stochastic process $\{ \mc Z_t \; ; \; t \in [0,T]\}$ with trajectories in $\mc C([0,T],\SN')$,
called \emph{almost stationary energy solution} of \eqref{eq:KPZ}, such that:
\begin{enumerate}
\item $\tmop{law}(\mc Z_0) = \tmop{law} (Z)$;

\item there exists a stationary energy solution $\{\mc Y_t \; ; \; t \in [0,T]\}$ to the stochastic Burgers equation~\eqref{eq:SBE} -- with trajectories in $\mc C([0,T],\SD')$ -- such that $\nabla_{\rm{Dir}}\mc Z = \mc Y$;

\item for any $\varphi \in \SN$ and any $t\in[0,T]$, the process
\[
   \mc N_t(\varphi) = \mc Z_t(\varphi)- \mc Z_0(\varphi) - A\int_0^t  \mc Z_s(\Delta \varphi)\; ds - \bar{E} \mc B_t(\varphi)
\]
is a continuous martingale with respect to the natural filtration associated to $\mc Z_\cdot$, defined as in~\eqref{eq:naturalfiltration}  but with $\varphi\in \SN$ and with $\mc Z_s$ in the place of $\mc Y_s$, of quadratic variation \[\big\langle \mc N(\varphi) \big\rangle_t = t D \big\| \varphi \big\|^2_{L^2([0,1])}\; ,\] where the process $\{\mc B_t\; ; \; t\in[0,T]\}$ is obtained through the following limit, which holds in the $ L^2$-sense:
\begin{equation}\label{eq:B-def}
   \mathcal{B}_t (\varphi) = \lim_{\varepsilon \to 0} \int_0^t \int_0^1 \left\{ (\nabla_{\varepsilon} \mc Z_s (u))^2 - \frac{1}{4\varepsilon} \right\} \varphi (u) d u d s 
\end{equation}
  where $ \nabla_{\varepsilon} \mc Z_s (u)$ has been defined in \eqref{eq:nablah}. 
\end{enumerate}
\end{theorem}
\begin{proof}[Proof of Theorem \ref{def:energy-KPZ}]
The proof of that result is also  contained in Section \ref{sec:energy-pr}.
\end{proof}

\begin{remark}
Note that we did not require $\mc Z$ to be a continuous function in $u$, so it is not obvious that $\nabla_\varepsilon \mc Z(u)$ as defined in~\eqref{eq:nablah} is well defined. But of course $\nabla_\varepsilon \mc Z(u) = \mc Y(\iota_\varepsilon(u))$, and as discussed before the right hand side can be made sense of with the arguments of~\cite[Remark 4]{gj2014}.
\end{remark}

\subsection{Stochastic heat equation with Robin boundary conditions}\label{app:she}

Finally we define the notion of  solutions to the  \emph{stochastic heat equation with Robin boundary condition}, written as
\begin{equation}\label{eq:she-rob}
    d \Phi_t = A\; \Delta_{\mathrm{Rob}} \Phi_t dt + \sqrt{D}\; \Phi_t d \mc W_t, 
    \end{equation} 
    with $A, D>0$ and $\{\mc W_t\; ; \; t\in[0,T]\}$ is a standard $\SN'$--valued Brownian motion with covariance~\eqref{eq:Bm-cov}, and where we want to impose (formally): 
    \begin{equation}
    \label{eq:BC-rob}
 \nabla \Phi_t(0) = \alpha \Phi_t(0), \qquad \nabla \Phi_t(1) = \beta \Phi_t(1),
\end{equation}
with $\alpha, \beta \in \RR$. To see how we should implement these boundary conditions, consider $f \in \mc C^2([0,1])$ and $\varphi \in \SN$, and assume that $f$ satisfies the Robin boundary conditions $\nabla f(0) = \alpha f(0)$, and $\nabla f(1) = \beta f(1)$. Integrating by parts twice and using the Neumann boundary conditions for $\varphi$ and the Robin boundary conditions for $f$, we obtain
\[
   \int_0^1 \Delta f(u) \varphi(u) du = \int_0^1 f(u) \Delta \varphi(u) du - \alpha f(0) \varphi(0) + \beta f(1) \varphi(1).
\]
This suggests to define the operator
\[
   \Delta_{\rm{Rob}} \colon \SN'\cap\mc C([0,1]) \to \SN',\qquad \Delta_{\rm{Rob}} f(\varphi) =  f(\Delta \varphi) - \alpha f(0) \varphi(0) + \beta f(1) \varphi(1),
\]
and in principle this leads to a definition of solutions to \eqref{eq:she-rob}. But for technical reasons we do not want to a priori assume our solution to be continuous in $(t,u)$, and this means we could change the values of $\Phi_t(0)$ and $\Phi_t(1)$ without changing $\Phi_t(\varphi)$, so we cannot hope to get uniqueness without further assumptions. Let us introduce a suitable class of processes for which the boundary term is well defined: we write $\Phi \in \mc L^2_C([0,T])$ if $\Phi$ takes values in the Borel measurable functions on $[0,T] \times [0,1]$, if  \[\sup_{(u,t) \in [0,1] \times [0,T]} \bb E\big[\Phi_t(u)^2\big] < \infty,\] and if for all $t > 0$ the process $u \mapsto \Phi_t(u)$ is continuous in $L^2(\bb P)$, where $\bb P$ is the law of the process $\Phi_\cdot$ and above $\bb E$ is the expectation w.r.t.~to $\bb P$. For $\Phi \in \mc L^2_C([0,T])$ we cannot change the value of $\Phi_t(0)$ without changing $\Phi_t$ in an environment of $0$, and this would also change $\Phi_t(\varphi)$ for some test functions $\varphi$.

\begin{proposition}\label{def:she-rob} 
Let $F$ be a random variable with values in the Borel measurable functions on $[0,1]$, such that $\sup_{u \in [0,1]} \bb E[F(u)^2] < \infty$.
\medskip

Then there exists at most one (law of a) process $\{\Phi_t\; ; \; t \in [0,T]\}$ with trajectories in $\mc L^2_C([0,T])$, called \emph{weak solution} to~\eqref{eq:she-rob} with boundary condition \eqref{eq:BC-rob}, such that:

\begin{enumerate}
\item $\tmop{law}(\Phi_0) = \tmop{law}(F)$; 
 \item for any $\varphi \in \mc \SN$ and any $t\in[0,T]$, the process
\begin{equation}\label{eq:she-rob-weak}
   \mathfrak M_t (\varphi) =  \Phi_t(\varphi)- \Phi_0(\varphi) - A \int_0^t  \Phi_s(\Delta \varphi)\; ds - A \int_0^t (- \alpha \Phi_s(0) \varphi(0)+ \beta \Phi_s(1) \varphi(1)) ds 
\end{equation}
is a continuous martingale of quadratic variation \[ \big\langle \mathfrak M (\varphi) \big\rangle_t = D \int_0^t \int_0^1 |\Phi_s(u) \varphi(u)|^2 du ds\; .\]
\end{enumerate}
\end{proposition}

\begin{proof}
The proof of this proposition is given in Appendix \ref{app:she-proof}.
\end{proof}

\begin{remark}\label{rmk:CS}
	In~\cite[Proposition~2.7]{Corwin2016} the authors show uniqueness for (the mild solution of) a similar equation, but they require $\alpha >0$ and $\beta < 0$, while here we will need to take $\alpha = - E^2/2$ and $\beta = E^2/2$ (where $E>0$ is the parameter ruling the weak asymmetry). As the nice probabilistic representation given in \cite{Freidlin1985, Papanicolaou1990} shows, the key difference is that our boundary conditions correspond to sources at the boundary, while the case of~\cite{Corwin2016} corresponds to sinks. Such sink terms make the Robin Laplacian negative and in particular its spectrum is negative, while for the Robin Laplacian with a source term some eigenvalues may be positive. In particular, the method used in Section~4.1 of ~\cite{Corwin2016} for proving heat kernel estimates breaks down for our choice of parameters. Theorem~3.4 of~\cite{Papanicolaou1990} gives $L^1$ and $L^\infty$ bounds for the heat kernel that are sufficient for our purposes, but it is not quite obvious whether the heat kernel is in $\mc C^{1,2}((0,\infty) \times [0,1])$ and satisfies the Robin boundary condition in the strong sense, which is what we would need here. While we expect this to be true, we avoid the problem by formulating the equation slightly differently: we do not work with the Robin heat kernel but we use the Neumann heat kernel instead and deal with the resulting boundary corrections by hand.
\end{remark}

\begin{remark}
	It would be possible to weaken the assumptions on the initial conditions (say to allow for a Dirac delta initial condition), and of course we would also be able to prove existence and not just uniqueness. Since given the computations in Appendix~\ref{app:she-proof} these results follow from standard arguments as in~\cite{Quastel2011,Walsh1986} and we do not need them here, we choose not to include the proofs.
\end{remark}

We conclude this section by making the link between the KPZ equation \eqref{eq:KPZ} and the stochastic heat equation \eqref{eq:she-rob}, with their respective boundary conditions. This link is, as expected, done through the Cole-Hopf transformation, although the boundary conditions are linked in a more complicated way then one might guess naively:

\begin{proposition}\label{prop:link}
Let $\mc Z$ be the almost stationary energy solution of the KPZ equation \eqref{eq:KPZ} with Neumann boundary condition as defined in Theorem \ref{def:energy-KPZ}.  Then, for any $t\in [0,T]$ we have \[\mc Z_t = \frac{A}{\bar{E}} \log \Phi_t + \frac{D^2(\bar{E})^3}{48 A^4}t,\] where $\Phi$ solves the stochastic heat equation 
\[  d \Phi_t = A \Delta_{\mathrm{Rob}} \Phi_t dt + \frac{\sqrt{D}\; \bar{E}}{A} \Phi_t d \mc N_t  \]
with the Robin boundary condition
\[
   \nabla \Phi_t(0) = - \frac{D(\bar{E})^2}{4A^3} \Phi_t(0), \qquad \nabla \Phi_t(1) = \frac{D(\bar{E})^2}{4A^3} \Phi_t(1),
\] 
and where above $\mc N_t$ has been defined in item (2) of Theorem \ref{def:energy-KPZ}.
\end{proposition}
\begin{proof}[Proof of Proposition \ref{prop:link}]
This proof is contained in Section \ref{sec:energy-pr}. 
\end{proof}

\subsection{Boundary behavior for singular SPDEs} \label{ssec:boundary}
The stochastic Burgers equation is an important example of a \emph{singular SPDE}, a class of equations for which tremendous progress was made in the last years~\cite{Hairer2014, Gubinelli2015Paracontrolled}. The vast majority of all papers on singular SPDEs treats domains without boundaries (mostly the torus, sometimes Euclidean space), and only quite recently some articles appeared that deal with boundaries~\cite{Corwin2016, Gerencser2017}. 

But in some cases it is not quite obvious how to formulate the boundary conditions. For example, the solution to the stochastic Burgers equation is almost surely distribution valued and therefore we cannot evaluate it at the boundary. In Theorem~\ref{def:energy} we proposed a formulation for the notion of solutions to the stochastic Burgers equation with Dirichlet boundary condition $\mc Y_t(0) \equiv \mc Y_t(1) \equiv 0$ (for any $t \in [0,T]$), that seems natural to us. But as we saw in Proposition~\ref{prop:link}, we  have $\mc Y = \nabla_{\rm Dir} (\frac{A}{\bar{E}} \log \Phi)$, where $\Phi$ solves the stochastic heat equation with Robin boundary condition $\nabla \Phi_t(0) \equiv - \tfrac{D(\bar{E})^2}{4A^3} \Phi_t(0)$ and $\nabla \Phi_t(1) \equiv \tfrac{D(\bar{E})^2}{4A^3} \Phi_t(1)$. This means that according to the definition of Corwin and Shen~\cite{Corwin2016}, $\mc Y$ would solve the stochastic Burgers equation with Dirichlet boundary condition $\mc Y_t(0) \equiv - \tfrac{D(\bar{E})^2}{4A^3}$ and $\mc Y_t(1) \equiv \tfrac{D(\bar{E})^2}{4A^3}$.

Moreover, Gerencs\'er and Hairer observe in~\cite[Theorem~1.6]{Gerencser2017} that the classical solution $\mc Y^\varepsilon$ to the stochastic  Burgers equation without renormalization,
\begin{equation}\label{eq:gh-approx}
   d\mc Y^\varepsilon_t = \Delta_{\rm Dir} \mc Y^\varepsilon_t dt + \nabla_{\rm Dir} \big((\mc Y^\varepsilon_t)^2\big) dt +  \nabla_{\rm Dir} d \mc W^\varepsilon_t \;,
\end{equation}
with Dirichlet boundary condition $\mc Y^\varepsilon_t(0) \equiv \mc Y^\varepsilon_t(1) \equiv 0$ and where $\mc W^\varepsilon$ is a mollification of $\mc W$, may converge to different limits satisfying different boundary conditions as $\varepsilon \to 0$, depending on which mollifier was used for $\mc W^\varepsilon$. But if the noise is only mollified in space and white in time, then the limit is always the same and it agrees with the Cole-Hopf solution of~\cite{Corwin2016}. 

So it is not obvious whether there is a ``canonical'' way of formulating singular SPDEs with boundary conditions, and in the case of the stochastic Burgers equation with Dirichlet boundary conditions it may seem that the Cole-Hopf solution is the most canonical solution. But below, in Proposition \ref{prop:boundary}, we will see that our solution $\mc Y$  indeed satisfies $\mc Y_t(0) \equiv \mc Y_t(1) \equiv 0$, as long as we do not try to evaluate $\mc Y_t(0)$ at a fixed time but allow for a bit of averaging in time instead. We also show, in Lemma \ref{prop:CH-boundary},  that the (approximate) Cole-Hopf transformation $\Psi$ of $\mc Y$ satisfies the Robin boundary condition $\nabla \Psi_t(0) \equiv - \tfrac{D(\bar{E})^2}{4A^3} \Psi_t(0)$ and $\nabla \Psi_t(1) \equiv \tfrac{D(\bar{E})^2}{4A^3} \Psi_t(1)$ after averaging in time. This sheds some light on the actual boundary behavior of solutions to singular SPDEs and indicates that our formulation of the equation is maybe more natural than the Cole-Hopf formulation of~\cite{Corwin2016}. Although the approximation result in~\eqref{eq:gh-approx}  looks natural at first sight, note that \emph{not} renormalizing $ \nabla_{\rm Dir} \big((\mc Y^\varepsilon_t)^2\big)$ means that we renormalize $(\mc Y^\varepsilon_t)^2$ with a constant which is killed by $\nabla_{\rm Dir}$ and therefore does not appear in the equation. But of course $\mc Y^\varepsilon$ is not spatially homogeneous, so there is no reason why the renormalization should be spatially constant. And as our results show, taking it constant actually leads to an unnatural boundary behavior in the limit.

\begin{proposition}\label{prop:boundary}
 Let $\mc Y$ be the solution to the stochastic Burgers equation \eqref{eq:SBE} as defined in Theorem \ref{def:energy}, and let $\{\rho_{\varepsilon}\}_{\varepsilon >0} \subset L^2 ([0, 1])$ be a sequence of positive functions
  that converges weakly to the Dirac
  delta at $0$ (respectively $1$), in the sense that, for any function  $f \in \mc C ([0, 1])$, $\lim_{\varepsilon
  \rightarrow 0} \int_0^1\rho_{\varepsilon}(u) f(u)du = f (0)$ (respectively
  $= f (1)$). 
  \medskip

  Then we have for all $p \in [1, \infty)$
  \[ \lim_{\varepsilon \rightarrow 0} \mathbb{E} \bigg[ \sup_{t \in [0,T]}
     \bigg| \int_0^t \mc Y_s (\rho_{\varepsilon}) d s \bigg|^p \bigg] = 0.
  \]
\end{proposition}

\begin{proof}[Proof of Proposition \ref{prop:boundary}]
We prove this proposition in Appendix \ref{sec:proof-boundary}. 
\end{proof}

Consider now, for any $u,v\in[0,1]$
\[
   \Theta_u(v) =\mathbf{1}_{[0, u]} (v) + v-1.
\]
This is an integral kernel which will be used in Section \ref{ssec:map}  to map the energy solution $\mc Y$ to an approximate solution of the stochastic heat equation\footnote{In the notation of Section~\ref{ssec:map} we have $\Theta^\varepsilon_u(v) = \langle \Theta_u, p_\varepsilon^{\rm Dir}(v,\cdot)\rangle$.}. We set $\mc Z_t(u) = \mc Y_t(\Theta_u)$, so that $\nabla_{\rm Dir} \mc Z = \mc Y$, and then $\Psi_t(u) = \exp( \tfrac{\bar{E}}{A} \mc Z_t(u))$. Then one could compute formally
\[
   \nabla \Psi_t(0) = \Psi_t(0) \tfrac{\bar{E}}{A} \nabla \mc Z_t(0) = \Psi_t(0) \tfrac{\bar{E}}{A} \mc Y_t(0) = 0.
\]
But we will show now that this formal computation is incorrect and the boundary condition for $\Psi$ is not of Neumann type (i.e.~$\nabla \Psi_t(0) \equiv \nabla \Psi_t(1) \equiv 0$), but rather of Robin type, more precisely $\nabla \Psi_t(0) \equiv - \tfrac{D(\bar{E})^2}{4A^3} \Psi_t(0)$ and $\nabla \Psi_t(1) \equiv \tfrac{D(\bar{E})^2}{4A^3} \Psi_t(1)$.

\begin{proposition}\label{prop:CH-boundary}
Let $\mc Y$ be the solution to the stochastic Burgers equation \eqref{eq:SBE} as defined in Theorem \ref{def:energy} and  let $\Psi_t(u)=\exp(\frac E A \mc Y_t(\Theta_u))$.
\medskip

Then, we have for all $t \in [0,T]$
  \[ \lim_{\varepsilon \rightarrow 0} \mathbb{E} \bigg[ 
     \bigg| \int_0^t \bigg( \frac{\Psi_s (\varepsilon) - \Psi_s(0)}{\varepsilon}  + \frac{D(\bar{E})^2}{4A^3} \Psi_s(0) \bigg)d s \bigg|^2 \bigg] = 0
  \]
  and
 \[ \lim_{\varepsilon \rightarrow 0} \mathbb{E} \bigg[ 
     \bigg| \int_0^t \bigg( \frac{\Psi_s (1) - \Psi_s(1-\varepsilon)}{\varepsilon}  - \frac{D(\bar{E})^2}{4A^3} \Psi_s(1) \bigg)d s \bigg|^2 \bigg] = 0.
  \]
\end{proposition}

\begin{proof}[Proof of Proposition \ref{prop:CH-boundary}]
The proof will be exposed in Appendix \ref{sec:proof-boundary}.
\end{proof}

\subsection{Statement of the convergence theorems} \label{ssec:results}
From now and for the rest of the paper we assume $\rho=\frac12$. We are now ready to state all the convergence results for the different fields that we defined previously in Section \ref{sec:fields}:

\begin{theorem}[Convergence of the density field]
\label{theo:convergence}
Fix $T>0$, $k>\frac52$ and $\rho=\frac12$. 

Then, the sequence of processes $\{\mathcal{Y}_t^n \; ; \; t\in[0,T]\}_{n\in\bb N}$ converges in distribution in $\mc D([0,T],\mc H_{\mathrm{Dir}}^{-k})$ as $n\to \infty$ towards:
\begin{enumerate}
\item if $\gamma>\frac12$,  the Ornstein-Uhlenbeck process solution of
\eqref{eq:OU} with Dirichlet boundary conditions (as defined in Proposition \ref{prop:OU}) with $A=1$ and $D=\frac12$;

\medskip
\item if $\gamma=\frac12$, the unique stationary energy solution of the stochastic  Burgers equation \eqref{eq:SBE} with Dirichlet boundary conditions (as defined in Theorem \ref{def:energy}) with $A=1$, $\bar E=E$ and $D=\frac12$.
\end{enumerate} 
\end{theorem}

\begin{remark} We note that the previous theorem, when the strength of the asymmetry is taken with $\gamma =1$, has been already proved in \cite{GLM} by using a different strategy, namely, considering the microscopic Cole-Hopf transformation -- which avoids the derivation of a Boltzmann-Gibbs principle -- but in a wide scenario since there the initial measure is quite general.  
\end{remark}

\begin{theorem}[Convergence of the height field] \label{theo:convheight}
Fix $T>0$, $k>\frac52$ and $\rho=\frac12$. 

Then, the sequence of processes $\{\mc Z^n_t\; ;\; t \in [0,T]\}_{n\in{\bb N}}$ converges in distribution in $\mc D([0,T], \mc H_{\mathrm{Neu}}^{-k})$ as $n\to\infty$ towards: 
\begin{enumerate}

 \item if $\gamma > \frac12$, the unique almost stationary energy solution of the KPZ equation~\eqref{eq:KPZ} with Neumann boundary conditions  (as defined in Theorem \ref{def:energy-KPZ}) with $A=1$, $\bar E=0$ and $D=\frac12$;
 \medskip
 \item if $\gamma=\frac12$, the unique almost stationary energy solution of the KPZ equation~\eqref{eq:KPZ} with Neumann boundary conditions (as defined in Theorem \ref{def:energy-KPZ}) with $A=1$, $\bar E = E$ and $D=\frac12$.
\end{enumerate}
\end{theorem}

The strategy of the proof of Theorem \ref{theo:convergence} is quite well known and has been widely used in the past literature. Let us recall here the main steps:
\begin{enumerate}
\item first, prove that the sequence of probability measures $\{\mc Q_n\}_{n\in\bb N}$, where $\mc Q_n$ is induced by the density fluctuation field $\mc Y^n$ and the initial measure $\nu^n_\rho$, is tight. Note that $\mc Q_n$ is a measure on the Skorokhod space $\mc D([0,T],\mc H_{\mathrm{Dir}}^{-k})$. This is the purpose of Section \ref{sec:tightness} below;
\item second, write down the approximate martingale problem satisfied by $\mc Y_t^n$ in the large $n$ limit, and prove that it coincides with the martingale characterization of the solutions of the SPDEs given in Theorem \ref{theo:convergence}. The closure of the martingale problem is explained in  Section \ref{sec:martingale}. In the case $\gamma=\frac12$, we need to prove an additional important tool, the so-called \emph{second order Boltzmann-Gibbs principle}, which is stated in Section \ref{sec:BG} below and proved in Section \ref{sec:proofBG}. 
\end{enumerate}

The strategy of the proof of Theorem \ref{theo:convheight} is completely similar to the one described above.

\begin{remark}\label{rem:rho}
	Most of our arguments work for any $\rho \in (0,1)$. However, the restriction $\rho=\tfrac12$ is not just for convenience of notation: otherwise we would pick up an additional diverging transport term in the martingale decomposition for $\mc Y^n(\varphi)$, roughly speaking $E(1-2\rho)n^{3/2-\gamma} \mc Y^n(\nabla\varphi)$. In the periodic case or if the underlying lattice is $\bb Z$ we can kill that term by observing our system in a moving frame, see~\cite{gj2014} for instance, but of course this does not work in finite volume with boundary conditions. Therefore, we need to assume either $\gamma \ge \tfrac32$ or $\rho=\tfrac12$. Since we are mostly interested in the case $\gamma = \tfrac12$, we take $\rho=\tfrac12$.
\end{remark}

 From now on and up to the end of the paper we will mainly assume $\gamma=\frac12$ but we will point out the differences with respect to the case $\gamma>\frac12$. We also essentially focus on the convergence of $\mc Y^n$, since the convergence of $\mc Z^n$ will follow by very similar arguments.

\section{Proof of Theorem \ref{theo:convergence} and Theorem \ref{theo:convheight}}

\label{s4}

We start by giving all the details on the proof of Theorem \ref{theo:convergence}, and at the end of this section we present the only necessary steps which need to be adapted for the proof of Theorem \ref{theo:convheight}. They mainly concern the control of boundary terms for the height fluctuation field. 

The section is split in the following way. In Section \ref{sec:martingale} we write down the martingale decomposition which is associated to the density fluctuation field. In Section \ref{sec:BG} we state the second order Boltzmann-Gibbs principle, whose proof will be  given later in Section \ref{sec:proofBG}. This principle is needed to control the term in the martingale decomposition which gives rise in the regime $\gamma=\frac12$ to the Burgers non-linearity. In Section  \ref{sec:tightness} we prove tightness of the density fluctuation field $\mc Y^n$, and finally in Section \ref{sec:charact} we characterize the limit point as a solution to the corresponding SPDE. In Section \ref{app:height} we give the martingale decomposition for the field $\mc Z^n$, and we present the estimate that is needed in order to control extra terms that appear at the boundary.  

For the sake of clarity from now on we denote $\eta_{tn^2}:=\eta^n_{tn^2}$.

\subsection{Martingale decomposition for the density fluctuation field $\mc Y$}
\label{sec:martingale}

Fix a test function $\varphi\in\mc S_{\tmop{Dir}}$ so that $\varphi(0)=\varphi(1)=0$. 
By Dynkin's formula, we know that
\begin{equation} \label{eq:firstmar}
\mathcal{M}_t^n(\varphi)=\mathcal{Y}_t^n(\varphi)-\mathcal{Y}_0^n(\varphi)-\int_0^t n^2\mathcal{L}_n \mathcal{Y}_s^n(\varphi)\, ds
\end{equation}
and
\begin{equation}
\big(\mc M_t^n(\varphi)\big)^2 - \int_0^t n^2\mc L_n \big(\mc Y^n_s(\varphi)^2\big)-2\mc Y^n_s(\varphi)\; n^2\mc L_n \mc Y^n_s(\varphi) ds \label{eq:secondmar}\end{equation}
are martingales.
 The computations  from Appendix \ref{app:mart-dens} show that the integral term in the first martingale \eqref{eq:firstmar} rewrites as
\begin{equation}\label{eq:martdyn} \int_0^t n^2 \mc L_n\mc Y_s^n(\varphi)\; ds = \Big(1+\frac{E}{2n^\gamma}\Big)\mc I_t^n(\varphi) + E\mc A_t^n(\varphi),\end{equation} 
where 
\begin{align}
\mc I_t^n(\varphi)&= \int_0^t \mc Y_s^n(\Delta_n \varphi) \; ds, \label{eq:I}\\
\mc A_t^n(\varphi)&= - \frac{\sqrt n}{n^\gamma}\int_0^t\; \sum_{x=1}^{n-2} \nabla_n^+\varphi\big(\pfrac{x}{n}\big) \bar\eta_{sn^2}(x)\bar\eta_{sn^2}(x+1) \; ds, \label{eq:Bt}
\end{align}
and where,  above, $\nabla_n^+ \varphi$ and $\Delta_n \varphi$ are the two functions that approximate on the discrete line the gradient and Laplacian of $\varphi$, respectively. They are  defined for $x \in \Lambda_n$ by:
\begin{align*}
\nabla_n^+ \varphi\big(\pfrac{x}{n}\big) & = n\big(\varphi\big(\pfrac{x+1}{n}\big)-\varphi\big(\pfrac{x}{n}\big)\big), \\
\Delta_n\varphi\big(\pfrac{x}{n}\big) & =n^2\big(\varphi\big(\pfrac{x+1}{n}\big)-2\varphi\big(\pfrac{x}{n}\big) + \varphi\big(\pfrac{x-1}{n}\big)\big).
\end{align*}
Moreover, in \eqref{eq:Bt} we  have used a short notation for the centered variable defined as: $\bar \eta(x)=\eta(x)-\rho$ for any $x\in\Lambda_n$.
From \eqref{eq:martdyn}  we get the identity
\begin{equation}\label{mart_decomp}
\mathcal{M}_t^n(\varphi)=\mathcal{Y}_t^n(\varphi)-\mathcal{Y}_0^n(\varphi)-\Big(1+\frac{E}{2n^\gamma}\Big)\mc I_t^n(\varphi) - \mc A_t^n(\varphi).
\end{equation}
It is quite easy to see that in the macroscopic limit, the integral term $\mc I_t^n$ shall correspond to the diffusive macroscopic term $\Delta_{\rm Dir} \mc Y_t$.  Moreover, when $\gamma=\frac12$, $\mc A_t^n$ shall give rise to the non-linear term in the stochastic Burgers equation, as explained in the next lines, and it will disappear when $\gamma > \frac12$.

We also note that since  $\varphi\in \SD$, a simple computation shows that the integral term in the second martingale \eqref{eq:secondmar} rewrites as
\begin{equation}\label{eq:QV_density}
\begin{split}
&\int_0^t\frac{1}{n}\sum_{x=1}^{n-2}\Big(1+\frac{E}{n^\gamma}\Big) \big(\eta_{sn^2}(x)-\eta_{sn^2}(x+1)\big)^2 \big(\nabla_n\varphi\big(\tfrac{x}{n}\big)\big)^2ds\\
 + & \int_0^t\frac{1}{n} r_{0,1}(\eta_{sn^2}) \big(\nabla_n\varphi\big(\tfrac{1}{n}\big)\big)^2+\frac{1}{n} r_{n-1,n}(\eta_{sn^2}) \big(\nabla_n\varphi\big(\tfrac{n-1}{n}\big)\big)^2ds.
\end{split}
\end{equation}

\subsection{Case $\gamma=\tfrac12$:  second order Boltzmann-Gibbs principle} \label{sec:BG}

In this section we state another important result of this work, which is essential to the proof of Theorem \ref{theo:convergence}, since we will be able to treat the term $ \mc A_t^n(\varphi)$ given in \eqref{eq:Bt}. We focus on the case $\gamma=\frac12$, but ahead we  make some comments on the case $\gamma>\frac12$.  Before proceeding, we need to introduce some notations.

\begin{definition} For any $x\in\Lambda_n$  and $\ell_1 \in \bb N$ that satisfy $x+\ell_1 \in \Lambda_n$ (resp. $\ell_2 \in \bb N$ that satisfy $x-\ell_2 \in \Lambda_n$), we denote by $\vec\eta^{\ell_1}(x)$  (resp. $\vecleft\eta^{\ell_2}(x)$) the average centered configuration on a box of size $\ell_1$ (resp. $\ell_2$) situated to the right (resp. left) of the site $x \in \Lambda_n$:
\[
\vec\eta^{\ell_1}(x)=\frac{1}{\ell_1}\sum_{z=x+1}^{x+\ell_1} \bar\eta(z)\quad \Big(\textrm{resp.} \quad 
\vecleft\eta^{\ell_2}(x)=\frac{1}{\ell_2}\sum_{z=x-\ell_2}^{x-1}\bar\eta(z)\Big).
\]
\end{definition}

For any measurable function $v:\Lambda_n\to\bb R$, let us define $\|v\|_{2,n}^2=\frac{1}{n}\sum_{x\in\Lambda_n}v^2(x).$
From now on and up to the end,  $C>0$ is a constant that does not depend on $t > 0$, nor on $n,\ell \in \mathbb{N}$, and that may change from line to line.

\begin{theorem}[Second order Boltzmann-Gibbs principle]\label{theo:BG}
Let $v:\Lambda_n\to\bb R$ be a measurable function. 
There exists $C>0$ such that, for any $n,\ell\in\bb N$ such that $\ell < \frac n 4$ and any $t>0$,
\begin{equation}
\label{eq:BG-BCbis}
\bb E_\rho\bigg[\bigg(\int_0^t \sum_{x=1}^{n-2} v(x) \Big[\bar\eta_{sn^2}(x)\bar\eta_{sn^2}(x+1) - \mc Q(x,\ell,\eta_{sn^2})\Big] \; ds\bigg)^2\bigg] \leq Ct \Big(\frac{\ell}{n} + \frac{tn}{\ell^2}\Big) \|v\|^2_{2,n},
\end{equation}
where \begin{equation}\label{eq:Q}
\mc Q(x,\ell,\eta)= \begin{cases} 
\displaystyle \big(\vec\eta^\ell(x)\big)^2 - \frac{\chi(\rho)}{\ell},  &\text{ if }\;  x \in\{1, \dots, n-2\ell-1\}\; ; \\
~ \\
\displaystyle \big(\vecleft\eta^\ell(x)\big)^2 - \frac{\chi(\rho)}{\ell}, & \text{ if }\; x\in\{ n-2\ell, \dots, n-2\}. \end{cases}
\end{equation}
\end{theorem}
\begin{remark} 
Notice that the assumption $\ell < \frac n 4$ ensures that one of the two conditions in \eqref{eq:Q} is always satisfied and $\vec\eta^\ell(x)$ and $\vecleft\eta^\ell(x)$ are always well defined.
\end{remark}

\begin{proof}[Proof of Theorem \ref{theo:BG}]
The proof  is quite involved and is postponed to Section \ref{sec:proofBG}.
\end{proof}

Now,  let us apply Theorem \ref{theo:BG} when $\ell=\varepsilon n$ (which actually means $\lfloor \varepsilon n \rfloor$ with some abuse of notation): this choice makes the right hand side of \eqref{eq:BG-BCbis} vanish when we let first $n\to \infty$ and then $\varepsilon \to 0$. 

As a consequence, when $\gamma=\frac12$,  $\mc A_t^n(\varphi)$ is well approximated in $ L^2(\bb P_\rho)$ by the time integral of the following quantity:
\begin{align}
-E\sum_{x=1}^{n-2} &\;  \nabla_n^+ \varphi\big(\pfrac{x}{n}\big)\;  \mc Q(x,\varepsilon n,\eta_{sn^2}) \label{eq:first}\\ & = -E\sum_{x=1}^{n-1-2\varepsilon n}\nabla_n^+\varphi\big(\pfrac{x}{n}\big) \bigg[\frac{1}{\varepsilon n} \sum_{y=x+1}^{x+\varepsilon n} \bar\eta_{sn^2}(y)\bigg]^2 \notag \\
& \quad  -E \sum_{x=n-2\varepsilon n}^{n-2}\nabla_n^+\varphi\big(\pfrac{x}{n}\big) \bigg[\frac{1}{\varepsilon n} \sum_{y=x-\varepsilon n}^{x-1} \bar\eta_{sn^2}(y)\bigg]^2  -E\;\frac{\chi(\rho)}{\varepsilon n} \; n \Big[\varphi\big(\pfrac{n-1}{n}\big)-\varphi\big(\pfrac{1}{n}\big)\Big] \notag\\
& =- \frac{E}{n}\sum_{x=1}^{n-2} \nabla_n^+ \varphi\big(\pfrac{x}{n}\big)\bigg[\frac{1}{\sqrt n} \sum_{y=1}^{n-1} \bar\eta_{sn^2}(y)\; \iota_\varepsilon\big(\pfrac{x}{n}\big)\big(\pfrac{y}{n}\big)\bigg]^2   -E \;\frac{\chi(\rho)}{\varepsilon } \;  \Big[\varphi\big(\pfrac{n-1}{n}\big)-\varphi\big(\pfrac{1}{n}\big)\Big]\notag \\
& = -\frac{E}{n}\sum_{x=1}^{n-2} \nabla_n^+\varphi\big(\pfrac{x}{n}\big) \bigg[ \mc Y^n_s\Big(\iota_\varepsilon\big(\pfrac{x}{n}\big)\Big)\bigg]^2   -E \;\frac{\chi(\rho)}{\varepsilon } \;  \Big[\varphi\big(\pfrac{n-1}{n}\big)-\varphi\big(\pfrac{1}{n}\big)\Big]. \label{last}
\end{align} 
Since $\varphi(0)=\varphi(1)=0$, the second term of \eqref{last} is of order $1/n$ and therefore vanishes  when we let $n\to \infty$. We also note that for the case $\gamma>\frac12$ the previous term \eqref{eq:first} has a factor $\sqrt n/n^\gamma$ in front of it and for that reason it vanishes as $n\to\infty$. Finally, the computation above (and more precisely the first term of \eqref{last}) motivates the definition of energy solutions as defined in Theorem \ref{def:energy}, that is the definition of the macroscopic field $\mc A_t$ as given in  \eqref{eq:B}. Indeed, putting all these considerations together in the case $\gamma=\frac12$, we see that \eqref{mart_decomp} rewrites as 
\begin{align}
\notag \mc M_t^n(\varphi)=\mc Y_t^n(\varphi) - \mc Y_0^n(\varphi) &- \int_0^t \mc Y_s^n(\Delta_n\varphi) \; ds \\ & - E \int_0^t \frac{1}{n}\sum_{x=1}^{n-2} \nabla_n^+\varphi\big(\pfrac{x}{n}\big) \bigg[ \mc Y_s^n\Big(\iota_\varepsilon\big(\pfrac{x}{n}\big)\Big)\bigg]^2 \; ds + o^n_t(1), \label{eq:mart_2}
\end{align} where $o^n_t(1)$ is deterministic and satisfies $\sup_{t \in [0,T]} |o^n_t(1)| \to 0$ as $n \to \infty$. 
This computation will be useful to characterize limit points of the density fluctuation field (see Section \ref{sec:charact} below). Before that, let us prove tightness.

\subsection{Tightness of the density fluctuation  field}\label{sec:tightness}

In this section, for the sake of completeness  we show tightness of the sequence $\{{\mathcal Y}_t^n\; ; \; t \in [0,T]\}_{n\in\mathbb{N}}$, following closely \cite{kl}. The main difference  is the presence of the extra term $\mc A_t^n$ in the martingale decomposition. Tightness is a consequence of the following lemma.

\begin{lemma}\label{lem:tight}
For  $k>\frac52$,
\begin{equation*}
\begin{split}
(1) &\quad  \lim_{A\rightarrow{\infty}}\limsup_{n\rightarrow{\infty}}\; \mathbb{P}_{\rho}\Big[\sup_{t\in[0,T]}\|{{\mathcal Y}}_{t}^{n}\|_{-k}>A\Big]=0\\
(2) & \quad \lim_{\delta\rightarrow{0}}\; \limsup_{n\rightarrow{\infty}}\;\mathbb{P}_{\rho}\Big[\omega_{\delta}
({{\mathcal Y}}_{\cdot}^{n})\geq{\epsilon}\Big]=0
\end{split}
\end{equation*}
for every $\epsilon>0$, where for $\delta>0$ we define
\begin{equation*}
\omega_{\delta}({\mathcal{Y}}_{\cdot}^{n})=\sup_{\substack{|s-t|<\delta\\s,t \in [0,T]}}\|{\mathcal Y}_{t}^{n}-{{\mathcal Y}}_{s}^{n}\|_{-k}
\end{equation*}
and $\|\cdot\|_{-k}$ has been defined in \eqref{eq:norm-k}.
  \end{lemma}
  
  \begin{proof}[Proof of Lemma \ref{lem:tight}]
  We split the proof of this lemma into two steps. To prove $(1)$, from Markov's inequality it is enough to notice that
\begin{equation} \label{bound_df}
\mathbb{E}_{\rho}\bigg[\sup_{t\in[0,T]}\|{{\mathcal{Y}}}^n_{t}\|_{-k}^{2}\bigg] \le
\sum_{m\geq 1}(m\pi)^{-2k}\; \mathbb{E}_{\rho}\bigg[\sup_{t\in[0,T]}\big(\mathcal{Y}^n_{t}(e_m)\big)^{2}\bigg].
\end{equation}
Now we compute the expectation at the right hand side of \eqref{bound_df} using the martingale decomposition \eqref{mart_decomp} for $\varphi=e_m$ which makes sense since $e_m\in\SD$.
First, we have by independence under $\nu^n_\rho$,
\begin{equation}\label{initialtime}
\mathbb{E}_{\rho}\bigg[\big|{\mathcal Y}_{0}^{n}( e_m)\big|^2\bigg]\\
\leq \frac{1}{4n}\sum_{x=1}^{n-1}\Big(e_m\big(\pfrac x n\big)\Big)^2\leq \frac 1 2.
\end{equation}
Now, from the Cauchy-Schwarz inequality, invariance of $\nu^n_\rho$ and independence, we get that
\begin{equation} \label{boundintj}
\mathbb{E}_{\rho}\bigg[\sup_{t\in[0,T]}\big|\mathcal{I}_t^n(e_m)\big|^2\bigg]
\leq 
\frac{T^2}{4n}\; \sum_{x=1}^{n-1}\Big(\Delta_ne_m\big(\pfrac x n\big)\Big)^2\leq \frac{\pi^4T^2}{2} m^4.
\end{equation}
The martingale term can be easily estimated by Doob's inequality as
\begin{equation*}
\mathbb{E}_{\rho}\bigg[\sup_{t\in[0,T]}\big|\mathcal{M}_{t}^n(e_m)\big|^{2}\bigg]\leq 
C\mathbb{E}_{\rho}\bigg[\int_0^T\frac{1}{n}\sum_{x=1}^{n-1}\Big(\nabla_n^+ e_m\big(\pfrac x n\big)\Big)^2\,ds\bigg]\leq CTm^2,
\end{equation*}
where $C>0$ is a positive constant. 
Finally, in order to estimate the remaining term $\mc A_t^n(e_m)$,  we sum and subtract to $\bar\eta(x)\bar\eta(x+1)$ the term $\mathcal {Q}(x,\ell,\eta)$, and from the elementary inequality $(x+y)^2\leq 2x^2+2y^2$,  it remains  to bound the following two expectations:
\begin{equation} \label{eq:eq1}
\mathbb{E}_{\rho}\bigg[\sup_{t\in[0,T]}\bigg(\int_{0}^tE n\sum_{x=1}^{n-2}\nabla_n^+ e_m\big(\pfrac x n \big)\big(\bar\eta_{sn^2}(x)\bar\eta_{sn^2}(x+1)-\mathcal{Q}(x,\ell,\eta_{sn^2})\big)\, ds\bigg)^2\bigg] \end{equation}
and 
\begin{equation} \label{eq:eq2}
\mathbb{E}_{\rho}\bigg[\sup_{t\in[0,T]}\bigg(\int_{0}^tE\sum_{x=1}^{n-2}\nabla_n^+ e_m\big(\pfrac x n \big)\mathcal{Q}(x,\ell,\eta_{sn^2})\, ds
\bigg)^2\bigg].
\end{equation}
From the Cauchy-Schwarz inequality, the expectation \eqref{eq:eq2} is bounded from above by $T^2m^2n/\ell$ and by choosing $\ell=Cn$, we can bound it from above by $T^2m^2$.
The remaining expectation \eqref{eq:eq1} can be bounded by $C(T)m^2$, from \cite[Lemma 4.3]{CLO} and following the same steps as in the proof of Theorem \ref{theo:BG} (given in Section \ref{sec:proofBG}). 
Collecting all the previous computations we get that \eqref{bound_df}  is bounded from above by
 $C(T)\sum_{m\geq1}{m}^{4-2k},$ which is finite as long as $2k-4>1$. 

Now we prove $(2)$. For that purpose, at first we notice that from the previous computations we have:  
for $k>\frac52$
\begin{equation*}
\lim_{j\rightarrow{+\infty}}\limsup_{n\rightarrow{+\infty}}\mathbb{E}_{\rho}\bigg[\sup_{t\in[0,T]}\sum_{|m|\geq{j}}
\Big(\mathcal Y^n_{t}(e_m)\Big)^{2}(m\pi)^{-2k}\bigg]=0.
\end{equation*}
Therefore, to conclude the proof we just have to show that for any $j\in{\mathbb{N}}$ and $\epsilon>{0}$,
\begin{equation}\label{eq_1_tigh}
\lim_{\delta\rightarrow{0}}\limsup_{n\rightarrow{+\infty}}\mathbb{P}_{\rho}
\bigg[\sup_{\substack{|s-t|<\delta\\s,t\in[0,T]}}\quad
\sum_{|m|\leq{j}}\Big(\mathcal Y^n_{t}(e_m)-{\mathcal Y}^n_{s}(e_{m})\Big)^{2}(m\pi)^{-2k}>\epsilon\bigg]=0.
\end{equation}
In fact we prove that for every $m\geq 1$ and $\epsilon>0$
\begin{equation*}
\lim_{\delta\rightarrow{0}}\limsup_{n\rightarrow{+\infty}}\mathbb{P}_{\rho}\bigg[\sup_{\substack{|s-t|<\delta\\s,t\in[0,T]}}\quad
\Big|\mathcal Y^n_{t}(e_m)-{\mathcal Y}^n_{s}(e_{m})\Big|>\epsilon\bigg]=0,
\end{equation*}
from which \eqref{eq_1_tigh} follows. 
Now, as before, we recall \eqref{mart_decomp} so that the previous result is accomplished  if we derive the same result for each term in the martingale decomposition of $\mathcal Y_t^n(e_m)$. We start by the most demanding one, which is the term that involves the martingales, more precisely, we show that for every $m\geq 1$ and  $\epsilon>0$
\begin{equation} \label{eq:lim}
\lim_{\delta\rightarrow{0}}\limsup_{n\rightarrow{+\infty}}\mathbb{P}_{\rho}\bigg[\sup_{\substack{|s-t|<\delta\\s,t\in[0,T]}}\quad
\big|\mathcal{M}_{t}^{n}(e_m)-\mathcal{M}_{s}^{n}(e_m)\big|>\epsilon\bigg]=0.
\end{equation}
Since it is easy to see that 
\begin{equation*}
\limsup_{n\rightarrow{+\infty}}\mathbb{P}_{\rho}\bigg[\sup_{\substack{t\in[0,T]}}\big|\mathcal{M}_{t}^{n}(e_m)-\mathcal{M}_{t_{-}}^{n}(e_m)\big|>\epsilon\bigg]=0,
\end{equation*}
the claim \eqref{eq:lim} becomes a consequence of the following fact: 
\begin{equation*}
\lim_{\delta\rightarrow{0}}\limsup_{n\rightarrow{+\infty}}\mathbb{P}_{\rho}\Big[\omega'_{\delta}(\mathcal{M}_{t}^{n}(e_m))>\epsilon\Big]=0,
\end{equation*}
where $\omega'_{\delta}(\mathcal{M}^{n}_t(e_m))$ is the modified modulus of continuity defined by
\begin{equation*}
\omega'_{\delta}(\mathcal{M}^{n}_t(e_m))=\inf_{\substack{\{t_{i}\}}}\;\max_{\substack{0\leq{i}\leq{r}}}\;\sup_{t_{i}
\leq{s}<{t}\leq{t_{i+1}}}\big|\mathcal{M}_{t}^{n}(e_m)-\mathcal{M}_{s}^{n}(e_m)\big|,
\end{equation*}
the infimum being taken over all partitions of $[0,T]$ such that $0=t_{0}<t_{1}<...<t_{r}=T$ with $t_{i+1}-t_{i}>\delta$. 
By the Aldous criterion, see for example \cite[Proposition 4.1.6]{kl}, it is enough to show that:
\begin{equation*}
\lim_{\delta\rightarrow{0}}\limsup_{n\rightarrow{+\infty}}\sup_{\substack{\tau\in{\mathfrak
{T}_{T}}\\0\leq{\tilde\tau}\leq{\delta}}}\mathbb{P}_{\rho}\Big[\big|\mathcal{M}_{\tau+\tilde\tau}^{n}(e_m)-\mathcal{M}_{\tau}^{n}(e_m)\big|>\epsilon\Big]=0
\end{equation*}
for every $\epsilon>0$, were $\mathfrak {T}_{T}$ denotes the family of all stopping times bounded
by $T$. Using  Tchebychev's inequality  together with the optional stopping Theorem, last probability is bounded from above by
\begin{equation*}
\mathbb{E}_{\rho}\Big[\big(\mathcal{M}_{\tau+\tilde\tau}^{n}(e_m)\big)^{2}-\big(\mathcal{M}_{\tau}^{n}(e_m)\big)^{2}\Big],
\end{equation*}
which, by definition of the quadratic variation of the martingale can be bounded from above by
  $Cm^2\delta$, and vanishes as $\delta\to 0$.

Now, we compute the remaining term that involves $\mathcal I_t^n(e_m)$. We have to show that
 for every $\epsilon>0$
\begin{equation*}
\lim_{\delta\rightarrow{0}}\limsup_{n\rightarrow{+\infty}}
\mathbb{P}_{\rho}\bigg[\sup_{\substack{|s-t|<\delta\\s,t\in[0,T]}}\bigg|
\int_{s}^{t}\mathcal Y_r^n(\Delta_n e_m)\,dr\bigg|>\epsilon\bigg]=0
\end{equation*}
By Tchebychev's inequality, the last probability is bounded by  $T\delta m^4/\epsilon ^2$, which vanishes as $\delta\to0$.
For the last term involving $\mathcal A_t^n(e_m)$ we can repeat the computations that we did above: we sum and substract the term $\mc Q(x,\ell,\eta)$ as in \eqref{eq:eq1} and \eqref{eq:eq2}, then we chose $\ell = C n$, $C>0$, and we prove that each  contribution is of order $T\delta m^2/\varepsilon^2$, and therefore it goes to zero as $\delta\to0$, which finishes the proof.
\end{proof}

\subsection{Characterization of limiting points}\label{sec:charact}
In this section we prove that any limit point of the tight sequence $\{{\mathcal Y}_t^n\; ; \; t \in [0,T]\}_{n\in\mathbb{N}}$ concentrates on stationary energy solutions of \eqref{eq:SBE} as defined in Theorem \ref{def:energy}. Up to extraction, one can consider that the four sequences 
\[
\{{\mathcal Y}_t^n\; ; \; t \in [0,T]\}_{n\in\mathbb{N}};\{{\mathcal M}_t^n\; ; \; t \in [0,T]\}_{n\in\mathbb{N}}  ; \{{\mathcal I}_t^n\; ; \; t \in [0,T]\}_{n\in\mathbb{N}}; \{{{\mathcal A}_t^n}\; ; \; t \in [0,T]\}_{n\in\mathbb{N}}
\] converge as $n \to \infty$ to 
\[
\{{\mathcal Y}_t\; ; \; t \in [0,T]\}\ ;\ \{{\mathcal M}_t\; ; \; t \in [0,T]\}\ ;\  \{{\mathcal I}_t\; ; \; t \in [0,T]\}\ ;\  \{{{\mathcal A}_t}\; ; \; t \in [0,T]\}
\] respectively.  First, one can repeat the argument taken from \cite[Section 5.3]{gj2014} to prove that the limit point $\{{\mathcal Y}_t\; ; \; t \in [0,T]\}$ has continuous trajectories and it is stationary in the sense of Theorem \ref{def:energy} (see item \emph{(1)}). The characterization will be complete if we prove that this limit process also satisfies the remaining three items of Theorem \ref{def:energy}. This is what we explain briefly in the next paragraphs, since the argument is now standard and is given for example in \cite{fgsimon2015,gj2014,GJS}.

\subsubsection{Proof of item (2)}  We give a few proof elements for the sake of completeness: the Boltzmann-Gibbs principle stated in Theorem \ref{theo:BG} implies that there exists $C>0$ such that, for any {$\varphi \in \mc S_{\tmop{Dir}}$}, we have
\begin{equation}\label{eq:EstB}
\bb E_\rho\Big[\Big({\mc A_t(\varphi)-\mc A_s(\varphi)}- \chi(\rho) \mc A_{s,t}^\varepsilon(\varphi)\Big)^2\Big] \leq C(t-s)\varepsilon \big\|\nabla\varphi\big\|_{L^2([0,1])}^2\; , 
\end{equation}
where $\mc A_{s,t}^\varepsilon$ has been defined in \eqref{eq:A}.
The last claim is proved as follows: in {$\mc A_t^n(\varphi)$} given in \eqref{eq:Bt} we sum and subtract $\nabla_n^+\varphi(\frac x n)\mc Q(x,\ell,\eta_{sn^2})$ inside the sum, and we use a standard convexity inequality in order to treat two terms separately. The first one is handled using the Boltzmann-Gibbs principle, the second one is estimated thanks to the computation \eqref{last}. Then, the energy estimate \eqref{eq:lipschitz} is a consequence of  \eqref{eq:EstB}, as it follows from adding and subtracting the quantity {$\chi(\rho)^{-1}(\mc A_t(\varphi)-\mc A_s(\varphi))$} inside the square.

\subsubsection{Proof of item (3)} This point is now a straightforward consequence of the martingale decomposition given in \eqref{eq:mart_2} and in \eqref{eq:QV_density}, in which one can pass to the limit $n\to\infty$, together  with the previous paragraph.

\subsubsection{Proof of item (4)} This last property can be obtained easily by considering the reversed dynamics with the adjoint of the infinitesimal generator $\mc L_n^\star$ with respect to the Bernoulli product measure $\nu^n_{\rho}$ and repeating the same exact arguments as we did above.

\subsection{Sketch of the proof of Theorem \ref{theo:convheight}}
\label{app:height}
As mentioned previously, the proof of Theorem \ref{theo:convheight} is essentially the same as for $\mc Y^n$. Let us give here some hints and follow the sketch from the previous paragraphs. 

First, let us note that the analogue of the martingale decomposition \eqref{eq:martdyn} contains also one boundary term. Indeed, fix a  test function $\varphi\in\mc S_{\tmop{Neu}}$, and let $n^2\mc L_n^{\otimes}$ denote the generator of the joint process $\big\{ \{\eta_{tn^2}(x),h_{{t n^2}}^n(1)\}_{x\in\Lambda_n}\; ; \; t\geq 0  \big\}$. This generator acts on functions $f:\Omega_n \times \bb Z \to \bb R$ as follows: 
\begin{align}
\mc L_n^\otimes f(\eta,h) = & \; \sum_{x=1}^{n-2} r_{x,x+1}(\eta) \Big(f(\sigma^{x,x+1}\eta,h)-f(\eta,h)\Big) \notag\\
& + \eta(1)(1-\rho)\Big( f(\sigma^1\eta,h+1)-f(\eta,h)\Big) \vphantom{\frac{E}{n^\gamma}} \notag\\
& +  \Big(1+\frac{E}{n^\gamma}\Big)\rho(1-\eta(1))\Big( f(\sigma^1\eta,h-1)-f(\eta,h)\Big)  \notag\\
& +  \Big\{\Big(1+\frac{E}{n^\gamma}\Big)(1-\rho)\eta(n-1)+\rho(1-\eta(n-1))\Big\}\Big( f(\sigma^{n-1}\eta,h)-f(\eta,h)\Big) 
 \vphantom{\frac{E}{n^\gamma}}\label{eq:joint_gen}
\end{align}
where $r_{x,x+1}$ has been already defined in \eqref{eq:rate}.
From the computations in Appendix \ref{app:mart_dec_height} we get  that 
\begin{equation}
\mc N_t^n(\varphi)= \mathcal{Z}_t^n(\varphi)-\mathcal{Z}_0^n(\varphi)-\Big(1+\frac{E}{2n^\gamma}\Big)\int_0^t\mc Z_s^n(\tilde\Delta_n\varphi) ds - E{\mc B_t^n(\varphi)}+{\mc R_t^n(\varphi)}+o^n_t(1),\label{mart_decomp_height}
\end{equation}
is a martingale. Above the terms $\mc R_t^n$ and $\mc B_t^n$ are given by:
\begin{align}\label{eq:height-remainder_new}
   \mc R^n_t(\varphi) &= \sqrt n \int_0^t \bigg[ \Big( \varphi \big( \tfrac{1}{n} \big) - \varphi (0) \Big) \big(h^n_{{sn^2}}(1)-c_n s\big) - \Big( \varphi (1) - \varphi \big( \tfrac{n - 1}{n} \big) \Big)
   \big( h^n_{{sn^2}} (n)-c_ns\big)\bigg] ds\\
   \mc B^n_t(\varphi) &= \frac{\sqrt n}{n^\gamma} \int_0^t\sum_{x = 2}^{n-1} \varphi \big( \tfrac{x}{n} \big) \nabla^- h^n_{{sn^2}} (x) \nabla^+ h^n_{{sn^2}} (x) ds, \notag
\end{align}
where we set $\nabla^- h(x) = h(x) - h(x-1) =\bar\eta(x-1)$ and $\nabla^+ h(x) = h(x+1) - h(x) =\bar\eta(x)$. Moreover, in \eqref{mart_decomp_height}, $o^n_t(1)$ is a deterministic sequence of real numbers that vanishes as $n\to \infty$,  uniformly in $t\in[0,T]$, and  we also have used the notation $\tilde\Delta_n\varphi$ to denote the following approximation of the Laplacian:
\begin{equation}
\label{eq:lap2}
\tilde\Delta_n\varphi\big(\tfrac x n\big) = \begin{cases} \Delta_n\varphi \big(\tfrac x n\big)  & \text{ if } x \in \{1,\dots,n-1\},\\
2 n^2\big(\varphi\big(\tfrac{n-1}{n}\big)-\varphi\big(\tfrac{n}{n}\big)\big) & \text{ if }x=n.\end{cases}
\end{equation}
Note that, if $\varphi \in \SN$ and therefore satisfies $\nabla \varphi(1)=0$, then $\tilde\Delta_n\varphi$ is indeed an approximation of the usual Laplacian as $n\to\infty$.

Let us start with $\mc B_t^n$.  Note that we have for any test function $\varphi$, 
\[\mc B_t^n(\nabla_n^- \varphi) = -\mc A_t^n( \varphi),\] where $\nabla_n^-$ is defined similarly to $\nabla_n^+$ except that the discrete gradient is shifted, namely: $\nabla_n\varphi^-(\frac x n) = n(\varphi(\frac x n)-\varphi(\frac{x-1}{n}))$. As a consequence, this term can be treated as $\mc A_t^n$, using the Boltzmann-Gibbs principle (Theorem \ref{theo:BG}): it  gives rise to the KPZ non-linearity as soon as $\gamma=\frac12$ and  vanishes when $\gamma >\frac12$.

Next, in \eqref{mart_decomp_height}, the term $\mc R_t^n$ (which does not depend on $\gamma$) comes from boundary effects, but does not contribute to the limit if $\varphi \in \SN$, as a consequence of  the following lemma.

\begin{lemma}\label{lem:height-boundary}
  For any $\rho \in (0,1)$ we have
  \[ \lim_{n \rightarrow \infty} \bigg\{ \mathbb{E}_\rho \bigg[\sup_{t \in [0,T]} \big| n^{-
     3/2} \big(h^n_{{tn^2}} (1) - c_n t\big) \big|^2\bigg] +\mathbb{E}_\rho \bigg[\sup_{t \in [0, T]} \big| n^{- 3/2} \big(h^n_{{tn^2}} (n) - c_n t\big) \big|^2\bigg] \bigg\} = 0,
  \]
  and in particular, for any $\varphi \in \SN$, the term $\mc R^n(\varphi)$ defined in \eqref{eq:height-remainder_new} converges to $0$ in $L^2(\bb P_\rho)$, locally uniformly in time.
\end{lemma}

\begin{proof}[Proof of Lemma \ref{lem:height-boundary}]
  Since $h^n(1)$ increases by $1$ whenever a particle leaves the system and decreases by $1$ whenever a particle enters, one can easily write
  \begin{align}\label{eq:height-boundary-pr1} \nonumber
      h^n_{{tn^2}} (1) & = h_0^n(1) +  \int_0^{{tn^2}}  \bigg\{ \frac12 \eta_{{s}} (1) -  \frac12 \Big( 1 +
     \frac{E}{n^{\gamma}} \Big) \big(1-\eta_{{s}}(1)\big) \bigg\} d s + M^n_{{tn^2}}\\ 
      & =  \int_0^{{tn^2}} \bigg\{ \Big( 1 + 
     \frac{E}{2n^{\gamma}} \Big) \bar\eta_{{s}} (1) + {n^{-2}} c_n \bigg\} d s + M^n_{{tn^2}},
  \end{align}
  where $c_n=-n^{2-\gamma}E/4$ has been defined in \eqref{eq:cn} and $M^n$ is a martingale with predictable quadratic variation
  \[
     \langle M^n \rangle_t  = \frac{{1}}{2} \bigg( \Big( 1 +
     \frac{E}{n^{\gamma}} \Big)  -  \frac{E}{n^{\gamma}} \eta^n_{{t}} (1)  \bigg).
  \]
  The Burkholder-Davis-Gundy inequality implies for all $p \ge 1$ the following: there exists $C>0$ such that
  \begin{align*}
    \mathbb{E}_\rho \Big[\sup_{t \in [0,{{Tn^2}}]} \big| n^{- 3 / 2} M^n_t \big|^p\Big] & \lesssim n^{- 3
    p / 2} \; \mathbb{E}_\rho \bigg[ \bigg\{ \int_0^{{Tn^2}} \frac{{1}}{2} \bigg( \Big( 1 +
     \frac{E}{n^{\gamma}} \Big)  -  \frac{E}{n^{\gamma}} \eta^n_{{t}} (1)  \bigg) d t \bigg\}^{p / 2} \bigg]\\
    &\quad + n^{- 3 p / 2} \mathbb{E}_\rho \Big[\big| \sup_{t \in [0,{{Tn^2}}]} \Delta_t M^n \big|^p\Big],\vphantom{\int}
  \end{align*}
  where $\Delta_t M^n$ is the jump of $M^n$ at time $t$ and therefore bounded
  by $1$. In the integrand we can bound $-\eta^n_{{t}} (1)$ from above by $0$ and therefore $n^{-3/2}M_{{tn^2}}^n$ vanishes in the limit.
From \eqref{eq:height-boundary-pr1} we get
  \[
     n^{-3/2}\big(h_{{tn^2}}^n(1)-c_nt\big) =  \int_0^t \sqrt n \Big( 1 + 
     \frac{E}{2n^{\gamma}} \Big) \bar\eta_{sn^2} (1) ds + n^{-3/2} M_{{tn^2}}^n,
  \]
 and therefore we are left with bounding $\sqrt n \int_0^t  \bar\eta_{sn^2} (1)d s$. With the
  Kipnis-Varadhan Lemma given for instance in \cite[Proposition A.1.6.1]{kl} we estimate
  \[ \mathbb{E}_{\rho} \Bigg[ \sup_{t \in [0,T]} \bigg| \int_0^t
     \bar\eta_{sn^2} (1) d s \bigg|^2 \Bigg] \lesssim T \sup_{f\in L^2(\nu^n_\rho)} \bigg\{ 2
    \int \bar\eta (1) f(\eta) \nu^n_\rho(d\eta) -  n^2\mathfrak{D}_n(f) \bigg\}
  \]
where $\mathfrak{D}_n(f)$ is the Dirichlet form defined in \eqref{eq:dir_form}. From the decomposition \eqref{eq:direxpl}  we easily obtain
  \[
   \mathfrak{D}_n(f) 
     \ge (\rho \wedge (1-\rho)) \int  \big(f (\sigma^1 \eta) - f (\eta)\big)^2 \nu^n_\rho(d\eta).
  \]
  Moreover, as in \cite[Lemma~3]{Franco2016} we can use Young's inequality to get
  \begin{align*}
    2\int \bar\eta (1)  f(\eta)  \nu^n_\rho(d\eta) & =  \int \bar\eta
    (1) \big(f (\eta) - f (\sigma^1 \eta)\big) \nu^n_\rho(d\eta) \\ &  \leq \varepsilon \int \big(\bar\eta(1)\big)^2 \nu^n_\rho(d\eta)  +  \frac{1}{4\varepsilon}\int \big(f
    (\eta) - f (\sigma^1 \eta)\big)^2 \; \nu^n_\rho(d\eta) ,
  \end{align*}
 for any $\varepsilon >0$, that we choose now such that 
 $(4\varepsilon)^{-1} = {(\rho \wedge (1-\rho))\; n^2}$.  And therefore we obtain
  \[
    2
    \int \bar\eta (1)  f(\eta) \nu^n_\rho(d\eta) -  n^2\mathfrak{D}_n(f)
    \lesssim  \frac{1}{n^2} \int \big(\bar\eta(1)\big)^2 \nu^n_\rho(d\eta) \simeq \frac{1}{n^2},
  \]
  which leads to
  \[  \mathbb{E}_{\rho} \Bigg[ \sup_{t \in [0,T]} \bigg|\sqrt n \int_0^t
     \bar\eta_{sn^2} (1) d s \bigg|^2 \Bigg] \lesssim
     \frac{T}{n}.  \]
  The bound for $h^n (n)$ is shown with the same arguments.
    To conclude the proof, take $\varphi \in \SN$, so that $\nabla \varphi(0)=\nabla\varphi(1)=0$. In that case 
\[ n^2\big(\varphi\big(\tfrac1n\big)-\varphi(0)\big) \xrightarrow[n\to\infty]{}\Delta\varphi(0), \qquad \text{and}\qquad  n^2\big(\varphi(1)-\varphi\big(\tfrac{n-1}{n}\big)\big) \xrightarrow[n\to\infty]{}\Delta\varphi(1),  \]  
  and therefore 
\[\bb E_\rho\Big[  \sup_{t\in[0,T]} \big|\mc R_t^n(\varphi)\big|^2\Big] \lesssim \frac{T}{n}.\] 
\end{proof}

There are two remaining steps, the first one being tightness of the sequence $\{\mc Z_t^n \; ; \; t \in [0,T]\}_{n\in\bb N}$.  We let the reader repeat the same proof as in Lemma \ref{lem:tight}, noting that since the height fluctuation field  is now defined in $\mc H_{\rm Neu}^{-k}$, the basis that one has to use is $\tilde e_m$ given in Section \ref{ssec:test_functions}. The arguments remain unchanged, we only note that the restriction $k>\frac52$ comes from the analog of \eqref{boundintj}.

Finally, for the characterization of limit points we essentially use the relation between $\mc Z_t^n$ and $\mc Y_t^n$, which reads: for any $\varphi \in \SD$,
\begin{equation}
\label{eq:relation}
\mc Z_t^n(\tilde\nabla_n\varphi) = -\mc Y_t^n(\varphi) - \tfrac{1}{\sqrt n}\varphi\big(\tfrac1n\big) \big(h_{tn^2}^n(1)-c_nt\big) +  \tfrac{1}{\sqrt n}\varphi\big(\tfrac n n\big)\big(h_{tn^2}^n(n)-c_nt\big),
\end{equation}
where $\tilde\nabla_n\varphi(\frac x n\big)=\mathbf{1}_{\{1,\dots,n-1\}}(x)\;\nabla_n^+\varphi(\frac x n\big)$, in particular $\tilde\nabla_n\varphi(\frac n n)=0$. Since $\varphi \in \SD$, Lemma \ref{lem:height-boundary} implies that the last two terms in \eqref{eq:relation} vanish in $L^2(\bb P_\rho)$ as $n\to\infty$ uniformly in $t\in[0,T]$. If $\mc Z$ is the limit point of $\mc Z^n$ then passing to the limit in \eqref{eq:relation}, we get: for any $\varphi \in \SD$, 
\[\nabla_{\rm Dir}\mc Z_t(\varphi) = - \mc Z_t(\nabla \varphi) = \mc Y_t(\varphi),\]
where we used the definition of $\nabla_{\rm Dir}$ given in \eqref{eq:nabladir}. From this, we deduce item \emph{(1)} of Theorem \ref{def:energy-KPZ}. The last item \emph{(3)} is obtained similarly combining \eqref{eq:relation} with the two martingale decompositions \eqref{mart_decomp} and \eqref{mart_decomp_height}.  For the quadratic variation we observe that
by Dynkin's formula
\begin{equation*}
\Big(\mc N_t^n(\varphi)\Big)^2-\int_0^t n^2\mc L_n^\otimes\big(\mc Z^n_s(\varphi)^2\big)-2\mc Z^n_s(\varphi)\; n^2\mc L_n^\otimes \mc Z^n_s(\varphi) ds
\end{equation*}
is a martingale. Here we used that the drift $-c_n t$ gives rise to a first order differential operator $\mc G^n$ which satisfies Leibniz's rule and therefore the difference $\mc G^n\big(\mc Z^n_s(\varphi)^2\big)-2\mc Z^n_s(\varphi) \mc G^n \mc Z^n_s(\varphi) $ vanishes.
A simple computation shows that last integral term can be rewritten as
\begin{equation}
\begin{split}
&\int_0^t\frac{1}{n}\sum_{x=2}^{n-2}\Big(1+\frac{E}{n^\gamma}\Big) \big(\eta_{sn^2}(x)-\eta_{sn^2}(x+1)\big)^2 \big(\varphi\big(\tfrac{x}{n}\big)\big)^2ds\\
+ &\int_0^t\frac{1}{n} r_{0,1}(\eta_{sn^2}) \big(\varphi\big(\tfrac{1}{n}\big)\big)^2+\frac{1}{n} r_{n-1,n}(\eta_{sn^2}) \big(\varphi\big(\tfrac{n-1}{n}\big)\big)^2ds.
\end{split}
\end{equation}
By taking expectation w.r.t.~$\nu_\rho^n$ and sending  $n\to\infty$  we conclude item \emph{(3)}.

\section{Proof of the second order Boltzmann-Gibbs principle}\label{sec:proofBG}

\label{s5}

In this section we prove the Boltzmann-Gibbs principle stated in Theorem \ref{theo:BG}.

Let us illustrate how the proof of this principle works: let us choose a site $x$ which is not too close to the right boundary, in the sense that there are at least $2\ell$ sites between $x$ and $n-1$, then we can replace the local function $\bar\eta(x)\bar\eta(x+1)$ by the square of the average to its right $\big(\vec\eta^\ell(x)\big)^2$ (see Figure \ref{fig:bg}). The main reason to keep at least $2\ell$ sites between $x$ and $n-1$ is because the proof makes uses of the sites situated between $x+\ell+1$ and $x+2\ell$, as explained in Section \ref{ssec:strategy} below. Otherwise, when $x+2\ell > n-1$, we replace the same local function by the square of the average to its left $(\vecleft\eta^\ell(x))^2$ (see Figure \ref{fig:bg2}). 

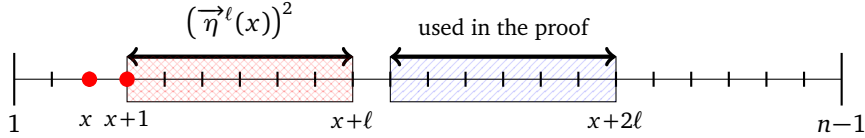
\begin{figure}[H]
\centering
\begin{tikzpicture}
\draw (-3,0) -- (8,0);
\draw[thick] (-3,-0.35) -- (-3,0.35);
\draw[thick] (8,-0.35) -- (8,0.35);
\draw[thick] (-2.5,-0.1) -- (-2.5,0.1);
\draw[thick] (-2,-0.1) -- (-2,0.1);
\draw[thick] (-1.5,-0.1) -- (-1.5,0.1);
\draw[thick] (-1,-0.1) -- (-1,0.1);
\draw[thick] (-0.5,-0.1) -- (-0.5,0.1);
\draw[thick] (0,-0.1) -- (0,0.1);
\draw[thick] (0.5,-0.1) -- (0.5,0.1);
\draw[thick] (1,-0.1) -- (1,0.1);
\draw[thick] (1.5,-0.1) -- (1.5,0.1);
\draw[thick] (2,-0.1) -- (2,0.1);
\draw[thick] (2.5,-0.1) -- (2.5,0.1);
\draw[thick] (3,-0.1) -- (3,0.1);
\draw[thick] (3.5,-0.1) -- (3.5,0.1);
\draw[thick] (4,-0.1) -- (4,0.1);
\draw[thick] (4.5,-0.1) -- (4.5,0.1);
\draw[thick] (5,-0.1) -- (5,0.1);
\draw[thick] (5.5,-0.1) -- (5.5,0.1);
\draw[thick] (6,-0.1) -- (6,0.1);
\draw[thick] (6.5,-0.1) -- (6.5,0.1);
\draw[thick] (7,-0.1) -- (7,0.1);
\draw[thick] (7.5,-0.1) -- (7.5,0.1);
\draw (-2,-0.3) node[anchor=north] {\small $x$ \vphantom{$+1$}};
\draw (-1.5,-0.3) node[anchor=north] {\small $x\!+\!1$};
\draw (1.5,-0.3) node[anchor=north] {\small $x\!+\!\ell$};
\draw (5,-0.3) node[anchor=north] {\small $x\!+\!2\ell$};

\draw (-3,-0.35) node[anchor=north] { $1$};
\draw (8,-0.35) node[anchor=north] { $n\!-\!1$};
\draw[color=white] (0,-1) node[anchor=north] {x};

\draw[ultra thick, <->] (-1.5,0.3) -- (1.5,0.3);
\draw[-] (-1.5,-0.3)--(-1.5,0.3)--(1.5,0.3)--(1.5,-0.3)--cycle;
\draw[-] (2,-0.3)--(2,0.3)--(5,0.3)--(5,-0.3)--cycle;
\fill[pattern=crosshatch, pattern color=red, opacity=0.4] (-1.5,-0.3)--(-1.5,0.3)--(1.5,0.3)--(1.5,-0.3)--cycle;
\fill[pattern=north east lines, pattern color=blue, opacity=0.4] (2,-0.3)--(2,0.3)--(5,0.3)--(5,-0.3)--cycle;
\draw (0,0.4) node[anchor=south] {$\big(\vec\eta^\ell(x)\big)^2$};
\draw (3.5,0.4) node[anchor=south] {\small \textrm{used in the proof}};
\draw[ultra thick, <->] (2,0.3) -- (5,0.3);
\filldraw[black, thick, color=red] (-2,0) circle (2.5pt) ;
\filldraw[black, thick, color=red] (-1.5,0) circle (2.5pt) ;
%

\end{tikzpicture}
\caption{Replacement for the local function $\bar\eta(x)\bar\eta(x+1)$ when $x+2\ell \leq n-1$.}
\label{fig:bg}
\end{figure}

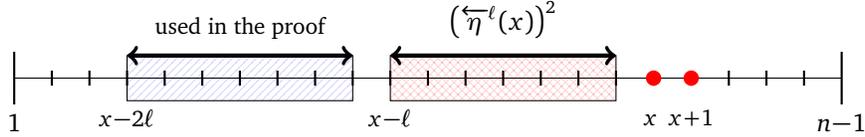
\begin{figure}[H]
\centering
\begin{tikzpicture}
\draw (-3,0) -- (8,0);
\draw[thick] (-3,-0.35) -- (-3,0.35);
\draw[thick] (8,-0.35) -- (8,0.35);
\draw[thick] (-2.5,-0.1) -- (-2.5,0.1);
\draw[thick] (-2,-0.1) -- (-2,0.1);
\draw[thick] (-1.5,-0.1) -- (-1.5,0.1);
\draw[thick] (-1,-0.1) -- (-1,0.1);
\draw[thick] (-0.5,-0.1) -- (-0.5,0.1);
\draw[thick] (0,-0.1) -- (0,0.1);
\draw[thick] (0.5,-0.1) -- (0.5,0.1);
\draw[thick] (1,-0.1) -- (1,0.1);
\draw[thick] (1.5,-0.1) -- (1.5,0.1);
\draw[thick] (2,-0.1) -- (2,0.1);
\draw[thick] (2.5,-0.1) -- (2.5,0.1);
\draw[thick] (3,-0.1) -- (3,0.1);
\draw[thick] (3.5,-0.1) -- (3.5,0.1);
\draw[thick] (4,-0.1) -- (4,0.1);
\draw[thick] (4.5,-0.1) -- (4.5,0.1);
\draw[thick] (5,-0.1) -- (5,0.1);
\draw[thick] (5.5,-0.1) -- (5.5,0.1);
\draw[thick] (6,-0.1) -- (6,0.1);
\draw[thick] (6.5,-0.1) -- (6.5,0.1);
\draw[thick] (7,-0.1) -- (7,0.1);
\draw[thick] (7.5,-0.1) -- (7.5,0.1);
\draw (5.5,-0.3) node[anchor=north] {\small $x$ \vphantom{$+1$}};
\draw (6,-0.3) node[anchor=north] {\small $x\!+\!1$};
\draw (2,-0.3) node[anchor=north] {\small $x\!-\!\ell$};
\draw (-1.5,-0.3) node[anchor=north] {\small $x\!-\!2\ell$};

\draw (-3,-0.35) node[anchor=north] { $1$};
\draw (8,-0.35) node[anchor=north] { $n\!-\!1$};
\draw[color=white] (0,-1) node[anchor=north] {x};

\draw[ultra thick, <->] (-1.5,0.3) -- (1.5,0.3);
\draw[-] (-1.5,-0.3)--(-1.5,0.3)--(1.5,0.3)--(1.5,-0.3)--cycle;
\draw[-] (2,-0.3)--(2,0.3)--(5,0.3)--(5,-0.3)--cycle;
\fill[pattern=north east lines, pattern color=blue, opacity=0.4] (-1.5,-0.3)--(-1.5,0.3)--(1.5,0.3)--(1.5,-0.3)--cycle;
\fill[pattern=crosshatch, pattern color=red, opacity=0.4] (2,-0.3)--(2,0.3)--(5,0.3)--(5,-0.3)--cycle;
\draw (0,0.4) node[anchor=south] {\small \textrm{used in the proof}};
\draw (3.5,0.4) node[anchor=south] {$\big(\vecleft\eta^\ell(x)\big)^2$};
\draw[ultra thick, <->] (2,0.3) -- (5,0.3);
\filldraw[black, thick, color=red] (5.5,0) circle (2.5pt) ;
\filldraw[black, thick, color=red] (6,0) circle (2.5pt) ;
%

\end{tikzpicture}
\caption{Replacement for the local function $\bar\eta(x)\bar\eta(x+1)$ when $x+2\ell > n-1$.}
\label{fig:bg2}
\end{figure}

Before going into the proof details, let us introduce some notations: for a function $g:\Omega_n\to\bb R$ we denote by $\|g\|_2$ its $ L^2(\nu^n_\rho)$--norm: 
\[
\|g\|_2^2=\int_{\Omega_n} g^2(\eta) \; \nu^n_\rho(d\eta).
\]
In the following, $C=C(\rho)$ denotes a constant that does not depend on $n$ nor on $t$ nor on the sizes of all boxes involved, and that may change from line to line. We fix once and for all a measurable function $v:\Lambda_n\to\bb R$, for which $\|v\|_{2,n}^2<\infty$. We say that a function $g:\{0,1 \}^{\bb Z} \to \bb R$ has its support (denoted below by $ \mathrm{Supp}(g)$) included in some subset $\Lambda \subset \bb Z$ if $g$ depends on the configuration $\eta$ only through the variables $\{\eta(x) \; ; \; x \in \Lambda\}$. We denote by $\tau_x$ the usual shift operator, that acts on functions $g:\Omega_n\to\bb R$ as follows: $\tau_xg(\eta)=g(\tau_x\eta)$, which is well defined for any $x$ such that $\mathrm{Supp}(g) \subset \{1,\dots,n-1-x\}$. 
To keep the presentation as clear as possible, we define two quantities that are needed in due course:
\begin{definition} Let $m\in\Lambda_n$ be an integer such that $m  <\frac n 2$, and  let  \begin{itemize} 
\item  $g^\rightarrow:\{0,1\}^{\bb Z}\to\bb R$ be a   function whose support is included in $\{0,..,m\}$, \item $g^\leftarrow:\{0,1\}^{\bb Z}\to\bb R$ be a function whose support is included in $\{-m,...,1\}$. 
\end{itemize}
Let us define
\begin{align*}
\mathfrak{I}^{\mathrm{left}}_{t,n}(g^\rightarrow)&=\bb E_\rho\bigg[\bigg(\int_0^t \sum_{x=1}^{n-2\ell-1} v(x) \tau_xg^\rightarrow(\eta_{sn^2}) \; ds\bigg)^2\bigg],\\
\mathfrak{I}^{\mathrm{right}}_{t,n}(g^\leftarrow)&=\bb E_\rho\bigg[\bigg(\int_0^t \sum_{x=n-2\ell}^{n-2} v(x) \tau_xg^\leftarrow(\eta_{sn^2}) \; ds\bigg)^2\bigg]. 
\end{align*}
\end{definition}

With this definition, \eqref{eq:BG-BCbis}  follows from showing that for any $n,\ell \in \bb N$ such that $\ell <\frac n 4$, and any $t>0$,
\begin{align}
\mathfrak{I}^{\mathrm{left}}_{t,n}(g_\ell^\rightarrow)&\leq Ct\bigg(\frac{\ell}{n}+\frac{tn}{\ell^2}\bigg) \|v\|_{2,n}^2 \label{eq:Ileft}\\
\mathfrak{I}^{\mathrm{right}}_{t,n}(g_\ell^\leftarrow)& \leq Ct\bigg(\frac{\ell}{n}+\frac{tn}{\ell^2}\bigg)\|v\|_{2,n}^2 \label{eq:Iright}
\end{align}
where the two local functions $g_\ell^\rightarrow$ and $g_\ell^\leftarrow$ are given by \begin{align} g_\ell^\rightarrow(\eta)&=\bar\eta(0)\bar\eta(1) - \big(\vec\eta^\ell(0)\big)^2 + \frac{\chi(\rho)}{\ell}, \label{eq:psileft}\\
g_\ell^\leftarrow(\eta)&=\bar\eta(0)\bar\eta(1) - \big(\vecleft\eta^\ell(0)\big)^2 + \frac{\chi(\rho)}{\ell}. \label{eq:psiright}
\end{align}

\subsection{Strategy of the proof} 

\label{ssec:strategy}

We prove \eqref{eq:Ileft} and \eqref{eq:Iright} separately. For both of them, we need to decompose $g^\rightarrow_\ell$ and $g^\leftarrow_\ell$ as sums of several local functions, for which the estimates are simpler.  With a small abuse of language, we say that, at each step, we \emph{replace} a local function with another one. More precisely, let $\ell_0 \leq \ell$ and assume first (to simplify) that $\ell= 2^M\ell_0$ for some integer $M\in\bb N$. Denote $\ell_k=2^k\ell_0$ for $k\in\{0,...,M\}$. One can easily check the decomposition
\begin{align}
g^\rightarrow_\ell(\eta) = & \quad \bar\eta(0)\Big[\bar\eta(1)-\vec\eta^{\ell_0}(\ell_0)\Big] \label{eq:right1} \\
& + \vec\eta^{\ell_0}(\ell_0) \Big[ \bar\eta(0)-\vec\eta^{\ell_0}(0)\Big] \label{eq:right2}\\
& + \sum_{k=0}^{M-1} \vec\eta^{\ell_k}(0) \Big[\vec\eta^{\ell_k}(\ell_k)-\vec\eta^{2\ell_k}(2\ell_k)\Big] \label{eq:right3}\\
& + \sum_{k=0}^{M-1} \vec\eta^{2\ell_k}(2\ell_k) \Big[\vec\eta^{\ell_k}(0)-\vec\eta^{2\ell_k}(0)\Big]\label{eq:right4}\\
&+ \vec\eta^\ell(0) \Big[\vec\eta^\ell(\ell)-\bar\eta(\ell+1)\Big]\label{eq:right5}\\
&+ \vec\eta^\ell(0) \Big[\bar\eta(\ell+1)-\bar\eta(0)\Big]\label{eq:right6}\\
&+ \vec\eta^\ell(0) \Big[\bar\eta(0)-\vec\eta^\ell(0)\Big] + \frac{\chi(\rho)}{\ell}. \label{eq:right7}
\end{align}
For example, in \eqref{eq:right1} we say that we \emph{replace} $\bar\eta(1)$ by $\vec\eta^{\ell_0}(\ell_0)$, while $\bar\eta(0)$ is considered as \emph{fixed}. Seven terms appear from \eqref{eq:right1} to \eqref{eq:right7}. Let us denote them by order of appearance as follows:
\[
g^\rightarrow_{\rm I}(\eta), \quad g^\rightarrow_{\rm II}(\eta),\quad g^\rightarrow_{\rm III}(\eta),\quad g^\rightarrow_{\rm IV}(\eta), \quad g^\rightarrow_{\rm V}(\eta),\quad g^\rightarrow_{\rm VI}(\eta),\quad g^\rightarrow_{\rm VII}(\eta).
\]
The decomposition above can naturally be written for $\tau_x g^\rightarrow_\ell$ ($x\in\Lambda_n$) by translating any term.
Let us now illustrate the first steps of the decomposition: in Figure \ref{fig:bgleft1} below, we use the arrows as symbols for the replacements we perform, and we illustrate the consecutive replacements from \eqref{eq:right1} to \eqref{eq:right3}, the latter corresponding to  $\vec\eta^{\ell_k}(x+\ell_k) \mapsto \vec\eta^{2\ell_k}(x+2\ell_k).$

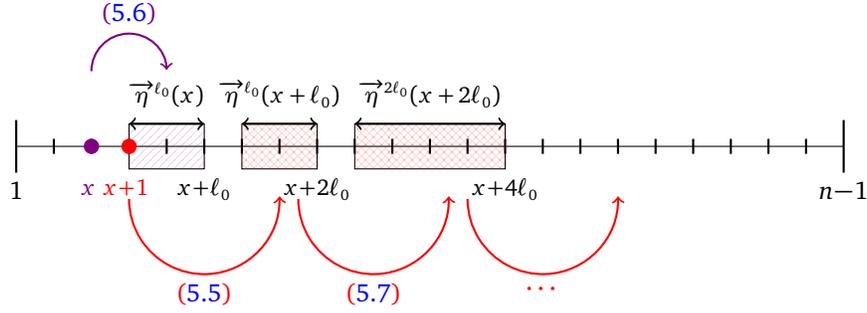
\begin{figure}[H]
\centering
\begin{tikzpicture}
\draw (-3,0) -- (8,0);
\draw[thick] (-3,-0.35) -- (-3,0.35);
\draw[thick] (8,-0.35) -- (8,0.35);
\draw[thick] (-2.5,-0.1) -- (-2.5,0.1);
\draw[thick] (-2,-0.1) -- (-2,0.1);
\draw[thick] (-1.5,-0.1) -- (-1.5,0.1);
\draw[thick] (-1,-0.1) -- (-1,0.1);
\draw[thick] (-0.5,-0.1) -- (-0.5,0.1);
\draw[thick] (0,-0.1) -- (0,0.1);
\draw[thick] (0.5,-0.1) -- (0.5,0.1);
\draw[thick] (1,-0.1) -- (1,0.1);
\draw[thick] (1.5,-0.1) -- (1.5,0.1);
\draw[thick] (2,-0.1) -- (2,0.1);
\draw[thick] (2.5,-0.1) -- (2.5,0.1);
\draw[thick] (3,-0.1) -- (3,0.1);
\draw[thick] (3.5,-0.1) -- (3.5,0.1);
\draw[thick] (4,-0.1) -- (4,0.1);
\draw[thick] (4.5,-0.1) -- (4.5,0.1);
\draw[thick] (5,-0.1) -- (5,0.1);
\draw[thick] (5.5,-0.1) -- (5.5,0.1);
\draw[thick] (6,-0.1) -- (6,0.1);
\draw[thick] (6.5,-0.1) -- (6.5,0.1);
\draw[thick] (7,-0.1) -- (7,0.1);
\draw[thick] (7.5,-0.1) -- (7.5,0.1);
\draw (-2,-0.3) node[anchor=north] {\small \color{violet}$x$ \vphantom{$+\ell_0$}};
\draw (-1.5,-0.3) node[anchor=north] {\small \color{red}$x\!+\!1$ \vphantom{$+\ell_0$}};
\draw (-0.5,-0.3) node[anchor=north] {\small $x\!+\!\ell_0$};

\draw (1,-0.3) node[anchor=north] {\small $x\!+\!2\ell_0$};
\draw (3.5,-0.3) node[anchor=north] {\small $x\!+\!4\ell_0$};

\draw (-3,-0.35) node[anchor=north] { $1$};
\draw (8,-0.35) node[anchor=north] { $n\!-\!1$};
\draw[color=white] (0,-1) node[anchor=north] {x};

\centerarc[thick,<-,color=red](-0.5,-0.7)(2:-180:1);
\draw (-0.5,-1.7) node[anchor=north] {\color{red}\bf \eqref{eq:right1}};
\centerarc[thick,<-,color=red](1.75,-0.7)(2:-180:1);
\draw (1.75,-1.7) node[anchor=north] {\color{red}\bf \eqref{eq:right3}};
\centerarc[thick,<-,color=red](4,-0.7)(2:-180:1);
\draw (4,-1.7) node[anchor=north] {\large \color{red}\bf $\cdots$};

\centerarc[thick,<-,color=violet](-1.5,1)(2:180:0.5);
\draw (-1.5,1.5) node[anchor=south] {\color{violet}\bf\eqref{eq:right2}};

\draw[thick, <->] (0,0.3) -- (1,0.3);
\draw[thick, <->] (-1.5,0.3) -- (-0.5,0.3);
\draw[-] (0,-0.3)--(0,0.3)--(1,0.3)--(1,-0.3)--cycle;
\draw[-] (-1.5,-0.3)--(-1.5,0.3)--(-0.5,0.3)--(-0.5,-0.3)--cycle;
\draw[-] (1.5,-0.3)--(1.5,0.3)--(3.5,0.3)--(3.5,-0.3)--cycle;
\fill[pattern=crosshatch, pattern color=red, opacity=0.4] (0,-0.3)--(0,0.3)--(1,0.3)--(1,-0.3)--cycle;
\fill[pattern=north east lines, pattern color=violet, opacity=0.4] (-1.5,-0.3)--(-1.5,0.3)--(-0.5,0.3)--(-0.5,-0.3)--cycle;
\fill[pattern=crosshatch, pattern color=red, opacity=0.4] (1.5,-0.3)--(1.5,0.3)--(3.5,0.3)--(3.5,-0.3)--cycle;
\draw (0.5,0.4) node[anchor=south] {\small $\vec\eta^{\ell_0}(x+\ell_0)$};
\draw (-1,0.4) node[anchor=south] {\small $\vec\eta^{\ell_0}(x)$};
\draw (2.5,0.4) node[anchor=south] {\small $\vec\eta^{2\ell_0}(x+2\ell_0)$};
\draw[thick, <->] (1.5,0.3) -- (3.5,0.3);
\filldraw[black, thick, color=violet] (-2,0) circle (2.5pt) ;
\filldraw[black, thick, color=red] (-1.5,0) circle (2.5pt) ;
%

\end{tikzpicture}
\caption{Illustration of steps \eqref{eq:right1}--\eqref{eq:right3} when $x\leq n-2\ell-1$.}
\label{fig:bgleft1}
\end{figure}
Simultaneously, one can see in \eqref{eq:right4} that $\vec\eta^{\ell_k}(x)$ is replaced with $\vec\eta^{2\ell_k}(x)$ as we illustrate now in Figure \ref{fig:bgleft2}:
\begin{figure}[H]
\centering
\begin{tikzpicture}
\draw (-3,0) -- (8,0);
\draw[thick] (-3,-0.35) -- (-3,0.35);
\draw[thick] (8,-0.35) -- (8,0.35);
\draw[thick] (-2.5,-0.1) -- (-2.5,0.1);
\draw[thick] (-2,-0.1) -- (-2,0.1);
\draw[thick] (-1.5,-0.1) -- (-1.5,0.1);
\draw[thick] (-1,-0.1) -- (-1,0.1);
\draw[thick] (-0.5,-0.1) -- (-0.5,0.1);
\draw[thick] (0,-0.1) -- (0,0.1);
\draw[thick] (0.5,-0.1) -- (0.5,0.1);
\draw[thick] (1,-0.1) -- (1,0.1);
\draw[thick] (1.5,-0.1) -- (1.5,0.1);
\draw[thick] (2,-0.1) -- (2,0.1);
\draw[thick] (2.5,-0.1) -- (2.5,0.1);
\draw[thick] (3,-0.1) -- (3,0.1);
\draw[thick] (3.5,-0.1) -- (3.5,0.1);
\draw[thick] (4,-0.1) -- (4,0.1);
\draw[thick] (4.5,-0.1) -- (4.5,0.1);
\draw[thick] (5,-0.1) -- (5,0.1);
\draw[thick] (5.5,-0.1) -- (5.5,0.1);
\draw[thick] (6,-0.1) -- (6,0.1);
\draw[thick] (6.5,-0.1) -- (6.5,0.1);
\draw[thick] (7,-0.1) -- (7,0.1);
\draw[thick] (7.5,-0.1) -- (7.5,0.1);
\draw (-2,-0.3) node[anchor=north] {\small \color{violet}$x$ \vphantom{$+\ell_0$}};
\draw (-1.5,-0.3) node[anchor=north] {\small \color{red}$x\!+\!1$ \vphantom{$+\ell_0$}};
%

\draw (-3,-0.35) node[anchor=north] { $1$};
\draw (8,-0.35) node[anchor=north] { $n\!-\!1$};
\draw[color=white] (0,-1) node[anchor=north] {x};

\centerarc[thick,<-,color=violet](-1,-1)(2:-180:1);
\draw (0,-2) node[anchor=north] {\color{violet}\bf \eqref{eq:right4} $\cdots$};

\centerarc[thick,<-,color=violet](0,-1)(2:-180:2);

\centerarc[thick,<-,color=violet](-1.5,1)(2:180:0.5);
\draw (-1.5,1.5) node[anchor=south] {\color{violet}\bf\eqref{eq:right2}};

\draw[thick, <->] (-1.5,-0.3) -- (0.5,-0.3);
\draw[thick, <->] (-1.5,0.3) -- (-0.5,0.3);
\draw[-] (-0.5,-0.3)--(-0.5,0.3)--(0.5,0.3)--(0.5,-0.3)--cycle;
\draw[-] (-1.5,-0.3)--(-1.5,0.3)--(-0.5,0.3)--(-0.5,-0.3)--cycle;
\fill[pattern=crosshatch, pattern color=violet, opacity=0.4] (-0.5,-0.3)--(-0.5,0.3)--(0.5,0.3)--(0.5,-0.3)--cycle;
\fill[pattern=north east lines, pattern color=violet, opacity=0.4] (-1.5,-0.3)--(-1.5,0.3)--(-0.5,0.3)--(-0.5,-0.3)--cycle;
\draw (0,-0.4) node[anchor=north] {\small $\vec\eta^{2\ell_0}(x)$};
\draw (-1,0.4) node[anchor=south] {\small $\vec\eta^{\ell_0}(x)$};
\filldraw[black, thick, color=violet] (-2,0) circle (2.5pt) ;
\filldraw[black, thick, color=red] (-1.5,0) circle (2.5pt) ;
%

\end{tikzpicture}
\caption{Illustration of step \eqref{eq:right4}: successive replacements  when $x \leq n-2\ell-1$.}
\label{fig:bgleft2}
\end{figure}
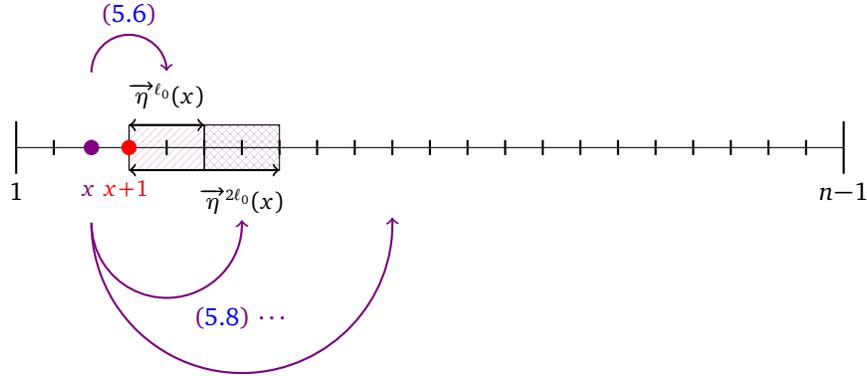

The role of the pre-factors $\vec\eta^{\ell_k}(x)$ in \eqref{eq:right3} and $\vec\eta^{2\ell_k}(x+2\ell_k)$ in \eqref{eq:right4} can be roughly understood as follows: these local functions have a variance of order $(\ell_k)^{-1}$ under $\nu^n_\rho$, which compensates the price to pay when one tries to replace $\vec\eta^{\ell_k}(x+\ell_k)$ by $\vec\eta^{2\ell_k}(x+2\ell_k)$ and $\vec\eta^{\ell_k}(x)$ by $\vec\eta^{2\ell_k}(x)$. More precisely, this compensation is optimal if the support of the pre-factor does not intersect the set of sites which are used in the replacement: for example, the support of $\vec\eta^{\ell_k}(x)$ is   $\{x+1,...,x+\ell_k\}$ and it does not intersect $\{x+\ell_k+1,...,x+4\ell_k\}$, which corresponds to the sites used in the replacement $\vec\eta^{\ell_k}(x+\ell_k)\mapsto \vec\eta^{2\ell_k}(x+2\ell_k)$, see \eqref{eq:right3}. More details are given below.

The decomposition which works for $g^\leftarrow_\ell$ is very similar and there is no difficulty to find it out, following closely \eqref{eq:right1}--\eqref{eq:right7}. 

 Let us go back to our goal estimate  \eqref{eq:Ileft}. From the standard convexity inequality $(a_1+\cdots+a_p)^2 \leq p(a_1^2+\cdots a_p^2)$, one can see that \eqref{eq:Ileft} follows from seven independent estimates. More precisely, it is enough to prove that
\begin{equation}\label{eq:separ}
\mf I_{t,n}^{\mathrm{left}}(g^\rightarrow_{\rm w}) \leq Ct\bigg(\frac{\ell}{n}+\frac{tn}{\ell^2}\bigg)\|v\|_{2,n}^2, \qquad \text{for any } {\rm w} \in \{\rm I, II, III, IV, V, VI, VII\}.\end{equation}
There is one further step in two particular cases: for $\rm w=III$ and $\rm w=IV$, we also use Minkowski's inequality, so that we can  bound as follows:
\begin{align}
 \mf I_{t,n}^{\mathrm{left}}(g^\rightarrow_{\rm III})  & \leq  C\bigg(\sum_{k=0}^{M-1} \bigg\{\bb E_\rho\bigg[\bigg(\int_0^t \sum_{x=1}^{n-2\ell-1} v(x)\vec\eta_{sn^2}^{\ell_k}(x) \notag \\ &    \qquad \qquad \qquad \times \Big[\vec\eta_{sn^2}^{\ell_k}(x+\ell_k)-\vec\eta_{sn^2}^{2\ell_k}(x+2\ell_k)\Big]ds\bigg)^2\bigg] \bigg\}^{1/2} \bigg)^2 \label{eq:right3bis} \\
 \mf I_{t,n}^{\mathrm{left}}(g^\rightarrow_{\rm IV}) 
& \leq C\bigg(\sum_{k=0}^{M-1} \bigg\{\bb E_\rho\bigg[\bigg(\int_0^t \sum_{x=1}^{n-2\ell-1} v(x)\vec\eta_{sn^2}^{2\ell_k}(x+2\ell_k) \notag \\ &   \qquad  \qquad \qquad  \times \Big[\vec\eta_{sn^2}^{\ell_k}(x)-\vec\eta_{sn^2}^{2\ell_k}(x)\Big]ds\bigg)^2\bigg] \bigg\}^{1/2} \bigg)^2.\label{eq:right4bis}
\end{align}
Finally, the proof of \eqref{eq:separ} can  almost be resumed in one general statement, which we are going to apply several times. Let us state here our main estimate:
\begin{proposition}\label{prop:main}
Let $\bf A, B$ be two subsets of $\Lambda_n$, and let us denote by $\#{\bf B}$ the cardinality of $\bf B$. We assume that: for all $x\in\bf A$, the translated set \[\tau_x{\bf B}=\{x+y\; ; \; y \in \bf B\}\] satisfies $  \tau_x{\bf B}\subset \Lambda_n$. 
Consider $g:\Omega_n\to\bb R$ a local function whose support does not intersect $\bf B$, namely:  $\mathrm{Supp}(g) \cap \mathbf{B} = \emptyset,$ and which has mean zero with respect to $\nu^n_\rho$.
Then, there exists $C>0$ such that, for any $n \in \bb N$ and $t>0$, 
\begin{multline}
\label{eq:main-estim}
\bb E_\rho\bigg[\bigg(\int_0^t\sum_{x\in\bf A} \Big\{v(x) \tau_xg(\eta_{sn^2}) \sum_{z\in\tau_x\bf B} (\bar\eta_{sn^2}(z)-\bar\eta_{sn^2}(z+1)) \Big\}  ds \bigg)^2\bigg] \\ \leq  \frac{Ct(\#{\bf B})^2}{n}\; \|g\|_2^2 \; \|v\|_{2,n}^2.
\end{multline}
\end{proposition}

\begin{proof}[Proof of Proposition \ref{prop:main}] We prove it in Section \ref{ssec:proof}.\end{proof}

Before that, let us apply it to our purposes.

\subsection{End of the proof of Theorem \ref{theo:BG}}
 First, let us prove that we can apply Proposition \ref{prop:main} in order to estimate $\mf I_{t,n}^{\rm left}(g^\rightarrow_{\rm w})$, for ${\rm w} \in \{ \rm I,..., VI\}$.  The only estimate that has to be considered separately is the one involving $g^\rightarrow_{\rm VII}$. 
 
 We prove that the assumptions of Proposition \ref{prop:main} are satisfied for $g_{\rm III}^\rightarrow$ and $g_{\rm IV}^\rightarrow$ (see also \eqref{eq:right3bis} and \eqref{eq:right4bis}) and we let the reader to check the other ones. First, recall \eqref{eq:right3bis} and notice that
\begin{align}
\vec\eta^{\ell_k}(x)\Big[\vec\eta^{\ell_k}(x+\ell_k)-\vec\eta^{2\ell_k}(x+2\ell_k)\Big] & = \frac{\vec\eta^{\ell_k}(x)}{2\ell_k}\sum_{y=x+1}^{x+\ell_k} \sum_{z=y+\ell_k}^{y+2\ell_k-1} (\bar\eta(z)-\bar\eta(z+1)) \label{eq:decomp1} \\
& + \frac{\vec\eta^{\ell_k}(x)}{2\ell_k} \sum_{y=x+1}^{x+\ell_k} \sum_{z=y+\ell_k}^{y+3\ell_k-1} (\bar\eta(z)-\bar\eta(z+1)). \label{eq:decomp2}
\end{align}
Note that the above identity can be easily obtained by splitting each average $\vec\eta^{\ell_k}(x+\ell_k)$ and $\vec\eta^{2\ell_k}(x+2\ell_k)$ in two parts, as follows: 
\begin{align*}
\vec\eta^{\ell_k}(x+\ell_k) & = \frac{1}{2\ell_k} \sum_{y=1}^{\ell_k} \bar\eta(y+x+\ell_k) + \frac{1}{2\ell_k} \sum_{y=1}^{\ell_k} \bar\eta(y+x+\ell_k) \\
\vec\eta^{2\ell_k}(x+2\ell_k) & = \frac{1}{2\ell_k}  \sum_{y=1}^{\ell_k} \bar\eta(y+x+2\ell_k) + \frac{1}{2\ell_k} \sum_{y=1}^{\ell_k} \bar\eta(y+x+3\ell_k),
\end{align*} and subtracting the sums one by one. Let us first deal with \eqref{eq:decomp1}: we can use Proposition \ref{prop:main} with 
\[
g(\eta) = \frac{\vec\eta^{\ell_k}(0)}{2\ell_k}, \quad \text{ and } \quad {\bf B}  = \left\{ (y,z) \in \Lambda_n \; ; \; \left\{ \begin{aligned} 1 &\leq y \leq \ell_k \\  y+\ell_k &\leq z \leq y+2\ell_k-1 \end{aligned}\right.\right\}.
\]
Note that $\|g \|_2^2 = C/\ell_k^3$ and $\#{\bf B}=\ell_k^2$, and remember that $\ell_k=2^k\ell_0$.
Next, we deal with \eqref{eq:decomp2}: we only need to change ${\bf B}$ which now reads 
\[
{\bf B}  = \left\{ (y,z) \in \Lambda_n \; ; \; \left\{ \begin{aligned} 1 &\leq y \leq \ell_k \\  y+\ell_k &\leq z \leq y+3\ell_k-1 \end{aligned}\right.\right\}, \quad \mathrm{hence} \quad \#{\bf B}=2\ell_k^2.
\]
  Therefore, one can  see from \eqref{eq:right3bis} and Proposition \ref{prop:main} that
\[
\mf I_{t,n}^{\rm left}(g^\rightarrow_{\rm III}) \leq C\bigg(\sum_{k=0}^{M-1} \bigg\{ \frac{t\ell_k^4}{n}\frac{1}{\ell_k^2}\|v\|_{2,n}^2\bigg\}^{1/2}\bigg)^2 \leq Ct \frac{\ell_{M}}{n}\|v\|_{2,n}^2.
\] Let us treat similarly $g_{\rm IV}^\rightarrow$: recall \eqref{eq:right4bis} and write
\begin{equation}
\vec\eta^{2\ell_k}(x+2\ell_k)\Big[\vec\eta^{\ell_k}(x)-\vec\eta^{2\ell_k}(x)\Big]  = \frac{\vec\eta^{2\ell_k}(x+2\ell_k)}{2\ell_k}\sum_{y=x+1}^{x+\ell_k} \sum_{z=y}^{y+\ell_k-1} (\bar\eta(z)-\bar\eta(z+1)). \label{eq:decomp3} 
\end{equation}
Here we choose 
\[ g(\eta)  = \frac{\vec\eta^{2\ell_k}(2\ell_k)}{2\ell_k}  \quad \text{ and } \quad {\bf B}  =\left\{ (y,z) \in \Lambda_n \; ; \; \left\{ \begin{aligned} 1 &\leq y \leq \ell_k \\  y &\leq z \leq y+\ell_k-1 \end{aligned}\right.\right\}
\]
and then apply Proposition \ref{prop:main}, which gives the same bound as before:
\[
\mf I_{t,n}^{\rm left}(g^\rightarrow_{\rm IV}) \leq Ct \frac{\ell_{M}}{n}\|v\|_{2,n}^2.
\] 
Performing similar arguments and using Proposition \ref{prop:main} together with Minkowski's inequality, we get that, for any $\ell,n\in\bb N$ such that $\ell <\frac n 4$, and any $t>0$,
\begin{align*}
&\mf I_{t,n}^{\rm left}(g^\rightarrow_{\rm I})  \leq Ct \frac{\ell_0^2}{n}\|v\|_{2,n}^2,  \quad \quad 
\mf I_{t,n}^{\rm left}(g^\rightarrow_{\rm II})  \leq Ct \frac{\ell_0}{n}\|v\|_{2,n}^2, \\
&\mf I_{t,n}^{\rm left}(g^\rightarrow_{\rm w})  \leq Ct \frac{\ell}{n}\|v\|_{2,n}^2, \qquad \text{for any } {\rm w} \in \{\rm III, IV, V, VI\} .
\end{align*}
Finally, we have  to estimate the last remaining term involving $g_{\rm VII}^\rightarrow$,
 which is treated separately. More precisely, in Section \ref{ssec:separate} we will prove the following:
\begin{proposition}\label{prop:last}
For any $\ell,n\in\bb N$ such that $\ell <\frac n 4$, and any $t>0$,
\begin{multline}
\label{eq:right7bis}
\mf I_{t,n}^{\rm left}(g^\rightarrow_{\rm VII})=\bb E_\rho\bigg[\bigg(\int_0^t \sum_{x=1}^{n-2\ell-1} v(x)\Big\{\vec\eta_{sn^2}^\ell(x) \Big[\bar\eta_{sn^2}(x)-\vec\eta_{sn^2}^\ell(x)\Big] + \frac{\chi(\rho)}{\ell}\Big\} ds\bigg)^2\bigg]
\\ \leq Ct \bigg(\frac{\ell}{n}+\frac{tn}{\ell^2}\bigg)\|v\|_{2,n}^2.
\end{multline}
\end{proposition} 
Putting together the previous estimates into our decomposition \eqref{eq:right1}--\eqref{eq:right7} of $g_\ell^\rightarrow$, we obtain  straightforwardly the final bound \eqref{eq:Ileft}.
We let the reader  repeat all the arguments above to obtain the second part, namely \eqref{eq:Iright}. Theorem \ref{theo:BG} then easily follows. 

The next two sections are devoted to the proofs of Proposition \ref{prop:main} and Proposition \ref{prop:last}.

\subsection{Proof of Proposition \ref{prop:main}} 
\label{ssec:proof}
Take $\bf A, B$ two subsets of $\Lambda_n$ such that, for all $x\in\bf A$, $\tau_x{\bf B} \subset \Lambda_n$, and take $g:\Omega_n\to\bb R$ a mean zero function with respect to $\nu^n_\rho$ such that $\mathrm{Supp}(g)\cap {\bf B} = \emptyset$. 
From  \cite[Lemma 2.4]{KLO},  we bound the left hand side of \eqref{eq:main-estim} by
\begin{equation*}
Ct\sup_{f\in L^2(\nu^n_\rho)} \bigg\{ 2 \int f(\eta) \Big(\sum_{x\in\bf A} v(x) \tau_x g(\eta) \sum_{z \in \tau_x \bf B} (\bar\eta(z)-\bar\eta(z+1)) \nu^n_\rho(d\eta) - n^2\mathfrak D_n(f)\bigg\},
\label{eq:vari}
\end{equation*}
where $\mathfrak D_n$ is the Dirichlet form introduced in \eqref{eq:dir_form}.
We write the previous expectation as twice its half, and in one of the integrals we make the change of variables $\eta \mapsto \sigma^{z,z+1}\eta$ to rewrite it as 
\begin{align*}
\label{eq:sumint}
  \sum_{x\in\bf A} \sum_{z\in\tau_x \bf B}v(x) \int \Big\{f(\eta) \tau_xg(\eta)- f(\sigma^{z,z+1}\eta) \tau_xg(\sigma^{z,z+1}\eta)\Big\} (\bar\eta(z)-\bar\eta(z+1)) \Big\}\; \nu^n_\rho(d\eta)
\end{align*}
With our assumption, we have $\mathrm{Supp}(\tau_xg) \cap \tau_x {\bf B} = \emptyset$ for every $x \in \bf A$. Therefore: 
$ \tau_xg(\sigma^{z,z+1}\eta)=\tau_x g(\eta),$ for all $z\in\tau_x{\bf B}$, and as a consequence
 the last expression  equals   
\begin{equation*}
\sum_{x\in\bf A} \sum_{z \in \tau_x \bf B}v(x) \int \Big\{ \tau_x g(\eta) \big(\bar\eta(z)-\bar\eta(z+1)\big) \big(f(\eta)-f(\sigma^{z,z+1}\eta)\big) \Big\}\; \nu^n_\rho(d\eta).
\end{equation*}
For any $x\in\bf A$ and $z\in\tau_x\bf B$, we use Young's inequality with $\varepsilon_x >0$, and we bound the previous expression from above by 
 \begin{align}
& \sum_{x\in\bf A} \sum_{z\in\tau_x\bf B}\frac{\varepsilon_x}{2} \big(v(x)\big)^2 \int \Big\{\tau_x g(\eta)\big(\bar\eta(z)-\bar\eta(z+1)\big)\Big\}^2 \nu^n_\rho(d\eta) \label{eq:young1}\\
 &+ \sum_{x\in\bf A}\sum_{z\in\tau_x\bf B} \frac{1}{2\varepsilon_x} \int \Big\{f(\eta)-f(\sigma^{z,z+1}\eta)\Big\}^2 \nu^n_\rho(d\eta). \label{eq:young2}
 \end{align}
Now, since $\nu^n_\rho$ is invariant under translations  and since ${\bf A}\subset \Lambda_n$, it is easy to see that 
\[
\frac{1}{\#\bf B}\sum_{x\in\bf A}\sum_{z\in\tau_x\bf B} \int \Big\{f(\eta)-f(\sigma^{z,z+1}\eta)\Big\}^2 \nu^n_\rho(d\eta) \leq \mathfrak D_n(f).
\]
As a result, if we chose $2\varepsilon_x = \#{\bf B}/ n^2$, we have that \eqref{eq:young2} is bounded by $n^2 \mathfrak D_n(f)$, and \eqref{eq:young1} is bounded by 
\[
\sum_{x\in\bf A} \sum_{z\in\tau_x\bf B} \frac{\#\bf B}{4n^2}\big(v(x)\big)^2 \|g\|_2^2 \leq \frac{(\#{\bf B})^2}{4n} \|v\|_{2,n}^2 \|g\|_2^2.
\]
This ends the proof.

\subsection{Proof of Proposition \ref{prop:last}} \label{ssec:separate}

We have
\[
\mf I_{t,n}^{\rm left}(g_{\rm VII}^\rightarrow) \leq 2\mf I_{t,n}^{\rm left}(g_1) + 2 \mf I_{t,n}^{\rm left}(g_2),
\]
where we define
\begin{align*}
g_1(\eta)&=\vec\eta^\ell(0) \Big[\bar\eta(0)-\vec\eta^\ell(0)\Big] + \frac{1}{2\ell}(\eta(0)-\eta(1)),
\\ g_2(\eta)  &= \frac{\chi(\rho)}{\ell}-\frac{1}{2\ell}(\eta(0)-\eta(1)).
\end{align*}
Let us start with the easiest term to estimate, namely $\mf I_{t,n}^{\rm left}(g_2)$.
From the Cauchy-Schwarz inequality, together with the independence of $\eta(x)$ and $\eta(y)$ under the invariant measure $\nu^n_\rho$ (as soon as $x\neq y$), one can easily show that
\[
\mf I_{t,n}^{\rm left}(g_2) \leq Ct^2\frac{n}{\ell^2} \|v\|_{2,n}^2.
\]
Now let us look at the term with $g_1$. From \cite[Lemma 2.4]{KLO}, we bound $\mf I_{t,n}^{\rm left}(g_1)$ by 
\begin{multline}\label{eq:multline}
\sup_{f\in L^2(\nu^n_\rho)} \bigg\{2\int \sum_{x=1}^{n-2\ell-1}v(x)\Big\{\vec{\eta}^{\ell}(x)\Big[\bar\eta(x)-\vec{\eta}^{\ell}(x)\Big]+\frac{1}{2\ell}
\big(\bar{\eta}(x)-\bar{\eta}(x+1)\big)\Big\}f(\eta)\nu^n_\rho(d\eta)\\ - n^2 \mathfrak D_n(f)\bigg\}.
\end{multline}
 We note that
\begin{align}
&\int \sum_{x=1}^{n-2\ell-1}v(x)\vec{\eta}^\ell(x)\big\{\bar{\eta}(x)-\vec{\eta}^\ell(x)\big\}f(\eta)\nu^n_\rho(d\eta)\label{eq:firstpart}\\
& =\int \sum_{x=1}^{n-2\ell-1}v(x)\vec{\eta}^\ell(x)\big\{\bar{\eta}(x)-\bar{\eta}(x+1)\big\}f(\eta)\nu^n_\rho(d\eta)\label{eq:ff}\\
&\qquad +\int \sum_{x=1}^{n-2\ell-1}v(x)\vec{\eta}^\ell(x)\frac{\ell-1}{\ell}\big\{\bar{\eta}(x+1)-\bar{\eta}(x+2)\big\}f(\eta)\nu^n_\rho(d\eta)\notag\\
&\qquad +\cdots+ \int \sum_{x=1}^{n-2\ell-1}v(x)\vec{\eta}^\ell(x)\frac{1}{\ell}\big\{\bar{\eta}(x+\ell-1)-\bar{\eta}(x+\ell)\big\}f(\eta)\nu^n_\rho(d\eta).\notag
\end{align}
For each term of the last sum above, we do the following procedure: we write it as twice its half, and in one of the integrals we make the change of variables  $\eta$ to $\sigma^{z,z+1}\eta$ (for some suitable $z$), for which the measure $\nu^n_\rho$ is invariant. After doing this, one can check that the last expression equals 
\begin{align}
&\int \sum_{x=1}^{n-2\ell-1}v(x)\vec{\eta}^\ell(x)\big\{\bar{\eta}(x)-\bar{\eta}(x+1)\big\}\big(f(\eta)-f(\sigma^{x,x+1}\eta)\big)\nu^n_\rho(d\eta) \label{eq:lem6101}\\
+&\int \sum_{x=1}^{n-2\ell-1}v(x)\vec{\eta}^\ell(x)\frac{\ell-1}{\ell}\big\{\bar{\eta}(x+1)-\bar{\eta}(x+2)\big\}\big(f(\eta)-f(\sigma^{x+1,x+2}\eta)\big)\nu^n_\rho(d\eta) \notag\\
+&\cdots+ \notag \\ 
+& \int \sum_{x=1}^{n-2\ell-1}v(x)\vec{\eta}^\ell(x)\frac{1}{\ell}\big\{\bar{\eta}(x+\ell-1)-\bar{\eta}(x+\ell)\big\}\big(f(\eta)-f(\sigma^{x+\ell-1,x+\ell}\eta)\big)\nu^n_\rho(d\eta)\notag\\
+&\int\sum_{x=1}^{n-2\ell-1}v(x)\frac{\bar{\eta}(x+1)-\bar{\eta}(x)}{\ell}\big\{\bar{\eta}(x)-\bar{\eta}(x+1)\big\}f(\eta)\nu^n_\rho(d\eta). \label{eq:lem6102}
\end{align}
Note that the last term  \eqref{eq:lem6102} comes from the change of variables $\eta$ to $\sigma^{x,x+1}\eta$ in the first term \eqref{eq:ff} above, as well as \eqref{eq:lem6101}.
The whole sum can be rewritten as
\begin{align}
&\int \sum_{x=1}^{n-2\ell-1}v(x)\vec{\eta}^\ell(x)\frac{1}{\ell}\sum_{y=x+1}^{x+\ell}\sum_{z=x}^{y-1}\big\{\bar{\eta}(z)-\bar{\eta}(z+1)\big\}\big\{f(\eta)-f(\sigma^{z,z+1}\eta)\big\}\nu^n_\rho(d\eta)\label{eq:boundlem}\\
-&\int\sum_{x=1}^{n-2\ell-1}v(x)\frac{1}{\ell}(\bar{\eta}(x)-\bar{\eta}(x+1))^2f(\eta)\nu^n_\rho(d\eta).\label{eq:secondpart}
\end{align}
The integral  in \eqref{eq:multline} is exactly equal to the sum of \eqref{eq:firstpart} and \eqref{eq:secondpart}, therefore it is bounded by the first term in the previous expression, namely \eqref{eq:boundlem}.
Now, we use the same arguments as in the proof of Proposition \ref{prop:main}, namely, Young's inequality with $2\varepsilon_x=\ell/n^2 $ and we bound it by
\begin{multline*}
 C(\rho)\frac{\ell}{ n^2} \sum_{x=1}^{n-2\ell-1}v^2(x)+\frac{n^2}{{\ell^2}} \sum_{x=1}^{n-2\ell-1}\sum_{y=x+1}^{x+\ell}\sum_{z=x}^{y-1}\int \big\{f(\eta)-f(\sigma^{z,z+1}\eta)\big\}^2\; \nu^n_\rho(d\eta)\\
 \leq C(\rho) \frac{\ell}{n} \|v\|_{2,n}^2 + n^2 \mathfrak D_n(f),
\end{multline*}
which proves that
\[
\mf I_{t,n}^{\rm left}(g_1) \leq Ct\frac{\ell}{n} \|v\|_{2,n}^2, 
\]
so that the proof ends. \qed

\section{Uniqueness of energy solutions}\label{sec:energy-pr}

In this section we give all the details for the proof of Theorem~\ref{def:energy} and we show how the same arguments also apply to the proof of Theorem \ref{def:energy-KPZ}.

Recall that we are interested in energy solutions to
\begin{equation}\label{eq:SBE-gen}
d\mathcal Y_t = A \Delta_{\rm Dir}\; \mathcal Y_t dt + \bar{E}\;\nabla_{\rm Dir}\;  \big(\mathcal{Y}_t^2 \big) \;dt + \sqrt{D}\;  \nabla_{\rm Dir}\; d\mc W_t,
\end{equation}
where $A, D>0$ and $\bar{E} \in \RR$. But as in  \cite[Remark~2.6]{Gubinelli2015Energy} we can show that $\{\mathcal Y_t \; ; \; t \in [0,T] \}$ is an energy solution to~\eqref{eq:SBE-gen} if and only if $\big\{\sqrt{2A/D} \;  \mathcal Y_{t/A} \; ;\; t \in [0, A T]\big\}$ solves~\eqref{eq:SBE-gen} with $A \mapsto 1$, $D \mapsto 2$, and $ \bar{E} \mapsto \bar{E} \sqrt{D/(2A^3)}$. 

\medskip

So from now on we assume without loss of generality that $A = 1$ and $D=2$, and to simplify notation we write $E$ instead of $\bar{E}$, and we show in this section that the equation
\begin{equation}\label{eq:SBE-gen2}
d\mathcal Y_t =  \Delta_{\rm Dir}\; \mathcal Y_t dt + E\;\nabla_{\rm Dir}\;  \big(\mathcal{Y}_t^2 \big) \;dt + \sqrt{2}\;  \nabla_{\rm Dir}\; d\mc W_t,
\end{equation}
has a unique solution for any $E\in\bb R$, where the notion of Dirichlet boundary conditions has been properly defined in Theorem \ref{def:energy}.  

 To prove the uniqueness of energy solutions we will use the exact same strategy as in~\cite{Gubinelli2015Energy} and we will sometimes refer to that paper for additional details. The main idea consists in: first, mollifying an energy solution, then mapping the mollified process through the Cole-Hopf transform to a new process, and then taking the mollification away in order to show that the transformed process solves in the limit the linear multiplicative stochastic heat equation with certain boundary conditions. However, even using the strategy of \cite{Gubinelli2015Energy}, we have to redo all computations because our setting is somewhat different, and the boundary condition of the stochastic heat equation actually changes as we pass to the limit: in particular, it is not equal to the one we would naively guess.

\begin{remark}
Let us notice that the definition given in \cite{Gubinelli2015Energy} is not exactly the same as the one we adopted in Theorem \ref{def:energy}, but it is not difficult to check that they are indeed equivalent (see \cite[Proposition 4]{GJSi}), and therefore the same strategy can be implemented.
\end{remark}

This section is split as follows: in Section \ref{ssec:preli} we give some tools that will be used in Section \ref{ssec:construct} in order to show that in the definition the Burgers non-linearity (namely the process $\mc A$ of Theorem \ref{def:energy}) we can replace $\iota_\varepsilon$ by different approximations of the identity. We conclude in Section \ref{ssec:map} with the proof of uniqueness for the energy solution $\mc Y$, by developing the strategy explained above. In the following we always denote by $\mu$ the law of the standard white noise on $\SD'$. If $f\colon\bb R^n\to\bb R$ is multidimensional, then we denote by $\partial_\alpha f$ its derivative of order $\alpha \in \bb N_0^n$.

\subsection{Preliminaries} \label{ssec:preli}
In this section we give two ways to handle functionals written in the form $\int_0^\cdot F(\mc Y_s)ds$, where $\mc Y$ shall be the energy solution to \eqref{eq:SBE-gen2} as defined in Theorem \ref{def:energy}, and $F$ belongs to some general class of functions.

\subsubsection{It{\^o} trick and Kipnis-Varadhan Lemma}

We write $\mathcal{C}$ for the space of cylinder functions $F\colon \SD' \rightarrow \mathbb{R}$, which are such that there exist $d \in \bb N$ and $\varphi_i \in \SD \; (i=1,\dots,d)$ with  \[F (\mc{Y})= f (\mc{Y} (\varphi_1),
\ldots, \mc{Y} (\varphi_d)),\] 
where $f \in \mc C^2 (\mathbb{R}^d)$ has polynomial growth of all partial
derivatives up to order $2$. For $F \in \mathcal{C}$ we define the
operator $L_0$ as
\begin{align*}
    L_0 F (\mc Y) &= \sum_{i = 1}^d \partial_i f (\mc Y (\varphi_1), \ldots, \mc Y
   (\varphi_n)) \mc Y (\Delta \varphi_i) \\
   &\quad + \sum_{i, j = 1}^d \partial_{i j} f (\mc Y
   (\varphi_1), \ldots, \mc Y (\varphi_n)) \langle \nabla \varphi_i, \nabla
   \varphi_j \rangle_{L^2 ([0, 1])},
\end{align*}
and its domain $\tmop{Dom} (L_0)$ is defined as the closure in $L^2(\mu)$ of $\mathcal{C}$ with respect
to the norm \[\| F \|_{L^2 (\mu)} + \| L_0 F \|_{L^2 (\mu)}.\] 
First, let us take $\tilde{\mc Y}$ as the Ornstein-Uhlenbeck process with Dirichlet boundary conditions as defined in Proposition \ref{prop:OU} (with $A=1$, $D=2$) -- or, equivalently, an energy solution of \eqref{eq:SBE-gen2} with $E=0$ (as defined in Theorem \ref{def:energy}).
 Then, we have for every $F \in \mathcal{C}$,
\[ F (\tilde{\mc Y}_t) = F (\tilde{\mc Y}_0) + \int_0^t L_0 F (\tilde{\mc Y}_s)
   d s + M^F_t, \]
where $M^F$ is a continuous martingale with quadratic variation
\[ \langle M^F \rangle_t = \int_0^t 2 \big\| \nabla_u \mathrm{D}_u F (\tilde{\mc Y}_s)
   \big\|_{{L^2_u([0,1])}}^2 \; d s, \]
and where $\nabla_u$ is the usual derivative w.r.t.~$u$ and $\mathrm{D}_u F$ denotes the Malliavin derivative defined in terms of the
law of the white noise, i.e.
\[ \mathrm{D}_u F (\mc Y)= \sum_{i = 1}^d \partial_i f (\mc Y (\varphi_1), \ldots, \mc Y
   (\varphi_n)) \varphi_i (u), \qquad \text{for any } u \in [0, 1]. \]
     In the following, for any $\mc Y \in \SD'$ we denote
\begin{equation}\label{eq:E-def}
   \mathcal{E} (F) (\mc Y)= 2 \big\| \nabla_u \mathrm{D}_u F (\mc Y)
   \big\|_{{L^2_u([0,1])}}^2.
\end{equation}
Now, let $\mc Y$ be an energy solution to \eqref{eq:SBE-gen2} the stochastic
Burgers equation with Dirichlet boundary conditions, as defined in Theorem \ref{def:energy}.  Recall that $\mc Y_0$ is a $\SD'$--valued white noise, hence has law $\mu$.

Since \eqref{eq:lipschitz} implies that $\mc A$ has zero quadratic variation -- see  \cite[Proposition 4]{GJSi} for a proof -- the It{\^o} trick for additive functionals of the form
$\int_0^{\cdot} L_0 F (\mc Y_s) d s$ follows by the same
arguments as in \cite[Proposition~3.2]{Gubinelli2015Energy}. Thus, we
can prove that  for all $F \in \mathcal{C}$ and $p\geqslant 1$
\begin{equation}
  \label{eq:ito-bound-cylinder} \mathbb{E} \bigg[ \sup_{t \in [0, T]} \bigg|
  \int_0^t L_0 F (\mc Y_s) d s \bigg|^p \bigg] \lesssim T^{p / 2} \;
  \mathbb{E} \big[\mathcal{E} (F) (\mc Y_0)^{p / 2}\big],
\end{equation}
and for $p = 2$ we get in particular
\begin{equation}
  \label{eq:ito-bound-cylinder-1-norm} \mathbb{E} \bigg[ \sup_{t  \in [0, T]}
  \bigg| \int_0^t L_0 F (\mc Y_s) d s \bigg|^2 \bigg]\lesssim T \mathbb{E} \big[\mathcal{E} (F) (\mc Y_0)\big].
\end{equation}
For the sake of clarity, let us define from now on: 
\[\| F
  \|_{1,0}^2 = \mathbb{E} \big[\mathcal{E} (F) (\mc Y_0)\big]=2\mathbb{E} \big[F
  (\mc Y_0) (- L_0) F (\mc Y_0)\big],\]
where the second equality follows from the Gaussian integration by parts rule, see~\cite[Lemma~1.2.1]{Nualart2006}.

From this, let us now define two Hilbert spaces which will be useful in controlling additive functionals of $\mc Y$.

\begin{definition}
  Let us introduce an equivalence relation on $\mathcal{C}$ by identifying $F$
  and $G$ if $\|F - G\|_{1,0} = 0$, 
  so that $\| \cdot \|_{1,0}$ becomes a norm on
  the equivalence classes. We write $\mathfrak{H}^1_0$ for the completion of
  the equivalence classes with respect to $\| \cdot \|_{1,0}$.
  
  For $F \in \mc C$ we define
  \[ \| F \|_{- 1,0}^2 = \sup_{G \in \mathcal{C}} \big\{ 2\mathbb{E} [F (\mc Y_0) G
     (\mc Y_0)] - \| G \|_{1,0}^2 \big\} = \sup_{\substack{G \in \mathcal{C}\\ \|G \|_{1,0} = 1}} \mathbb{E} [F (\mc Y_0) G
     (\mc Y_0)] \]
  and we identify $F$ and $G$ if $\| F - G \|_{- 1,0} = 0$ and $\| F \|_{- 1,0} <
  \infty$. We write $\mathfrak{H}_0^{- 1}$ for the completion of the
  equivalence classes with respect to $\| \cdot \|_{- 1,0}$.
\end{definition}
\begin{remark} It is possible to show that $\| F\|_{1,0} = 0$ if and only if $F$ is constant, but we do not need this. \end{remark}

Let us now extend the It\^o trick to the entire domain of $L_0$:
\begin{lemma}[It{\^o} trick] \label{lem:ito} Let $F \in \tmop{Dom} (L_0)$. Then $F \in \mathfrak{H}^1_0$
  and 
  \begin{equation}
    \label{eq:ito-bound-general} \mathbb{E} \bigg[ \sup_{t \in [0,T]}
    \bigg| \int_0^t L_0 F (\mc Y_s) d s \bigg|^2 \bigg] \lesssim T \| F
    \|_{1,0}^2 .
  \end{equation}
  If also $\mathcal{E} (F) \in L^{p/2}(\mu)$ for $p\geqslant 1$, then
  \begin{equation}
    \label{eq:ito-bound-general-p} \mathbb{E} \bigg[ \sup_{t \in [0,T]}
    \bigg| \int_0^t L_0 F (\mc Y_s) d s \bigg|^p \bigg] \lesssim T^{p / 2}
    \mathbb{E} \big[\mathcal{E} (F) (\mc Y_0)^{p / 2}\big] .
  \end{equation}
\end{lemma}

\begin{proof}[Proof of Lemma \ref{lem:ito}]
  Since $F \in \tmop{Dom} (L_0)$, there exists a sequence of cylinder functions
  $\{F_n\}_{n\in\mathbb N}$ in $\mc C$ such that $\| F_n - F \|_{L^2} + \| L_0 (F_n - F) \|_{L^2}$
  converges to zero. But then also
  \[ \lim_{n \rightarrow \infty} \| F_n - F \|_{1,0}^2 = \lim_{n \rightarrow
     \infty} \Big\{ - 2\mathbb{E} \big[(F_n - F)(\mc Y_0) L_0 (F_n - F)(\mc Y_0)\big]\Big\} = 0, \]
  so $F \in \mathfrak{H}_0^1$ and~ \eqref{eq:ito-bound-general} follows
  from~\eqref{eq:ito-bound-cylinder-1-norm}. The second estimate follows by
  approximation from~\eqref{eq:ito-bound-cylinder}.
\end{proof}

\begin{remark}
  \label{rmk:chaos-moments}If moreover $F (\mc Y_0)$ has a finite chaos expansion of length
  $d$ (see Subsection~\ref{ssec:gaussian} below for the definition of the chaos expansion), then $\mathcal{E} (F) (\mc Y_0)$ has a chaos expansion of length $2 \times
  (d - 1)$, and therefore all its moments are comparable and we can estimate $\mathbb{E} \big[\mathcal{E} (F) (\mc Y_0)^{p / 2}\big] \simeq \mathbb{E}
     \big[\mathcal{E} (F) (\mc Y_0)\big]^{p / 2} = \| F \|_{1,0}^p$.
\end{remark}

\begin{corollary}[Kipnis-Varadhan inequality]
  \label{cor:KV} Let $F \in L^2 (\mu) \cap \mathfrak{H}^{-
  1}_0$. Then
  \[ \mathbb{E} \bigg[ \bigg\| \int_0^{\cdot} F (\mc Y_s) d s \bigg\|_{p
     - \tmop{var} ; [0, T]}^2 \bigg] +\mathbb{E} \bigg[ \sup_{t \in [0,T]}
     \bigg| \int_0^t F (\mc Y_s) d s \bigg|^2 \bigg]  \lesssim T \| F \|_{-
     1,0}^2,  \]
     where the $p$-variation norm $\|f\|_{p-\tmop{var}}$ of a function $f:[a,b]\to \bb R$ is defined as
        \begin{equation}\label{eq:pvar}
           \|f \|_{p - \tmop{var} ; [a, b]}^p = \sup\bigg\{ \sum_{k=0}^{d-1} |f(t_{k+1}) - f(t_k)|^p\; ; \; d \in \bb N, a = t_0 < t_1 < \dots < t_d = b\bigg\}.
        \end{equation}
\end{corollary}

\begin{proof}[Proof of Corollary \ref{cor:KV}]
  The usual proof by duality works. The statement about the $p$-variation is
  shown as in ~\cite[Corollary~3]{Gubinelli2015Energy}.
\end{proof}

As a result, with Lemma \ref{lem:ito} and Corollary \ref{cor:KV} we are provided with two important tools, which allow us to control in some sense $\int_0^\cdot F(\mc Y_s)ds$. Note that Lemma \ref{lem:ito} is convenient only if one is able to write $F$ as $L_0G$, which may not be easy. If one cannot solve the Poisson equation $F=L_0G$, then one relies on the variational norm of $F$ given by Corollary \ref{cor:KV}. 

The next paragraph is devoted to constructing solutions to the Poisson equation $L_0G=F$ using the Gaussian structure of $L^2(\mu)$, which is now standard and fully detailed in \cite[Section 3.2]{Gubinelli2015Energy}.

\subsubsection{Gaussian analysis}\label{ssec:gaussian}

In the following we develop some Gaussian analysis that is helpful for estimating the $\| \cdot \|_{1,0}$ and $\|\cdot \|_{-1,0}$ norms from above. We refer the reader to~\cite{Nualart2006, Janson1997} for details on Malliavin calculus and chaos decompositions.

Let $\mc Y$ be a white noise on $\SD'$ and write $\sigma(\mc Y)$ for the sigma algebra generated by $\mc Y$. Then we can define a \emph{chaos expansion} in $L^2(\sigma(\mc Y))$ as follows: for $d \in \bb N_0$ and $f \in L^2([0,1]^d)$ we write $W_d(f)(\mc Y)$ for the $d$--th order Wiener It\^o-integral of $f$ against $\mc Y$,
\[
   W_d(f)(\mc Y) = \int_{[0,1]^d} f(y_1, \dots, y_d) \mc Y(y_1) \cdots \mc Y(y_d) \, dy_1 \cdots dy_d,
\]
see~\cite[Section~1.1.2]{Nualart2006} for the construction; occasionally we will simply write $W_d(f)$ instead of $W_d(f)(\mc Y)$. Recall that $W_d(f) = W_d(\widetilde f)$ for $\widetilde{f} (u) = \frac{1}{d!} \sum_{\sigma \in \mathfrak{S}_d} f
  (u_{\sigma (1)}, \ldots, u_{\sigma (d)})$, where $\mathfrak{S}_d$ is the symmetric group on $\{1,\dots,d\}$. 
\medskip  
  
  The \emph{chaos expansion} of $F \in L^2(\sigma(\mc Y))$  is then given by
\[
   F = \sum_{d=0}^\infty W_d(f_d),
\]
where $\{f_d \in L^2([0,1]^d)\}_{d \in \bb N_0}$, are uniquely determined deterministic functions.

\medskip

Let us denote by $L^2_{\rm Sym}([0,1]^d)$ the space of symmetric $L^2$--functions. Since the $d$--th homogeneous chaos $\{W_d(f_d) \; ; \; f_d \in L^2_{\rm Sym}([0,1]^d)\}$ is also the closure of the span of all random variables of the form $\mc Y \mapsto H_d(\mc Y(\varphi))$, where $H_d$ is the $d$--th Hermite polynomial and $\varphi \in \SD$ with $\|\varphi\|_{L^2([0,1])}=1$, it will be convenient to write down the action of $L_0$ onto these random variables. Indeed, it is well-known that \[H_d(\mc Y(\varphi)) = W_d(\varphi^{\otimes d})(\mc Y)\] for all $\varphi \in L^2([0,1])$ with $\|\varphi\|_{L^2([0,1])} = 1$. For all $\varphi \in
\Ctwo$ we have
\begin{align}
\notag  L_0 H_d (\mc Y(\varphi)) & = H'_d (\mc Y(\varphi)) \mc Y(\Delta \varphi) +
  H''_d (\mc Y(\varphi)) \langle \nabla \varphi, \nabla \varphi \rangle_{L^2 ([0,
  1])}\\
  & = H'_d (\mc Y(\varphi)) \mc Y(\Delta \varphi) - H''_d (\mc Y(\varphi))
  \langle \varphi, \Delta \varphi \rangle_{L^2 ([0, 1])}. \label{eq:L0H}
\end{align}
From here let us define

\begin{align*}
	\mc \SD ([0, 1]^d) = \Big\{ g \in \mc C^\infty ([0, 1]^d) \; ; \; \partial_\alpha g(u) = 0, \; &\forall\; u \in \partial ([0, 1]^d),
	\\ &\forall\; \alpha=(2k_1,\dots, 2k_d) \in \bb N_0^d  \Big\}.
\end{align*}
Then, from \eqref{eq:L0H}, the same arguments as in \cite[Lemma~3.7]{Gubinelli2015Energy} show that for all symmetric functions
$ f_d \in \SD ([0, 1]^d) $
we have
\begin{equation}\label{eq:L0W} L_0 W_d (f_d) = W_d (\Delta f_d), \end{equation}
with $\Delta := \sum_{k = 1}^d \partial_{kk}$. Therefore, the operator $L_0$ leaves the $d$--th chaos invariant, and this will be useful to solve the Poisson equation. It only remains to compute its norm $\|W_d(f_d)\|_{0,1}$, which is the goal of the remainder of this section.
\medskip

To do so, let us introduce  other notations: similary to our definitions in Section \ref{ssec:test_functions}, we denote by 
$\mc H_{\rm Dir}^1
    ([0, 1]^d)$ the completion of $ \SD ([0, 1]^d)$ with respect to the norm 
    \[
    	\|g\|_{\mc H_{\rm Dir}^1
    ([0, 1]^d)} = \| \nabla g\|_{L^2([0,1])^d} = \bigg(\sum_{i=1}^d \|\partial_{i} g\|_{L^2([0,1])^d}^2\bigg)^{1/2}
  	\]
    and $\mc H_{\rm Dir,Sym}^1([0, 1]^d)$ denotes the completion of $\SD ([0, 1]^d)$
with respect to the norm 
\[\| g \|_{\mc H_{\rm Dir,Sym}^1([0, 1]^d)}= \| \widetilde{g} \|_{\mc H_{\rm Dir}^1([0, 1]^d)} \qquad \text{ for } \widetilde{g} (u) = \frac{1}{d!} \sum_{\sigma \in \mathfrak{S}_d} g
  (u_{\sigma (1)}, \ldots, u_{\sigma (d)}),\] and therefore we identify
non-symmetric functions $g$ in $\SD ([0, 1]^d)$ with their symmetrizations $\widetilde g$. 
From the Poincar{\'e} inequality we obtain that $\| g \|_{L^2 ([0,
1]^d)} \lesssim \| g \|_{\mc H_{\rm Dir}^1([0, 1]^d)}$ for all $g \in \SD ([0, 1]^d)$,
and therefore $\mc H_{\rm Dir,Sym}^1([0, 1]^d)$ is contained in $L^2_{\rm Sym} ([0, 1]^d)$. Note also that the symmetric $\mc S_{\rm Dir}([0,1]^d)$--functions are dense in $\mc H_{\rm Dir,Sym}^1([0, 1]^d)$.

From  \eqref{eq:L0W}, in the same way as in \cite[Lemma 3.13]{Gubinelli2015Energy},
we obtain:

\begin{lemma}\label{lem:1-norm chaos}
For symmetric $f_d \in \SD ([0, 1]^d)$ we have
  \begin{align*}
    \big\| W_d (f_d) \big\|_{1,0}^2 & = - 2\mathbb{E} \big[W_d (f_d) W_d (\Delta f_d)\big] = - 2 (d!) \big\langle f_d, \Delta f_d \big\rangle_{L^2([0,1]^d)}^2\\
    & = 2 (d!) \big\|  \nabla f_d  \big\|_{L^2([0,1]^d)}^2 = 2 (d!) \big\| f_d \big\|_{\mc H_{\rm Dir}^1
    ([0, 1]^d)}^2.
  \end{align*}
  Moreover, for general (not
  necessarily symmetric) $f_d \in \SD ([0, 1]^d)$ we obtain the inequality
  \[ \big\| W_d (f_d)\big\|_{1,0}^2 = 2(d!) \big\| f_d \big\|_{\mc H_{\rm Dir, Sym}^1
    ([0, 1]^d)}^2  \leq
     2 (d!) \big\| f_d \big\|_{\mc H_{\rm Dir}^1
    ([0, 1]^d)}^2. \]
\end{lemma}

\begin{proof}[Proof of Lemma \ref{lem:1-norm chaos}]
  Integration by parts works without boundary terms because $f_d \in \SD ([0, 1]^d)$. The bound $\| f_d \|_{\mc H_{\rm Dir, Sym}^1
    ([0, 1]^d)}^2  \leq \| f_d \|_{\mc H_{\rm Dir}^1
    ([0, 1]^d)}^2$ follows from the Cauchy-Schwarz inequality.
\end{proof}

 We also write $\mc H_{\rm Dir, Sym}^{-1}
    ([0, 1]^d)$
for the completion of $L^2_{\rm Sym} ([0, 1]^d)$ with respect to the norm
\begin{equation} \label{eq:H0s-var-def}
	\| f \|_{\mc H_{\rm Dir, Sym}^{-1}
    ([0, 1]^d)} = \sup_{\substack{
     g \in \mc H_{\rm Dir, Sym}^{1}
    ([0, 1]^d)\\
     \| g \|_{\mc H_{\rm Dir, Sym}^{1}
    ([0, 1]^d)} = 1}} \big\langle f, g \big\rangle_{L^2_{\rm Sym} ([0, 1]^d)} . \end{equation}
In all what follows, to simplify notation, whenever $d=1$ we will simply denote, as before \[\mc H_{\rm Dir}^{1}=\mc H_{\rm Dir}^{1}
    ([0, 1])=\mc H_{\rm Dir, Sym}^{1}
    ([0, 1]), \qquad \text{and}\qquad \mc H_{\rm Dir}^{-1}=\mc H_{\rm Dir, Sym}^{-1}
    ([0, 1]).\]If we need to make the variable $y$ precise we will highlight it by writing $\mc H_{\rm Dir}^{1}(y)$ and $\mc H_{\rm Dir}^{-1}(y)$.

\begin{lemma}\label{lem:explicit-norms}
  For $f \in \mc H_{\rm Dir, Sym}^1
    ([0, 1]^d)$ we have the more explicit representation
  \[ \| f \|_{\mc H_{\rm Dir, Sym}^1
    ([0, 1]^d)}^2 = \sum_{k \in \mathbb{N}^d} | \pi
     k |^2 \big\langle f, e_{k_1} \otimes \cdots \otimes e_{k_d}  \big\rangle_{L^2([0,1]^d)}^2 \]
  and for any $g \in \mc H_{\rm Dir, Sym}^{-1}
    ([0, 1]^d)$
 \[ \| g \|_{\mc H_{\rm Dir, Sym}^{-1}
    ([0, 1]^d)}^2 = \sum_{k \in \mathbb{N}^d}|
     \pi k |^{- 2} \big\langle g,e_{k_1} \otimes \cdots \otimes
     e_{k_d} \big\rangle_{L^2([0,1]^d)}^2\; , \]
 where $e_m(u)=\sqrt 2\sin(m\pi u)$ has been already defined in Section \ref{ssec:test_functions} and    where $|k|^2 = |k_1|^2 + \cdots + |k_d|^2$ is the squared Euclidean norm of $k=(k_1,\dots,k_d) \in \bb N^d$.
\end{lemma}

\begin{proof}[Proof of Lemma \ref{lem:explicit-norms}]
    Recall that $\{e_m = \sqrt{2} sin(m \pi \cdot)\; ; \; m \ge 1\}$, is an orthonormal basis of $L^2([0,1])$, as is $\{\tilde e_0 \equiv 1, \tilde e_m = \sqrt{2} cos(m \pi \cdot) \; ;\;  m \ge 1\}$. Therefore,
   \[
      \big\{e_{k_1} \otimes \dots \otimes e_{k_{i-1}} \otimes \tilde e_\ell \otimes e_{k_i} \dots \otimes e_{k_{d-1}}\; ; \; k \in \NN^{d-1}, \ell \in \NN_0\big\}
   \] 
   is an orthonormal basis of $L^2([0,1]^d)$ for all $i = 1,\dots, d$. If we use this orthonormal basis to compute $\| \partial_{i} f \|_{L^2_{\rm Sym} ([0, 1]^d)}^2$ and apply integration by parts, then the first equality follows. For the second equality we use the duality of $\mc H_{\rm Dir, Sym}^{-1}
    ([0, 1]^d)$ and $\mc H_{\rm Dir, Sym}^1
    ([0, 1]^d)$.
\end{proof}

We have now all at hand to state and prove the main result of this section:

\begin{corollary}
  \label{cor:poisson-solution norm}The Poisson equation
  \[ L_0 F = W_d (g) \]
  is solvable for all $g \in L^2_{\rm Sym} ([0, 1]^d)$ and moreover \[\| F \|_{1,0}^2 = 2 (d!)\| g
  \|_{\mc H_{\rm Dir, Sym}^{-1}
    ([0, 1]^d)}^2.\]
\end{corollary}

\begin{proof}[Proof of Corollary \ref{cor:poisson-solution norm}]
  Since $L_0$ leaves the $d$--th chaos invariant, $F$ must be of the form $F = W_d (f)$
  and thus $L_0 F = W_d (g)$ is equivalent to $\Delta f = g$,
  which has the explicit solution
  \[ f = \sum_{k \in \mathbb{N}^d}  | \pi k |^{- 2} \big\langle g, e_{k_1} \otimes \cdots \otimes e_{k_d} \big\rangle e_{k_1}\otimes \cdots \otimes e_{k_d}, \]
  so Lemma~\ref{lem:1-norm chaos} and Lemma~\ref{lem:explicit-norms} give
  \begin{align*}
    \| F \|_{1,0}^2 & = 2 (d!) \| f \|_{\mc H_{\rm Dir, Sym}^1
    ([0, 1]^d)}^2 = 2 (d!) \sum_{k \in
    \mathbb{N}^d} 2^d | \pi k |^2 \big\langle f, \sin (\pi k_1 \cdot) \otimes
    \cdots \otimes \sin (\pi k_d \cdot) \big\rangle^2\\
    & = \sum_{k \in \mathbb{N}^d} 2^d | \pi k |^{- 2} \big\langle g, \sin (\pi k_1
    \cdot) \otimes \cdots \otimes \sin (\pi k_d \cdot) \big\rangle^2 =  2 (d!) \|
    g \|_{\mc H_{\rm Dir, Sym}^{-1}
    ([0, 1]^d)}^2 .
  \end{align*}
\end{proof}

\subsection{Burgers/KPZ non-linearity and existence of the process $\mc A$/$\mc B$}
\label{ssec:construct}
All along this section, the integration spaces denoted $L^p$, if not made precise, are in fact $L^p([0,1])$. 

\begin{lemma}
  \label{lem:kpz nonlin}There exists a unique process $\int_0^{\cdot}
  \mc Y_s^{\diamond 2} d s \in \mc C (\mathbb{R}_+, \SN')$ such
  that for all $\psi \in \SN$, all $p \geq 1$ and
  all $\rho \in L^{\infty} ([0, 1]^2)$ the following bound holds:
  \begin{align}\label{eq:kpz-nonlin} \nonumber 
    \mathbb{E} \bigg[ \sup_{t \in [0, T]} &\bigg| \bigg(\int_0^t \mc Y_s^{\diamond 2}
    d s\bigg) (\psi) - \int_0^t \bigg( \int_0^1 \Big\{ \mc Y^2_s \big(\rho (u,
    \cdot)\big) - \big\| \rho (u, \cdot) \big\|_{L^2}^2 \Big\} \psi (u) d u
    \bigg) d s \bigg|^p \bigg]\\ \nonumber
    & \lesssim T^{p / 2} \bigg( \sup_{u\in[0,1]} \Big\{\big\| | \rho (u, \cdot)   | \times | \cdot - u |^{1/2} \big\|_{L^1}\Big\} + \big\| \big\langle \rho (u, \cdot), 1 \big\rangle_{L^2} - 1 \big\|_{L^2_u}
    \bigg)^p \vphantom{\int_0^t}\\
    &\quad \times \bigg(\sup_{u\in[0,1]} \Big\{ \big\| \rho (\cdot, u) \big\|_{L^1} + \big\| \rho (u, \cdot)
    \big\|_{L^1} \Big\}\bigg)^p \big\| \psi \big\|_{L^{\infty}}^p.\vphantom{\int_0^t}
  \end{align} 
\end{lemma}

\begin{proof}[Proof of Lemma \ref{lem:kpz nonlin}] 
  Let us define
  \[
     \widetilde{\mc Y}^{\diamond 2}\big(\rho(u,\cdot)\big) = \mc Y^2 \big(\rho (u, \cdot)\big) - \big\| \rho (u, \cdot) \big\|_{L^2}^2.
  \]
  The strategy to define the process $(\int_0^\cdot \mc Y_s^{\diamond 2} ds)(\psi)$ is to obtain it as the limit of \begin{equation}\label{eq:sequence}\int_0^\cdot \int_0^1\widetilde{\mc Y}_s^{\diamond 2}(\rho_m(u,\cdot))\psi(u)du ds
\end{equation} for some suitable sequence $\{\rho_m\}_{m \in \bb N}$. First, observe that \begin{equation}\label{eq:poiss}\int_0^1 \widetilde{\mc Y}^{\diamond 2} (\rho (u, \cdot)) \psi (u) d u = W_2
     \left( \int_0^1 \rho (u, \cdot)^{\otimes 2} \psi (u) d u
     \right) (\mc Y).\end{equation} So if $\{\rho_m\}_{m \in
  \mathbb{N}}$ is a sequence of functions in $L^{\infty} ([0, 1]^2)$, we can use Corollary~\ref{cor:poisson-solution norm} to solve the Poisson equation with \eqref{eq:poiss} and $\rho=\rho_m$ and then estimate the norm of the solution with the It{\^o} trick given in Lemma \ref{lem:ito}. This gives, for any $m, n \in \mathbb{N}$
  \begin{align*}
    \mathbb{E} \bigg[ \sup_{t \in [0, T]} &\bigg| \int_0^t \bigg( \int_0^1
   \widetilde{\mc Y}^{\diamond 2}_s (\rho_m (u, \cdot)) \psi (u) d u \bigg) d
    s - \int_0^t \bigg( \int_0^1 \widetilde{\mc Y}^{\diamond 2}_s (\rho_n (u, \cdot)) \psi (u)
    d u \bigg) d s \bigg|^p \bigg]\\
    & \lesssim T^{p / 2} \bigg\| \int_0^1 \big(\rho_m (u, \cdot)^{\otimes 2} -
    \rho_n(u, \cdot)^{\otimes 2} \big) \psi (u) d u \bigg\|_{\mc H^{-
    1}_{\rm Dir, Sym} ([0, 1]^2)}^p .
  \end{align*}
  To bound the norm on the right hand side we argue by duality and apply~\eqref{eq:H0s-var-def}: let $f$ be a
  symmetric function in $\SD ([0, 1]^2)$ and consider
  \begin{align*}
    &\bigg\langle \int_0^1 \big(\rho_m (u, \cdot)^{\otimes 2} - \rho_n (u, \cdot)^{\otimes
    2}\big) \psi (u) d u, f \bigg\rangle_{L^2 ([0, 1]^2)}\\
    &\quad = \int_{[0, 1]^3} d u d v_1 d v_2 \big(\rho_m (u, v_1) \rho_m
    (u, v_2) - \delta_0(u - v_1) \delta_0(u - v_2)\big) \psi (u) f (v_1,
    v_2)\\
    &\quad+ \int_{[0, 1]^3} d u d v_1 d v_2 \big(\delta_0(u - v_1) \delta_0
    (u - v_2) - \rho_n (u, v_1) \rho_n (u, v_2)\big) \psi (u)  f (v_1,
    v_2),
  \end{align*}where $\delta_a(\cdot)$ denotes the Dirac delta function at point $a$.
  Both terms on the right hand side are of the same form, so we argue for the
  first one only. For this purpose we decompose
  \begin{align*}
    \rho_m (u, v_1) &\rho_m (u, v_2) - \delta_0(u - v_1) \delta_0(u - v_2)\\
    &= \big(\rho_m (u, v_1) - \delta_0(u - v_1)\big) \rho_m (u, v_2) + \delta_0(u - v_1)
    \big(\rho_m (u, v_2) - \delta_0(u - v_2)\big)
  \end{align*}
  and again only treat the first contribution.  In the following list of inequalities, each step will be made clear by using a notation of the form $(*)$ over the inequality $\leq$ in order to explain where does the inequality come from. We are going to: sum and substract $f(u,v_2)$ and use the triangular inequality (denoted by $\pm f(u,v_2))$; use the Cauchy-Schwartz inequality (C-S); 
   use $L^\infty$ bounds $(L^\infty)$; and finally, use the fact that for any $\varphi \in \Ctwo$, $\|\varphi\|_{L^\infty} \leqslant \|\nabla\varphi\|_{L^2}$, which  will be  denoted below by $(\star\star)$.  Let us bound as follows:
  \begingroup
\allowdisplaybreaks
  \begin{align*}
    &\bigg| \int_{[0, 1]^3} d u d v_1 d v_2 \big(\rho_m (u, v_1) -
    \delta_0(u- v_1)\big) \rho_m (u, v_2) \psi (u) f (v_1, v_2) \bigg|\\
    &\overset{\pm f(u,v_2)}{\leq} \left| \int_{[0, 1]^3} d u d v_1 d v_2 \rho_m (u,
    v_1) \big(f (v_1, v_2) - f (u, v_2)\big) \rho_m (u, v_2) \psi (u) \right|\\
    &\qquad \quad + \bigg| \int_{[0, 1]^2} d u d v_2 \big(\big\langle \rho_m (u, \cdot),
    1 \big\rangle_{L^2} - 1\big) f (u, v_2) \rho_m (u, v_2) \psi (u) \bigg|\\
    &\overset{(C-S)}{\leq} \quad\bigg| \int_{[0, 1]^3} d u d v_1 d v_2 \rho_m (u,
    v_1) \int_u^{v_1}d w\; \partial_1 f (w, v_2)\;  \rho_m (u, v_2) \psi
    (u) \bigg|\\
    &\qquad \quad+ \big\| \big\langle \rho_m (u, \cdot), 1 \big\rangle_{L^2} - 1 \big\|_{L^2_u} \Bigg( \int_0^1
    d u \bigg| \int_0^1 d v_2 f (u, v_2) \rho_m (u, v_2) \psi (u)
    \bigg|^2 \Bigg)^{1/2}\\
    &\overset{\substack{(C-S)\\+L^\infty}}{\leq}\quad \int_{[0, 1]^3} d u d v_1 d v_2 \big| \rho_m (u, v_1)
    \big| \times \big| v_1 - u \big|^{1/2} \big\| \partial_1 f(\cdot,v_2) \big\|_{L^2} \big| \rho_m (u, v_2) \psi (u) \big|\\
    &\qquad \quad + \big\| \big\langle \rho_m (u, \cdot), 1 \big\rangle_{L^2} - 1 \big\|_{L^2_u} \Bigg( \int_0^1
    d x \big\| \rho_m (u, \cdot) \big\|_{L^1}^2 \; \big\| f (u, \cdot)
    \big\|_{L^{\infty}}^2 \; \big\| \psi \big\|_{L^{\infty}}^2 \Bigg)^{1/2}\\
    &\overset{\substack{L^\infty+(\star\star)\\ +(C-S)}}{\leq} \sup_{u\in[0,1]} \Big\{ \big\| | \rho_m (u, \cdot) | \times  | \cdot - u |^{1/2}
    \big\|_{L^1} \Big\}  \sup_{v\in[0,1]} \Big\{\big\| \rho_m (\cdot, v) \big\|_{L^1} \Big\} \big\| \partial_1 f \big\|_{L^2([0,1]^2)}
   \; \big\| \psi \big\|_{L^{\infty}}\\
    &\qquad\quad+ \big\| \big\langle \rho_m (u, \cdot), 1 \big\rangle_{L^2} - 1 \big\|_{L^2_u} \;  \sup_{u\in[0,1]} \Big\{ \big\|
    \rho_m (u, \cdot) \big\|_{L^1} \;  \Big\}\; \big\| \partial_2 f \big\|_{L^2([0,1]^2)} \big\| \psi
    \big\|_{L^{\infty}}\\
    &\overset{\hphantom{(C-S)}}{\leq}  \quad \bigg( \sup_{u\in[0,1]} \Big\{ \big\| | \rho_m (u, \cdot) | \times  | \cdot - u |^{1/2}
    \big\|_{L^1} \Big\} + \big\| \big\langle \rho_m (u, \cdot), 1 \big\rangle_{L^2} - 1 \big\|_{L^2_u} \bigg)
    \\
    &\qquad\quad\times \sup_{u\in[0,1]} \Big\{ \big\| \rho_m (\cdot, u) \big\|_{L^1} + \big\| \rho_m (u, \cdot)
    \big\|_{L^1} \Big\} \times \big\| f \big\|_{\mc H_{\rm Dir,Sym}^1([0, 1]^2)} \; \big\| \psi \big\|_{L^{\infty}},
  \end{align*}
  where in the last step we used the definition of the norm $\|\cdot\|_{\mc H_{\rm Dir,Sym}^1([0, 1]^2)}$.
  By the density of the symmetric $\SD ([0, 1]^2)$--functions in $\mc H_{\rm Dir, Sym}^1
    ([0, 1]^2)$ and the duality of $\mc H_{\rm Dir, Sym}^{-1}
    ([0, 1]^2)$ and $\mc H_{\rm Dir, Sym}^1
    ([0, 1]^2)$ we conclude that
  \begin{align*}
    \mathbb{E} \bigg[ &\sup_{t \in[0 ,T]} \bigg| \int_0^t \bigg( \int_0^1
    \widetilde{\mc Y}^{\diamond 2}_s (\rho_m (u, \cdot)) \psi (u) d u \bigg) d
    s - \int_0^t \bigg( \int_0^1 \widetilde{\mc Y}^{\diamond 2}_s (\rho_n (u, \cdot)) \psi (u)
    d u \bigg) d s \bigg|^p \bigg]\\
    &\lesssim T^{p / 2} \big\| \psi \big\|^p_{L^{\infty}} \vphantom{\int} \\
    & \quad \times \Bigg[\bigg( \sup_{u\in[0,1]} \Big\{ \big\| | \rho_m (u, \cdot) | \times  | \cdot - u |^{1/2}
    \big\|_{L^1} \Big\} + \big\| \big\langle \rho_m (u, \cdot), 1 \big\rangle_{L^2} - 1 \big\|_{L^2_u} \bigg)
    \\
    &\qquad \qquad \times \sup_{u\in[0,1]} \Big\{ \big\| \rho_m (\cdot, u) \big\|_{L^1} + \big\| \rho_m (u, \cdot)
    \big\|_{L^1} \Big\}\\
    & \qquad \quad +  \bigg( \sup_{u\in[0,1]} \Big\{ \big\| | \rho_n (u, \cdot) | \times  | \cdot - u |^{1/2}
    \big\|_{L^1} \Big\} + \big\| \big\langle \rho_n (u, \cdot), 1 \big\rangle_{L^2} - 1 \big\|_{L^2_u} \bigg)
    \\
    &\qquad \qquad \times \sup_{u\in[0,1]} \Big\{ \big\| \rho_n (\cdot, u) \big\|_{L^1} + \big\| \rho_n (u, \cdot)
    \big\|_{L^1} \Big\} \Bigg]^p     . 
  \end{align*}
  \endgroup
  Now let us choose the sequence $\{\rho_m\}_{m \in \bb N}$. For that purpose, let $p^{\rm Dir}$ be the Dirichlet heat kernel on
  $[0, 1]$, i.e.
  \begin{equation}\label{eq:dirker} p^{\rm Dir}_t (u, v) = \sum_{k = 1}^{\infty} e^{- t \pi^2 k^2} e_k (u) e_k (v) , \qquad u,v\in[0,1], t > 0, 
  \end{equation}
and let us set $\rho_m (u, v) = p^{\rm Dir}_{1 / m} (u, v) $ and $m \leq
  n$. Note that  we have, for any $f \in L^2([0,1])$,  
 \[ \lim_{\varepsilon \to 0} \int_0^1 \big|\langle p_\varepsilon^{\rm Dir}(u,\cdot) , f\rangle_{L^2([0,1])}  - f(u)\big|^2 du = 0.\]
  The other properties of $p^{\rm Dir}$ are given in  Appendix~\ref{app:heat-kernel}: from Lemma~\ref{lem:Dirichlet-heat-kernel} we obtain that 
  \begin{align*}
    \mathbb{E} \bigg[ &\sup_{t \in [0, T]} \bigg| \int_0^t \bigg( \int_0^1
   \widetilde{\mc Y}^{\diamond 2}_s (\rho_m (u, \cdot)) \psi (u) d u\bigg) d
    s - \int_0^t \bigg( \int_0^1 \widetilde{\mc Y}^{\diamond 2}_s (\rho_n (u, \cdot)) \psi (u)
    d u \bigg) ds \bigg|^p \bigg]\\
    &\lesssim T^{p / 2} \big\| \psi \big\|^p_{L^{\infty}} \big(m^{- 1 / 4} + n^{- 1 /
    4}\big)^p \lesssim T^{p / 2} \big\| \psi \big\|^p_{L^{\infty}} m^{- p / 4} . \vphantom{\int}
  \end{align*}
  Therefore, the sequence \eqref{eq:sequence} is Cauchy and there exists a limit in $\mc C (\mathbb{R}_+,
  \bb R)$ which we denote with $(\int_0^{\cdot} \mc Y_s^{\diamond 2} d s) (\psi)$. Making use of the bound in terms of $\|\psi\|$, similar arguments as in Section~\ref{sec:tightness} show that even $\psi \mapsto \int_0^\cdot \int_0^1\widetilde{\mc Y}_s^{\diamond 2}(\rho_m(u,\cdot))\psi(u)du ds$ converges in $\mc C (\mathbb{R}_+,\SN')$. By the computation above the limit satisfies~\eqref{eq:kpz-nonlin} and clearly that estimate identifies
  $\int_0^{\cdot} \mc Y_s^{\diamond 2} d s$ uniquely as the limit of
  $\int_0^{\cdot} \widetilde{\mc Y}^{\diamond 2}_s (\rho_m (u, \cdot)) d s$.
\end{proof}

\begin{corollary}\label{cor:kpz-nonlin-A-B}
   We have $- (\int_0^\cdot \mc Y^{\diamond 2}_s d s)(\nabla \varphi) = \mc A_\cdot(\varphi)$ for all $\varphi \in \Ctwo$, where $\mc A$ denotes the process defined in \eqref{eq:B} from the statement of Theorem~\ref{def:energy}.
   
    Similarly $(\int_0^\cdot \mc Y^{\diamond 2}_s d s)(\varphi) = \mc B_\cdot(\varphi)$ for all $\varphi \in \mc S_{\rm{Neu}}$, where $\mc B$ is the process defined in~\eqref{eq:B-def}.
\end{corollary}
\begin{proof}[Proof of Corollary \ref{cor:kpz-nonlin-A-B}]
   According to Lemma~\ref{lem:kpz nonlin} we only need to verify that the approximation of the identity given in Definition \ref{def:approx}, namely:
   \[
\iota_\varepsilon(u, v)= \begin{cases} \varepsilon^{-1}  \; \mathbf{1}_{]u,u+\varepsilon]}(v)  & \text{ if } u \in [0, 1 - 2
       \varepsilon), \\
\varepsilon^{-1}\; \mathbf{1}_{[u-\varepsilon,u[}(v) & \text{ if } u \in [1 - 2
       \varepsilon, 1]. \end{cases}
\] 
   satisfies
   \begin{multline*}
      \lim_{\varepsilon \to 0}\bigg\{  \bigg( \sup_{u\in[0,1]} \Big\{\big\| | \iota_\varepsilon (u, \cdot)   | \times | \cdot - u |^{1/2} \big\|_{L^1}\Big\} + \big\| \big\langle \iota_\varepsilon (u, \cdot), 1 \big\rangle - 1 \big\|_{L^2_u}
    \bigg)^p \vphantom{\int_0^t}\\
 \times \bigg(\sup_{v\in[0,1]} \Big\{ \big\| \iota_\varepsilon (\cdot, v) \big\|_{L^1} + \big\| \iota_\varepsilon (v, \cdot)
    \big\|_{L^1} \Big\}\bigg)^p \bigg\} = 0.
   \end{multline*}
   Clearly $\langle \iota_\varepsilon (u, \cdot), 1 \rangle_{L^2} - 1 = 0$ and $\| \iota_\varepsilon (u, \cdot)\|_{L^1} \equiv 1$ for all $u \in [0,1]$, while
   \[
      \big\| | \iota_\varepsilon (u, \cdot)   | \times | \cdot - u |^{1/2} \big\|_{L^1} \le \varepsilon^{1/2} \big\|  \iota_\varepsilon (u, \cdot) \big\|_{L^1} = \varepsilon^{1/2},
   \]
   and therefore it remains to show that $\| \iota_\varepsilon (\cdot, v) \|_{L^1}$ is uniformly bounded in $v \in [0,1]$ and $\varepsilon \in (0,1]$. But $\iota_\varepsilon(u,v) = 0$ unless $|v - u| \le \varepsilon$, and therefore $\| \iota_\varepsilon (\cdot, v) \|_{L^1} \le 2$.
\end{proof}

By a similar interpolation argument as in the proof of  \cite[Corollary~3.17]{Gubinelli2015Energy} we get the following result:

\begin{lemma}
  \label{lem:kpz-nonlin-conv}For all $\alpha < \frac34$ and $\varphi
  \in \mc S_{\rm{Neu}}$, the process \[\bigg\{\bigg(\int_0^t \mc Y_s^{\diamond 2}
  d s \bigg)(\varphi) \; ; \; t \in [0, T]\bigg\}\] is almost surely in $\mc C^{\alpha}
  ([0, T], \mathbb{R})$, the space of $\alpha$-H\"older continuous functions from $[0,T]$ to $\bb R$. Moreover, writing $p^{\rm Dir}_\varepsilon$ for the Dirichlet heat kernel as defined in \eqref{eq:dirker} and
  \[
     \| f \|_{\mc C^{\alpha} ([0, T],\mathbb{R})} := \sup_{0 \le s < t \le T} \frac{|f(t)-f(s)|}{|t-s|^\alpha},
  \]
  we have for all $p \in [1, \infty)$
  \begin{multline*}  \mathbb{E} \bigg[ \bigg\|
     \bigg(\int_0^{\cdot} \mc Y_s^{\diamond 2} d s\bigg) (\varphi) - \int_0^{\cdot}
     \int_0^1 \Big\{ \mc Y^2_s \left( p^{\rm Dir}_{\varepsilon} (u, \cdot)
     \right) - \big\| p^{\rm Dir}_{\varepsilon} (u, \cdot) \big\|_{L^2}^2 \Big\} \varphi
     (u) d u d s \bigg\|_{\mc C^{\alpha} ([0, T],
     \mathbb{R})}^p \bigg] \\  \xrightarrow[\varepsilon \to 0]{} 0 \vphantom{\int}. \end{multline*}
  In particular, it follows together with Corollary~\ref{cor:kpz-nonlin-A-B} that $\mc B(\varphi)$ has zero quadratic variation.
\end{lemma}

\subsection{Mapping to the stochastic heat equation and conclusion}
 \label{ssec:map}
 
 To prove the uniqueness of our energy solution $\mc Y$ to~\eqref{eq:SBE-gen2}, we would like to apply the Cole-Hopf transform to map $\mc Y$ to a solution of the well posed stochastic heat equation. To do so, we should integrate $\mc Y$ in the space variable and then exponentiate the resulting process. But since we only have an explicit description of the dynamics of $\mc Y$ after testing against a test function $\varphi \in \SD$, we should first mollify $\mc Y$ with a kernel in $\SD$ before carrying out this program.
 
We find it convenient to mollify $\mc Y$ with the Dirichlet heat kernel $p^{\rm Dir}$ which was defined in \eqref{eq:dirker}, and we set
\[\mc Y^{\varepsilon}_t (u)   = \mc Y_t (p^{\rm Dir}_{\varepsilon} (u,
\cdot)).\]
The unique antiderivative $\Theta^\varepsilon_u(v)$ which satisfies $\nabla_u \Theta^\varepsilon_u(v) = p^{\rm Dir}_{\varepsilon} (u,v)$ and $\int_0^1 \Theta^\varepsilon_u(v)du = 0$ for all $v$ is explicitly given as
\[ \Theta^{\varepsilon}_u (v)= \sum_{\ell = 1}^{\infty} e^{-
   \varepsilon \pi^2 \ell^2} (- \ell \pi)^{- 1} 2 \cos (\ell \pi u) \sin (\ell
   \pi v), \qquad u,v \in [0,1], \varepsilon > 0. \]
Note that $\Theta^\varepsilon_u \in \Ctwo$ for all $u \in [0,1]$ and $\varepsilon > 0$, so to integrate $\mc Y^\varepsilon$ in the space variable we set
\[ \mc Z^{\varepsilon}_t (u) = \mc Y_t (\Theta^{\varepsilon}_u),\]
 which is a smooth function with 
$\nabla \mc Z^{\varepsilon} = \mc Y^{\varepsilon}$. Then we obtain
\[ d \mc Z^{\varepsilon}_t (u) = \mc Y_t (\Delta \Theta^{\varepsilon}_u) d t + E\; d \mathcal{A}_t
   (\Theta^{\varepsilon}_u) + \sqrt2 d \mc \nabla_{\rm{Dir}} \mc W_t (\Theta^{\varepsilon}_u), \]
where $\mc A$ is the non-linearity that has been previously defined and
\[
   \nabla_{\rm{Dir}} \mc W_t (\Theta^{\varepsilon}_u) = - \mc W_t (\nabla \Theta^\varepsilon_u) = \mc W_t (p^{\rm Neu}_{\varepsilon} (u, \cdot) - 1),
\]
where $p^{\rm Neu}$ is the heat kernel with Neumann boundary conditions, namely
\begin{equation}\label{eq:neuker} p^{\rm Neu}_{\varepsilon} (u, v) = - \nabla_v \Theta^{\varepsilon}_u (v) + 1
   = 1 + \sum_{\ell = 1}^{\infty} e^{- \varepsilon \pi^2 \ell^2} 2 \cos (\ell \pi
   u) \cos (\ell \pi v), \end{equation} for $u,v\in[0,1]$ and $\varepsilon > 0$. To shorten the notation we write
\[ J^{\varepsilon}_u (\cdot) = p^{\rm Neu}_{\varepsilon} (u, \cdot) - 1 = - \nabla
   \Theta^{\varepsilon}_u (\cdot) \in \mc S_{\rm{Neu}}, \]
so that Corollary~\ref{cor:kpz-nonlin-A-B} gives $\mathcal{A}_t (\Theta^{\varepsilon}_u) = \mathcal{B}_t(J^{\varepsilon}_u) = (\int_0^t
\mc Y_s^{\diamond 2} d s)(J^\varepsilon_u)$.
Furthermore, note that $\Delta \Theta^{\varepsilon}_u (v) = \Delta_u \Theta_u^{\varepsilon} (v)$,
from where we get
\[ d \mc Z^{\varepsilon}_t (u) = \Delta \mc Z^{\varepsilon}_t (u) d t + E
   d \mathcal{B}_t (J^{\varepsilon}_u) + \sqrt{2} d \mc W_t
   (J^{\varepsilon}_u), \]
and $d \langle \mc Z^{\varepsilon} (u) \rangle_t = 2 \| J^{\varepsilon}_u \|_{L^2}^2 \;
d t$ because $\mc B(J^\varepsilon_u)$ has zero quadratic variation by Lemma \ref{lem:kpz-nonlin-conv}.

Now, let us consider the process $\Psi^{\varepsilon}_t (u) = e^{E \mc Z^{\varepsilon}_t (u)}$. By the Cole-Hopf transform for the KPZ equation we expect that $\Psi^\varepsilon$ solves an approximate version of the stochastic heat equation as $\varepsilon \to 0$, and our goal for the remainder of this section is to prove this. It{\^o}'s formula applied to $\Psi_t^\varepsilon(u)$ gives
\begin{align*}
  d \Psi^{\varepsilon}_t (u) & = \Psi^{\varepsilon}_t (u) \left( E
  \Delta \mc Z^{\varepsilon}_t (u) d t + E^2 d \mathcal{B}_t
  (J^{\varepsilon}_u) + E \sqrt{2} d \mc W_t (J^{\varepsilon}_u) +
  E^2 \| J^{\varepsilon}_u \|_{L^2}^2 d t \right)\\
  & = \Delta \Psi^{\varepsilon}_t (u) d t + E^2
  \Psi^{\varepsilon}_t (u) \big(\| J^{\varepsilon}_u \|_{L^2}^2 -
  (\mc Y^{\varepsilon}_t (u))^2\big) d t \\
  &\quad + E^2 \Psi^{\varepsilon}_t (u) \left( d \mathcal{B}_t(p^{\rm Neu}_\varepsilon(u,\cdot)) -  d \langle \mathcal{B}_t, 1 \rangle + E^{- 1} \sqrt{2} d \mc W_t (J^{\varepsilon}_u)\right),
\end{align*}
where we wrote $\langle \mathcal{B}_t, 1 \rangle$ instead of $\mathcal{B}_1
(1)$ to avoid confusion between testing against the constant function $1$ and
evaluating in the point $1$, and where in the second step we applied the chain rule  for the Laplacian, namely $\Delta e^{E\mc Z^\varepsilon} = E e^{E\mc Z^\varepsilon} \Delta \mc Z^\varepsilon + E^2 e^{E\mc Z^\varepsilon} (\nabla \mc Z^\varepsilon)^2$. Next, recall that $\nabla
\mc Z^{\varepsilon}_t (0) = \mc Y_t(p_\varepsilon^{\rm Dir}(0,\cdot)) = \mc Y_t(0) = 0$ and similarly $\nabla \mc Z^{\varepsilon}_t (1) = 0$, which means that
for any $\varphi \in \mc C^2 ([0, 1])$ we can apply integration by parts to obtain for
all $t \geq 0$
\[
  \int_0^1 \Delta \Psi^{\varepsilon}_t (u) \varphi (u) d u = \int_0^1 \Psi^{\varepsilon}_t (u) \Delta \varphi (u) d u - \big(\nabla
  \varphi (1) \Psi^{\varepsilon}_t (1) - \nabla \varphi (0)
  \Psi^{\varepsilon}_t (0)\big) .
\]
We  apply this to derive the following weak formulation, for $\varphi \in \mc C^2 ([0, 1])$,
\begin{align*}
  d \Psi^{\varepsilon}_t (\varphi)  & = \Psi^{\varepsilon}_t (\Delta \varphi) d t + E \sqrt{2}
  \int_0^1 \Psi^{\varepsilon}_t (u) \varphi (u) d \mc W_t (J^{\varepsilon}_u)
  d x - E^2 \Psi^{\varepsilon}_t (\varphi)  d \langle \mathcal{B}_t, 1 \rangle\\
  &\quad + E^2 \int_0^1 (K^{\varepsilon}_u  - 1)\Psi^{\varepsilon}_t (u)
  \varphi (u) d u d t + \frac{E^2}{2} \big(\Psi^\varepsilon_t(0)\varphi(0) + \Psi^\varepsilon_t(1)\varphi(1)\big)dt \\
  &\quad + E^2 d R^{\varepsilon}_t (\varphi) + E^2 Q^{\varepsilon}_t
  (\varphi) d t,
\end{align*}
where
\begin{equation}\label{eq:remainder-1}
  R^{\varepsilon}_t (\varphi) = \int_0^1 \int_0^t \Psi^{\varepsilon}_s
  (u) \Big[ d \mathcal{B}_s (p^{\rm Neu}_{\varepsilon} (u, \cdot)) -
  \big((\mc Y^{\varepsilon}_s (u))^2 -\|p^{\rm Dir}_{\varepsilon} (u, \cdot) \|_{L^2}^2\big) ds -
  K^{\varepsilon}_u ds\Big] \varphi (u) d u
\end{equation}
for the deterministic function
\begin{equation}\label{eq:K-def}
   K^{\varepsilon}_u = E^2 \Bigg(\int_0^1 p^{\rm Neu}_{\varepsilon} (u, w) \big(\Theta^{\varepsilon}_u (w)\big)^2 d w - \bigg( \int_0^1 p^{\rm Dir}_{\varepsilon} (u, v) \Theta^{\varepsilon}_u (v) dv \bigg)^2 \Bigg),
\end{equation}
and
\begin{multline}\label{eq:remainder-2}
   Q^{\varepsilon}_t (\varphi) = \int_0^t \Bigg[ \int_0^1
   \Psi^{\varepsilon}_s (u) \left( \big\| J^{\varepsilon}_u \big\|_{L^2}^2 - \big\|
   p^{\rm Dir}_{\varepsilon} (u, \cdot) \big\|_{L^2}^2 + 1 - \tfrac12 (\delta_0(u) + \delta_1(u)) \right) \varphi (u) d u  \\ -
   E^{- 2} \big(\nabla \varphi (1) \Psi^{\varepsilon}_s (1) - \nabla \varphi
   (0) \Psi^{\varepsilon}_s (0)\big) \Bigg] d s.
\end{multline}
 In the next section we will show the following results concerning the three additional terms $R_t^\varepsilon(\varphi)$, $Q_t^\varepsilon(\varphi)$ and $K_u^\varepsilon$: 

\begin{lemma}
  \label{lem:remainder-1} We have for all $p > 2$ and $\varphi \in \mc C([0, 1])$
  \begin{equation}\label{eq:Rvan} \lim_{\varepsilon \rightarrow 0} \Big\{ \mathbb{E} \big[\| R^{\varepsilon}_\cdot
     (\varphi) \|_{p - \tmop{var} ; [0, T]}^2\big] +\mathbb{E} \big[\sup_{t \in[0,T]} | R^{\varepsilon}_t (\varphi) |^2\big] \Big\} = 0. \end{equation}
\end{lemma}

To control the term $Q^\varepsilon(\varphi)$, we need to assume that $\varphi$ satisfies suitable boundary conditions. But this will not be a problem because the boundary condition is compatible with our formulation of the stochastic heat equation with Robin boundary condition, as defined in Appendix~\ref{app:she}.

\begin{lemma}\label{lem:remainder-2}
For all $p > 2$ and $\varphi \in \mc C^2([0,1])$ with $\nabla \varphi(0) = \nabla \varphi(1) = 0$
  we have almost surely
  \[
     \lim_{\varepsilon \rightarrow 0} \big\| Q^{\varepsilon}_\cdot (\varphi) \big\|_{L^\infty([0,T])} = 0,
  \]
  and the $1$-variation norm (as defined in \eqref{eq:pvar}) of $\big\{Q^\varepsilon(\varphi)\big\}_{\varepsilon \in (0,1]}$ is uniformly bounded.
\end{lemma}

\begin{lemma}
  \label{lem:constant-conv}The function $K^{\varepsilon}$ converges in $L^2 ([0, 1])$ as
  $\varepsilon \rightarrow 0$  to $E^2 / 12$.
\end{lemma}

Given these lemmas, the same arguments as in the proof of \cite[Theorem~2.4]{Gubinelli2015Energy} (see more precisely \cite[Section 4.1]{Gubinelli2015Energy}) show that the process \[\Psi_t (u) =
e^{E \mc Y_t (\Theta^0_u)} = \lim_{\varepsilon \rightarrow 0}
\Psi^{\varepsilon}_t (u), \qquad (t,u) \in [0,T] \times [0, 1],\] solves for all
$\varphi \in \mc \SN$
\begin{align*}
  \Psi_t (\varphi)  = &\; \big\langle e^{E \mc Y_0 (\Theta^0_u)}, \varphi \big\rangle +
   \int_0^t \Psi_s (\Delta \varphi) d s \\& +  \bigg(\frac{E^4}{12} - E^2\bigg) \int_0^t \Psi_s (\varphi) d s - E^2 \int_0^t
  \Psi_s (\varphi) d \langle \mathcal{B}_s, 1 \rangle \\
  & +  \frac{E^2}{2} (\Psi_t(0)\varphi(0) + \Psi_t(1)\varphi(1))dt\\
  & + E \sqrt{2} \bigg(
  \int_0^1 \int_0^t \Psi_s (u) \varphi (u) d \mc W_s (u) d u - \int_0^t
  \Psi_s (\varphi) d \langle \mc W_s, 1 \rangle \bigg) .
\end{align*}
Therefore, setting 
\begin{equation}\label{eq:Z-def}
    \mc X_t = E \sqrt{2} \langle \mc W_t, 1 \rangle - \frac{E^4}{12}
   t + E^2 \langle \mathcal{B}_t, 1 \rangle  \quad \text{ and } \quad   \Phi_t (u)
   = \Psi_t (u) e^{\mc X_t},
\end{equation}
we have $\Phi_0 = \Psi_0$ and $d \langle \mc X\rangle_t = E^2 2 dt$ (recall that $\langle B, 1 \rangle$ has zero quadratic variation). Moreover, for $\varphi \in  \SN$, It\^o's formula gives
\begin{align*}
  d \Phi_t (\varphi) & = e^{\mc X_t} d \Psi_t (\varphi) + \Phi_t (\varphi)
  d \mc X_t + \frac{1}{2} \Phi_t (\varphi) d \langle \mc X \rangle_t + d
  \langle \Psi (\varphi), e^{\mc X} \rangle_t\\
  & = \Phi_t (\Delta \varphi) d t + \frac{E^2}{2} \big(\Phi_t(0)\varphi(0) + \Phi_t(1)\varphi(1)\big)dt \\ & \quad + E \sqrt{2} \int_0^1 \Phi_t (u)
  \varphi (u) d \mc W_t (u) d u,
\end{align*}
where we used that
\begin{align*}
   \langle \Psi (\varphi), e^{\mc X} \rangle_t & = E^2 2 \bigg\langle
  \int_0^1 \int_0^\cdot \Psi_s (u) \varphi (u) d \mc W_s (u) d u - \int_0^\cdot
  \Psi_s (\varphi) d \langle \mc W_s, 1 \rangle, \int_0^\cdot e^{\mc X_s} d \langle \mc W_s, 1\rangle \bigg\rangle_t \\
  &= E^2 2 \bigg( \int_0^1 \int_0^t \Phi_s (u) \varphi (u) ds d u - \int_0^t
  \Phi_s (\varphi) ds \bigg) = 0.
\end{align*}
Moreover, $\Phi$ is locally uniformly bounded in $L^2(\bb P)$,  and more precisely $\Phi \in \mc L_C^2([0,T])$, as given by the following:

\begin{lemma}\label{lem:exp-int}
   Consider the process $\Phi$ defined in~\eqref{eq:Z-def}. There exists $T>0$ such that
   \[
      \sup_{t \in [0,T], \, u\in [0,1]}\bb E[\Phi_t(u)^2] < \infty.
   \]
   Moreover, if $R\ge 0$ and $\sup_{t \in [0,R], \, u\in [0,1]}  \bb E[\Phi_t(u)^{8}] < \infty$, then for some $T > 0$, independent of $R$,
   \[
      \sup_{t \in [0,R+T], \, u\in [0,1]}   \bb E[\Phi_t(u)^{2}] < \infty.
   \]
\end{lemma}

\begin{proof}[Proof of Lemma \ref{lem:exp-int}]
	The proof is basically the same as for \cite[Lemma~B.1]{Gubinelli2015Energy}, and therefore we omit it.
\end{proof}

In other words, $\{\Phi_t \; ; \; t \in [0,T]\}$ is a weak solution to the stochastic heat equation
\[
   d \Phi_t = \Delta_{\rm Rob} \Phi_t d t + E \sqrt{2} \Phi_t d \mc W_t, \qquad \Phi_0 (u) = e^{E \mc Y_0 (\Theta^0_u)},
\]
with Robin boundary conditions
\[
   \nabla \Phi_t(0) = - \frac{E^2}{2} \Phi_t(0), \qquad \nabla \Phi_t(1) = \frac{E^2}{2} \Phi_t(1),
\]
as defined  in Proposition~\ref{def:she-rob}. By the uniqueness property given in that proposition, on $[0,T]$ the process $\mc Z$ is equal to the unique weak solution of~\eqref{eq:she-rob}. Then Lemma~\ref{lem:she-moment} stated ahead gives the better moment bound
 \[
    \sup_{t \in [0,T], \, u \in [0,1]} \bb E[\Phi_t(u)^8] < \infty,
 \]
 and now Lemma~\ref{lem:exp-int} shows that $\sup_{t \in [0,2T], \, u \in [0,1]} \bb E[\Phi_t(u)^2] < \infty$, which means that $\{\Phi_t \; ; \; t \in [0,2T]\}$ is the unique weak solution of~\eqref{eq:she-rob} on $[0,2T]$. Now we can keep iterating this argument to see that $\Phi$ is uniquely determined on all of $[0,\infty)$. 

Since for all $\varphi \in \Ctwo$
\begin{equation}\label{eq:lognabla} \mc Y(\varphi) = - E^{-1} \log \Psi (\nabla \varphi) = - E^{- 1} \log (\Phi- \mc X)(\nabla \varphi) = - E^{- 1}  \log (\Phi) (\nabla \varphi),
\end{equation}
which can be easily verified for $\mc Y^\varepsilon$ and $\Psi^\varepsilon$ and then carries over to the limit $\varepsilon \to 0$, 
the uniqueness of $\mc Y$ follows from that of $\Phi$.

As in the proof of \cite[Theorem~2.10]{Gubinelli2015Energy} we also obtain the uniqueness of the almost stationary energy solution to the KPZ equation, i.e.~Theorem~\ref{def:energy-KPZ}.

\subsection{Convergence of the remainders}

In this section we prove successively Lemma \ref{lem:remainder-1}, Lemma \ref{lem:remainder-2} and Lemma \ref{lem:constant-conv}.

\subsubsection{Proof of Lemma~\ref{lem:remainder-1}}

Recall that $R^\varepsilon(\varphi)$ was defined in~\eqref{eq:remainder-1} and that $\mc B_t(\psi)=(\int_0^t \mc Y_s^{\diamond 2} ds)(\psi)$. Given $\delta > 0$ we approximate $R^{\varepsilon}$ by
\begin{align} \label{eq:r-eps-del-def} \nonumber
  R^{\varepsilon, \delta}_t (\varphi) & =  \; \int_0^t \int_0^1
  \Psi^{\varepsilon}_s (u) \Big\langle (\mc Y^{\delta}_s(\cdot))^2 -\big\|p^{\rm Dir}_{\delta} (\cdot, v)
  \big\|_{L^2_v}^2 \; ,\;  p^{\rm Neu}_{\varepsilon} (u, \cdot)\Big\rangle_{L^2}  \varphi (u) d u d
  s\\ \nonumber
  &\quad - \int_0^t \int_0^1 \Psi^{\varepsilon}_s (u) \Big[  (\mc Y^{\varepsilon}_s (u))^2
  -\big\|p^{\rm Dir}_{\varepsilon} (u, \cdot) \big\|_{L^2}^2 + K^{\varepsilon, \delta}_u \Big]
  \varphi (u) d u d s \\
  & =: \int_0^t r^{\varepsilon,\delta}(\mc Y_s) ds,
\end{align}
where $K^{\varepsilon, \delta}$ will be defined in
equation~(\ref{eq:K-eps-delta-def}) below. Provided that $K^{\varepsilon,
\delta}$ converges to $K^{\varepsilon}$ in $L^2([0, 1])$ as $\delta \rightarrow
0$ we obtain from Lemma~\ref{lem:kpz-nonlin-conv} that
\[ \lim_{\delta \rightarrow 0} \Big\{ \mathbb{E} \Big[\big\| R^{\varepsilon, \delta}
   (\varphi) - R^{\varepsilon} (\varphi) \big\|_{p - \tmop{var} ; [0, T]}^2\Big]
   +\mathbb{E} \Big[\sup_{t \in [0,T]} \big| R^{\varepsilon, \delta}_t (\varphi) -
   R^{\varepsilon}_t (\varphi) \big|^2\Big] \Big\} = 0, \]
which holds because we can control \[\int_0^\cdot
 \Psi^{\varepsilon}_s (u) \Big\langle (\mc Y^{\delta}_s)^2(\cdot) -\big\|p^{\rm Dir}_{\delta} (\cdot, v)
  \big\|_{L^2_v}^2 \; ,\;  p^{\rm Neu}_{\varepsilon} (u, \cdot)\Big\rangle_{L^2}\; d s\] as a Young integral: indeed, by Lemma~\ref{lem:kpz-nonlin-conv} the integrator converges in $\alpha$-H\"older norm for any $\alpha < \frac34$ and the integrand is almost surely $\beta$-H\"older continuous for any $\beta < \frac12$ and $\alpha + \beta > 1$; see~\cite{Young1936} or~\cite[Section~1.3]{Lyons2007} for details on the Young integral. So, to prove the convergence claimed in
Lemma~\ref{lem:remainder-1}, it suffices to show  that $R^{\varepsilon, \delta} (\varphi)$ vanishes in the same sense as in \eqref{eq:Rvan}, as
first $\delta \rightarrow 0$ and then $\varepsilon \rightarrow 0$.

By Corollary~\ref{cor:KV} it suffices to show $\lim_{\varepsilon \rightarrow
0} \lim_{\delta \rightarrow 0} \| r^{\varepsilon, \delta} (\cdot)
\|_{- 1,0} = 0$, where the random variable $r^{\varepsilon, \delta}$ was defined in~\eqref{eq:r-eps-del-def}.
  Note that $r^{\varepsilon,\delta}$ satisfies the assumption of Corollary \ref{cor:KV}: it is clearly in $L^2(\mu)$ because all the kernels appearing in its definition are bounded continuous functions. The fact that $r^{\varepsilon,\delta}$ is in $\mc H_{\rm Dir}^{- 1}$ is not obvious but will be a consequence of our estimates below. 
  
Recall that $\mc Y(\psi_1) \mc Y(\psi_2) = W_2(\psi_1 \otimes \psi_2)(\mc Y) + \langle \psi_1, \psi_2\rangle_{L^2}$ for all $\psi_1, \psi_2 \in L^2$, see~\cite[Proposition~1.1.2]{Nualart2006}, and therefore
\[ \Big\langle (\mc Y^{\delta}(\cdot))^2 -\big\|p^{\rm Dir}_{\delta} (\cdot, v)
  \big\|_{L^2_v}^2 \; ,\;  p^{\rm Neu}_{\varepsilon} (u, \cdot)\Big\rangle_{L^2} - (\mc Y^{\varepsilon})^2 (u)
+\big\|p^{\rm Dir}_{\varepsilon} (u, \cdot) \big\|_{L^2}^2 = W_2 (g^{\varepsilon,
   \delta}_u) (\mc Y) \]
with
\begin{equation}
  \label{eq:g-eps-delta-def} g^{\varepsilon, \delta}_u (v_1, v_2) =
  \int_0^1 p^{\rm Dir}_{\delta} (w, v_1) p^{\rm Dir}_{\delta} (w, v_2) p^{\rm Neu}_{\varepsilon} (u,
  w) d w - p^{\rm Dir}_{\varepsilon} (u, v_1) p^{\rm Dir}_{\varepsilon} (u, v_2) .
\end{equation}
To control the $\mathfrak{H}_0^{- 1}$--norm of $r^{\varepsilon, \delta} (\cdot)$ we need to bound $\mathbb{E} [r^{\varepsilon, \delta} (\mc Y_0) F (\mc Y_0)]$ for an arbitrary cylinder function $F \in \mathcal{C}$. As
in \cite[Lemma~4.4]{Gubinelli2015Energy} we use Gaussian (partial) integration by
parts to show that, with the choice
\begin{equation}
  \label{eq:K-eps-delta-def} K^{\varepsilon, \delta}_u = E^2
  \int_{[0, 1]^2} g^{\varepsilon, \delta}_u (v_1, v_2) \Theta^{\varepsilon}_u
  (v_1) \Theta^{\varepsilon}_u (v_2) d v_1 d v_2,
\end{equation}
we have the following result:
\begin{lemma}\label{lem:r-eps-constants}
We can bound
  \begin{align*}
    \big\| r^{\varepsilon, \delta} (\cdot) \big\|_{- 1,0}^2  \le  & \;
   \mathbb{E} \bigg[ \bigg\| \int_0^1 d u \varphi (u) \big(W_1
    (g^{\varepsilon, \delta}_u (v_1, \cdot)) \diamond \Psi^{\varepsilon}_0
    (u)\big) \bigg\|_{\mc H_{\rm Dir}^{- 1}(v_1)}^2 \bigg]\\
    & + \mathbb{E} \bigg[ \bigg\| \int_0^1 d u \varphi (u) \int_0^1 d
    v_1 g^{\varepsilon, \delta}_u (v_1, v_2) \mathrm{D}_{v_1}
    \Psi^{\varepsilon}_0 (u) \bigg\|_{\mc H_{\rm Dir}^{- 1}(v_2)}^2 \bigg] = :
 A^{\varepsilon, \delta} + B^{\varepsilon, \delta}
  \end{align*}
  where the $\diamond$-product is defined below:
\[ W_1 (g^{\varepsilon, \delta}_u (v_1, \cdot)) \diamond
   \Psi^{\varepsilon}_0 (u) := W_1 (g^{\varepsilon, \delta}_u (v_1,
   \cdot)) \Psi^{\varepsilon}_0 (u) - \int_0^1 g^{\varepsilon, \delta}_u
   (v_1, v_2) \mr D_{v_2} \Psi^{\varepsilon}_0 (u) d v_2. \]

\end{lemma}

To prove Lemma \ref{lem:r-eps-constants}, we control separately $A^{\varepsilon, \delta}$ and $B^{\varepsilon,
\delta}$ and prove that they vanish respectively in Lemma \ref{lem:Adelta} and Lemma \ref{lem:Bdelta} below. Before that, we need some auxiliary results that we give now: 

\begin{lemma}\label{lem:g-bound}
  The sequence $\| g^{\varepsilon, \delta}_u (\cdot, v_2)\|_{\mc H_{\rm Dir}^{- 1}}$ is uniformly bounded in $(\delta, u, v_2) \in (0,1] \times [0,1]^2$, and we have for almost all (w.r.t.~Lebesgue) $v_2 \in [0,1]$, and all $\varepsilon \in (0,1]$
  \[
    \lim_{\delta \rightarrow 0} \big\| g^{\varepsilon, \delta}_u (\cdot, v_2)
    \big\|_{\mc H_{\rm Dir}^{- 1}} \lesssim \big| p^{\rm Neu}_{\varepsilon} (u, v_2) - p^{\rm Dir}_{\varepsilon}
    (u, v_2) \big| + p^{\rm Dir}_{\varepsilon} (u, v_2) \times \big( \varepsilon^{- 1/4} |
    v_2 - u | + \varepsilon^{1/4} \big).
  \] 
\end{lemma}

\begin{proof}[Proof of Lemma \ref{lem:g-bound}]
  Recall that $g^{\varepsilon, \delta}_u (v_1, v_2)$ is given in~\eqref{eq:g-eps-delta-def}. We use the explicit characterization of the $\mc H_{\rm Dir}^{- 1}$--norm in Lemma~\ref{lem:explicit-norms} to write
  \[
    \big\| g^{\varepsilon, \delta}_u (\cdot, v_2) \big\|_{\mc H_{\rm Dir}^{- 1}}^2 = \pi^{-2}
    \sum_{k = 1}^{\infty} k^{- 2} \big| \big\langle  g^{\varepsilon, \delta}_u
    (\cdot, v_2) \; ,\; \sin (k \pi \cdot) \big\rangle_{L^2} \big|^2,
  \]
  and since $\int_0^1 p^{\rm Dir}_{t} (u, v_1) \sin (k \pi v_1) d v_1 = e^{-
  t \pi^2 k^2} \sin (k \pi u)$ for $k \in \bb N$, $t > 0$, $u \in [0,1]$, and also $p^{\rm Neu}_{\varepsilon} (u, \cdot)$ integrates to $1$, we have
  \begin{multline*}
    \big| \big\langle  g^{\varepsilon, \delta}_u
    (\cdot, v_2) \; , \;  \sin (k \pi \cdot) \big\rangle_{L^2([0,1])} \big| \\ = \bigg| \int_0^1 \big[ e^{- \delta \pi^2 k^2}
    p^{\rm Dir}_{\delta} (w, v_2) - e^{- \varepsilon \pi^2 k^2} p^{\rm Dir}_{\varepsilon} (u,
    v_2) \big] \sin (k \pi w) p^{\rm Neu}_{\varepsilon} (u, w) d w \bigg|,
  \end{multline*}
  from where the uniform bound follows. Moreover, since $p^{\rm Dir}$ are the transition densities of a killed Brownian motion we have for all $v_2 \in (0,1)$
  \begin{align*}
    \lim_{\delta \rightarrow 0} \bigg| &\int_0^1 \big[ e^{- \delta \pi^2 k^2}
    p^{\rm Dir}_{\delta} (w, v_2) - e^{- \varepsilon \pi^2 k^2} p^{\rm Dir}_{\varepsilon} (u,
    v_2) \big] \sin (k \pi w) p^{\rm Neu}_{\varepsilon} (u, w) d w \bigg|\\
    &= \bigg| \sin (k \pi v_2) p^{\rm Neu}_{\varepsilon} (u, v_2) - e^{- \varepsilon
    \pi^2 k^2} p^{\rm Dir}_{\varepsilon} (u, v_2) \int_0^1 \sin (k \pi w)
    p^{\rm Neu}_{\varepsilon} (u, w) d w \bigg|\\
    &\leq \big| p^{\rm Neu}_{\varepsilon} (u, v_2) - p^{\rm Dir}_{\varepsilon} (u, v_2) \big| \vphantom{\int}\\
    & \quad  + \big|
    p^{\rm Dir}_{\varepsilon} (u, v_2) \big| \times \bigg| \int_0^1 \big(\sin (k \pi v_2) -
    e^{- \varepsilon \pi^2 k^2} \sin (k \pi w)\big) p^{\rm Neu}_{\varepsilon} (u, w)
    d w \bigg| .
  \end{align*}
  The integral term on the right hand side is bounded by
  \begin{align*}
    \bigg| \int_0^1 &\big(\sin (k \pi v_2) - e^{- \varepsilon \pi^2 k^2} \sin (k
    \pi w)\big) p^{\rm Neu}_{\varepsilon} (u, w) d w \bigg|\\
    & \leq |1 - e^{- \varepsilon \pi^2 k^2}| + \bigg| \int_0^1 \big(\sin (k \pi
    v_2) - \sin (k \pi w)\big) p^{\rm Neu}_{\varepsilon} (u, w) d w \bigg|\\
    & \lesssim  (\varepsilon \pi^2 k^2)^{1/2} + \bigg| \int_0^1 k \big(| v_2 - u |
    + | u - w |\big) p^{\rm Neu}_{\varepsilon} (u, w) d w \bigg| \lesssim k \big(| v_2 -
    u | + \varepsilon^{1/2}\big),
  \end{align*}
  where we applied the heat kernel bounds stated and proved in Lemma~\ref{lem:Dirichlet-heat-kernel} below for $p^{\rm Neu}$, and the bound $|1 - e^{-c}| \le  c^{1/2}$, true for any $c>0$. We use this bound for $k \leq
  \varepsilon^{- 1/2}$, while for $k > \varepsilon^{- 1/2}$ we simply
  bound the integral by $2$. Thus, we obtain
  \begin{align*}
    \lim_{\delta \rightarrow 0} \big\| g^{\varepsilon, \delta}_u (\cdot,
    v_2) \big\|_{\mc H_{\rm Dir}^{- 1}}^2 & \lesssim \sum_{k = 1}^{\infty} k^{- 2} \bigg(\big| p^{\rm Neu}_{\varepsilon} (u, v_2) -
    p^{\rm Dir}_{\varepsilon} (u, v_2) \big|^2  \\
    & \vphantom{\sum_{k=1}^\infty}  + \big|p^{\rm Dir}_{\varepsilon} (u, v_2)\big|^2 \times
    \big[\mathbf{1}_{(0, \varepsilon^{- 1/2}]}(k)\; \big(k^2 (| v_2 - u |^2 +
    \varepsilon)\big) +\mathbf{1}_{(\varepsilon^{- 1/2},+\infty)}(k)\big]\bigg)\\
    & \lesssim \Big(\big| p^{\rm Neu}_{\varepsilon} (u, v_2) - p^{\rm Dir}_{\varepsilon} (u, v_2) \big|^2  \\ & \qquad \qquad +
    \big|p^{\rm Dir}_{\varepsilon} (u, v_2)\big|^2 \times \big[\varepsilon^{- 1/2} | v_2 - u |^2 +
    \varepsilon^{1/2}\big]\Big)\vphantom{\int},
  \end{align*}
  and now it suffices to take the square root on both sides.
\end{proof}

The following lemma, which is a simple consequence of the estimates we derived
so far, will be useful for bounding both constants $A^{\varepsilon, \delta}$
and $B^{\varepsilon, \delta}$ of Lemma~\ref{lem:r-eps-constants}:

\begin{lemma}
  \label{lem:g-theta-bound}
For all $u' \in [0,1]$ the function $x\mapsto \big\| \int_0^1 g^{\varepsilon, \delta}_u (v_1, \cdot) \Theta^{\varepsilon}_{u'} (v_1) d v_1 \big\|_{\mc H_{\rm Dir}^{-1}}$ is uniformly bounded in $(\delta,u) \in (0,1] \times [0,1]$ and satisfies, for all $\varepsilon \in (0,1]$
  \[ \sup_{u' \in [0, 1]} \lim_{\delta \rightarrow 0} \bigg\| \int_0^1
     g^{\varepsilon, \delta}_u (v_1, \cdot) \Theta^{\varepsilon}_{u'} (v_1)
     d v_1 \bigg\|_{\mc H_{\rm Dir}^{- 1}} \lesssim \frac{\varepsilon^{1/2}}{u (1 - u)} \wedge 1  + \varepsilon^{1/4}.  \] 
\end{lemma}

\begin{proof}[Proof of Lemma \ref{lem:g-theta-bound}]
  From Lemma~\ref{lem:g-bound} (note that $g^{\varepsilon, \delta}_u$ is
  symmetric in its two arguments) we obtain
  \begin{align*}
    &\lim_{\delta \rightarrow 0} \bigg\| \int_0^1 g^{\varepsilon, \delta}_u
    (v_1, \cdot) \Theta^{\varepsilon}_{u'} (v_1) d v_1 \bigg\|_{\mc H_{\rm Dir}^{- 1}} \leq \lim_{\delta \rightarrow 0} \int_0^1 \big\| g^{\varepsilon,
    \delta}_u (v_1, \cdot) \big\|_{\mc H_{\rm Dir}^{- 1}} | \Theta^{\varepsilon}_{u'} (v_1)
    | d v_1\\
    &\quad \lesssim\int_0^1 \Big(\big| p^{\rm Neu}_{\varepsilon} (u, v_1) - p^{\rm Dir}_{\varepsilon} (u,
    v_1) \big| + p^{\rm Dir}_{\varepsilon} (u, v_1) \times \big( \varepsilon^{- 1/4} | v_1
    - u | + \varepsilon^{1/4} \big)\Big) | \Theta^{\varepsilon}_{u'} (v_1) | d
    v_1\\
    &\quad \leq  \| \Theta^{\varepsilon}_{u'} \|_{L^{\infty}} \Big(\big\|
    p^{\rm Neu}_{\varepsilon} (u, \cdot ) - p^{\rm Dir}_{\varepsilon} (u, \cdot)
    \big\|_{L^1} + \varepsilon^{- 1/4} \big\| p^{\rm Dir}_{\varepsilon} (u, \cdot) \times |
    \cdot - u |\; \big\|_{L^1} + \varepsilon^{1/4}\Big)\\
    & \quad \lesssim  \frac{\varepsilon^{1/2}}{u (1
     - u)} \wedge 1 + \varepsilon^{1/4}, \qquad \text{for any } u' \in [0,1],
  \end{align*}
  where the last estimate comes from 
  Lemma~\ref{lem:Dirichlet-heat-kernel} and Lemma~\ref{lem:Neumann-Dirichlet} of the Appendix.
\end{proof}

Finally we are now able to state and prove that $A^{\varepsilon,\delta}$ and $B^{\varepsilon, \delta}$ vanish:

\begin{lemma}\label{lem:Bdelta}
  We have
  \[ \lim_{\varepsilon \rightarrow 0} \lim_{\delta \rightarrow 0}
     B^{\varepsilon, \delta}  = 0. \]
\end{lemma}

\begin{proof}[Proof of Lemma \ref{lem:Bdelta}]
  Recall that $\Psi_0^\varepsilon(u)=e^{E\mc Z_0^\varepsilon(u)} = e^{E \mc Y_0(\Theta^\varepsilon_u)}$ and therefore its Malliavin derivative is $\mathrm{D}_{v_1} \Psi^{\varepsilon}_0 (u) = E
  \Psi^{\varepsilon}_0 (u) \Theta^{\varepsilon}_u (v_1)$. As a result,
  \begin{align*}
    B^{\varepsilon, \delta} &\overset{\hphantom{(C-S)}}{=}\mathbb{E} \bigg[ \bigg\| \int_0^1 d u
    \varphi (u) \int_0^1 d v_1 g^{\varepsilon, \delta}_u (v_1, \cdot)
    E \Theta^{\varepsilon}_u (v_1) \Psi^{\varepsilon}_0 (u)
    \bigg\|_{\mc H_{\rm Dir}^{- 1}}^2 \bigg]\\
    &\overset{(C-S)}{\leq} \int_0^1 d u | \varphi (u) |^2 E^2 \bigg\| \int_0^1
    d v_1 g^{\varepsilon, \delta}_u (v_1, \cdot) \Theta^{\varepsilon}_u
    (v_1) \bigg\|_{\mc H_{\rm Dir}^{- 1}}^2 \mathbb{E} \big[| \Psi^{\varepsilon}_0 (u) |^2\big],
  \end{align*}
  so from Lemma~\ref{lem:g-theta-bound} together with the dominated convergence theorem we get
  \[ \lim_{\varepsilon \rightarrow 0} \lim_{\delta \rightarrow 0}
     B^{\varepsilon, \delta} \lesssim \lim_{\varepsilon \rightarrow 0}
     \int_0^1 d u | \varphi (u) |^2 \left( \left( \frac{\varepsilon^{1/2}}{u (1 - u)} \wedge 1 \right) + \varepsilon^{1/4} \right)^2
     \mathbb{E} \big[| \Psi^{\varepsilon}_0 (u) |^2\big] = 0. \]
\end{proof}

To complete the proof of Lemma~\ref{lem:remainder-1} we need to control
$A^{\varepsilon, \delta}$, which is achieved in the next lemma.

\begin{lemma}\label{lem:Adelta}
  We have
  \[ \lim_{\varepsilon \rightarrow 0} \lim_{\delta \rightarrow 0}
     A^{\varepsilon, \delta}  = 0. \]
\end{lemma}

\begin{proof}[Proof of Lemma \ref{lem:Adelta}]
  We expand
  \begin{align}
    A^{\varepsilon,\delta}& = \mathbb{E} \bigg[ \bigg\| \int_0^1 d u \varphi (u) \big(W_1
    (g^{\varepsilon, \delta}_u (v_1, \cdot)\big) \diamond \Psi^{\varepsilon}_0
    (u)) \bigg\|_{\mc H_{\rm Dir}^{- 1} (v_1)}^2 \bigg] \notag\\
    & = \int_0^1 d u \int_0^1 d u' \varphi (u) \varphi (u') \notag
    \\
    & \qquad \qquad \times \mathbb{E} \Big[\big\langle W_1 (g^{\varepsilon, \delta}_u (v_1, \cdot))
    \diamond \Psi^{\varepsilon}_0 (u)\; ,\;  W_1 (g^{\varepsilon, \delta}_{u'} (v_1,
    \cdot)) \diamond \Psi^{\varepsilon}_0 (u') \big\rangle_{\mc H_{\rm Dir}^{- 1} (v_1)}\Big],\label{eq:Aest}
  \end{align}
  and as in the proof of \cite[Lemma~4.7]{Gubinelli2015Energy} we get
  from an integration by parts
  \begin{align}\label{eq:A-eps-delta-pr1}
    &\mathbb{E} \Big[\big\langle W_1 (g^{\varepsilon, \delta}_u (v_1, \cdot))
    \diamond \Psi^{\varepsilon}_0 (u)\; ,\;  W_1 (g^{\varepsilon, \delta}_{u'} (v_1,
    \cdot)) \diamond \Psi^{\varepsilon}_0 (u') \big\rangle_{\mc H_{\rm Dir}^{- 1} (v_1)} \Big] \nonumber\\ \nonumber
    &\quad =\mathbb{E} \big[\Psi^{\varepsilon}_0 (u) \Psi^{\varepsilon}_0 (u')\big] \int_0^1
    \big\langle g^{\varepsilon, \delta}_u (v_1, v_2), g^{\varepsilon, \delta}_{u'}
    (v_1, v_2) \big\rangle_{\mc H_{\rm Dir}^{- 1}(v_1)} d v_2\\ \nonumber
    &\quad \quad + E^2 \mathbb{E} \big[\Psi^{\varepsilon}_0 (u) \Psi^{\varepsilon}_0
    (u')\big] \vphantom{\int_0^1}\\ 
    & \qquad \qquad \times \int_0^1 d v_2 \int_0^1 d v_3 \Theta^{\varepsilon}_u (v_3)
    \Theta^{\varepsilon}_{u'} (v_2) \big\langle g^{\varepsilon, \delta}_u (v_1,
    v_2), g^{\varepsilon, \delta}_{u'} (v_1, v_3) \big\rangle_{\mc H_{\rm Dir}^{- 1}(v_1)} .
  \end{align}
  For the second term on the right hand side we simply estimate from the Cauchy-Schwartz inequality
  \begin{align*}
    \lim_{\delta \rightarrow 0} \bigg| &\int_0^1 d v_2 \int_0^1 d v_3
    \Theta^{\varepsilon}_u (v_3) \Theta^{\varepsilon}_{u'} (v_2) \big\langle
    g^{\varepsilon, \delta}_u (v_1, v_2) , g^{\varepsilon, \delta}_{u'} (v_1,
    v_3) \big\rangle_{\mc H_{\rm Dir}^{- 1} (v_1)} \bigg|\\
    &\quad = \lim_{\delta \rightarrow 0} \bigg| \bigg\langle \int_0^1 d v_2
    \Theta^{\varepsilon}_{u'} (v_2) g^{\varepsilon, \delta}_u (v_1, v_2)\; ,\;
    \int_0^1 d v_3 \Theta^{\varepsilon}_u (v_3) g^{\varepsilon,
    \delta}_{u'} (v_1, v_3) \bigg\rangle_{\mc H_{\rm Dir}^{- 1} (v_1)} \bigg|\\
    &\quad \leq \lim_{\delta \rightarrow 0} \bigg\| \int_0^1 d v_2
    \Theta^{\varepsilon}_{u'} (v_2) g^{\varepsilon, \delta}_u (\cdot, v_2)
    \bigg\|_{\mc H_{\rm Dir}^{- 1}} \bigg\| \int_0^1 d v_3 \Theta^{\varepsilon}_u
    (v_3) g^{\varepsilon, \delta}_{u'} (\cdot, v_3) \bigg\|_{\mc H_{\rm Dir}^{- 1}}\\
    &\quad \lesssim \bigg(  \frac{\varepsilon^{1/2}}{u (1 - u)} \wedge 1
     + \varepsilon^{1/4} \bigg) \bigg( \frac{\varepsilon^{1/2}}{u' (1 - u')} \wedge 1  + \varepsilon^{1/4} \bigg),
  \end{align*}
  where in the last step we used Lemma~\ref{lem:g-theta-bound} and the symmetry of $g^{\varepsilon,
  \delta}_u$. So when we plug this contribution into \eqref{eq:Aest} inside the integration w.r.t.~$u$ and $u'$,  one easily shows that it
  vanishes for $\varepsilon \rightarrow 0$. We are left with bounding the
  first contribution coming from~\eqref{eq:A-eps-delta-pr1}. We set $V^{\varepsilon} (u) = |
  \varphi (u) | \mathbb{E} [\Psi^{\varepsilon}_0 (u)^2]^{1 / 2}$ and obtain
  from Lemma~\ref{lem:g-bound} and Lemma~\ref{lem:Neumann-Dirichlet} as well
  as Lemma~\ref{lem:Dirichlet-heat-kernel}:
  \begingroup
  \allowdisplaybreaks
  \begin{align*}
    \lim_{\delta \rightarrow 0} \int_0^1 d u &\int_0^1 d u' \varphi
    (u) \varphi (u') \mathbb{E} \big[\Psi^{\varepsilon}_0 (u)
    \Psi^{\varepsilon}_0 (u')\big] \int_0^1 \big\langle g^{\varepsilon, \delta}_u
    (v_1, v_2), g^{\varepsilon, \delta}_{u'} (v_1, v_2) \big\rangle_{\mc H_{\rm Dir}^{- 1}
    (v_1)} d v_2\\
    & \overset{\substack{\hphantom{\text{Lem. \ref{lem:g-bound}}}\\(C-S)}}{\leq} \lim_{\delta \rightarrow 0} \int_0^1 d v_2 \bigg( \int_0^1
    d u V^{\varepsilon} (u) \big\| g^{\varepsilon,\delta}_u (\cdot, v_2) \big\|_{\mc H_{\rm Dir}^{- 1}}
    \bigg)^2\\
    & \overset{\text{(Lem. \ref{lem:g-bound})}}{\lesssim}  \int_0^1 d v_2 \bigg( \int_0^1 d u V^{\varepsilon} (u)
    \Big[| p^{\rm Neu}_{\varepsilon} (u, v_2) - p^{\rm Dir}_{\varepsilon} (u, v_2) | \\
    & \qquad \qquad \qquad \qquad \qquad \qquad \qquad +
    p^{\rm Dir}_{\varepsilon} (u, v_2) \times \big( \varepsilon^{- 1/4} | v_2 - u | +
    \varepsilon^{1/4} \big)\Big] \bigg)^2\\
    & \overset{\hphantom{\text{Lem. \ref{lem:g-bound}}}}{\lesssim}   \int_0^1 d v_2 \bigg( \| V^{\varepsilon} \|_{L^{\infty}}
    \bigg[ \bigg( \frac{\varepsilon^{1/2}}{v_2 (1 - v_2)} \wedge 1 \bigg) +
    \varepsilon^{1/4} \bigg) \bigg]^2 \lesssim \varepsilon^{1/2},
  \end{align*}
  which also vanishes as $\varepsilon \rightarrow 0$.
\end{proof}
\endgroup

This concludes the proof of Lemma \ref{lem:remainder-1}.

\subsubsection{Proof of Lemma~\ref{lem:remainder-2}}

Recall that $Q^{\varepsilon} (\varphi)$ was defined in~\eqref{eq:remainder-2}. One term that appears in its definition is
\[ \| J^{\varepsilon}_u \|_{L^2}^2 - \| p^{\rm Dir}_{\varepsilon} (u, \cdot)
   \|_{L^2}^2 = \sum_{\ell = 1}^{\infty} e^{- 2\varepsilon \pi^2 \ell^2} 2
   \big(\cos (\ell \pi u)^2 - \sin (\ell \pi u)^2\big) = \sum_{\ell = 1}^{\infty} e^{-2
   \varepsilon \pi^2 \ell^2} 2 \cos (2 \ell \pi u), \]
and for $u \in [0, \frac12]$ we have
\[ \sum_{\ell = 1}^{\infty} e^{- 2\varepsilon \pi^2 \ell^2} 2 \cos (2 \ell \pi
   u) = p^{\rm Neu}_{2\varepsilon} (0, 2 u) - 1, \]
while for $u \in (\frac12, 1]$
\begin{align*}
  \sum_{\ell = 1}^{\infty} e^{- 2\varepsilon \pi^2 \ell^2} 2 \cos \big(2 \ell \pi
  \big(\tfrac12 + (u - \tfrac12)\big)\big) & = \sum_{\ell = 1}^{\infty} e^{-2 \varepsilon \pi^2
  \ell^2} 2 \cos \big(\ell \pi \big[2 - (1 - 2 (u - \tfrac12))\big]\big)\\
  & = p^{\rm Neu}_{2\varepsilon} (0, 2 - 2 u) - 1.
\end{align*}
If $\{f^{\varepsilon}\}_{\varepsilon>0}$ is a uniformly bounded family of continuous
functions that converges to a continuous function $f$, then
\[ \lim_{\varepsilon \rightarrow 0} \bigg[ \int_0^{1/2} p^{\rm Neu}_{2\varepsilon}
   (0, 2 u)  f^{\varepsilon} (u) d u + \int_{1/2}^1
   p^{\rm Neu}_{2\varepsilon} (0, 2 - 2 u)  f^{\varepsilon} (u) d u \bigg] =
   \frac{1}{2} \big(f (0) + f (1)\big), \]
which means that for $\varphi \in \mc C^1([0,1])$
\begin{align*}
  Q^{\varepsilon}_t (\varphi) & = \int_0^t
  \bigg[ \int_0^1 \Psi^{\varepsilon}_s (u) \left( \big\| J^{\varepsilon}_u
  \big\|_{L^2}^2 - \big\| p^{\rm Dir}_{\varepsilon} (u, \cdot) \big\|_{L^2}^2 + 1\right) \varphi
  (u) d u  \\ & \qquad \qquad  \qquad  - \frac{1}{2} \big(\Psi^\varepsilon_s (0) \varphi (0) + \Psi^\varepsilon_s (1) \varphi (1)\big) \\
  & \qquad \qquad  \qquad - E^{- 2} \big(\nabla \varphi (1) \Psi^{\varepsilon}_s (1) -
  \nabla \varphi (0) \Psi^{\varepsilon}_s (0)\big) \bigg] d s\\
  & \xrightarrow{\varepsilon \to 0} \int_0^t \Big[  - E^{- 2} \big(\nabla \varphi (1) \Psi_s (1) - \nabla \varphi
  (0) \Psi_s (0)\big)\Big] d s,
\end{align*}
and if $\nabla \varphi(0) = \nabla \varphi(1) = 0$, the right hand side vanishes. Moreover, since the integrand in the time integral in the definition of $Q^\varepsilon(\varphi)$ converges absolutely, then $\{Q^\varepsilon(\varphi)\}_{\varepsilon \in (0,1]}$ is uniformly bounded in 1-variation norm.

\subsubsection{Proof of Lemma~\ref{lem:constant-conv}}

This is a consequence of the following result, which uses the approximation $K_u^{\varepsilon,\delta}$ defined in \eqref{eq:K-eps-delta-def}.

\begin{lemma}\label{lem:Kdelta}
  We have
  \[ \lim_{\varepsilon \rightarrow 0} \lim_{\delta \rightarrow 0}
     K^{\varepsilon, \delta}_u = \frac{E^2}{12}, \]
  where the convergence is in $L_u^2 ([0, 1])$.
\end{lemma}

\begin{proof}[Proof of Lemma \ref{lem:Kdelta}]
  Recall that, by definition
  \begin{align*}
    \lim_{\delta \rightarrow 0} E^{- 2} K^{\varepsilon, \delta}_u & =
    \lim_{\delta \rightarrow 0} \int_{[0, 1]^2} g^{\varepsilon, \delta}_u
    (v_1, v_2) \Theta^{\varepsilon}_u (v_1) \Theta^{\varepsilon}_u (v_2)
    d v_1 d v_2\\
    & = \lim_{\delta \rightarrow 0} \int_{[0, 1]^3} p^{\rm Dir}_{\delta} (w, v_1)
    p^{\rm Dir}_{\delta} (w, v_2) p^{\rm Neu}_{\varepsilon} (u, w) \Theta^{\varepsilon}_u
    (v_1) \Theta^{\varepsilon}_u (v_2) d w d v_1 d v_2\\
    &\quad - \int_{[0, 1]^2} p^{\rm Dir}_{\varepsilon} (u, v_1) p^{\rm Dir}_{\varepsilon} (u, v_2)
    \Theta^{\varepsilon}_u (v_1) \Theta^{\varepsilon}_u (v_2) d v_1
    d v_2\\
    & = \int_0^1 p^{\rm Neu}_{\varepsilon} (u, w) (\Theta^{\varepsilon}_u (w))^2 d w
    - \left( \int_0^1 p^{\rm Dir}_{\varepsilon} (u, v_1) \Theta^{\varepsilon}_u (v_1)
    d v_1 \right)^2,
  \end{align*}
  where the convergence is pointwise in $u$, with a uniform bound in
  $\delta$, so in particular in $L_u^2 ([0, 1])$. Since $- \nabla
  \Theta^{\varepsilon}_u (w) = p^{\rm Neu}_{\varepsilon} (u, w) - 1$, the first term
  in the last line is
  \begin{align*} \int_0^1 p^{\rm Neu}_{\varepsilon} (u, w) (\Theta^{\varepsilon}_u (w))^2 d w& =
     - \frac{1}{3} \int_0^1 \nabla (\Theta^{\varepsilon}_u (w))^3 d w +
     \int_0^1 (\Theta^{\varepsilon}_u (w))^2 d w \\ & = \int_0^1
     (\Theta^{\varepsilon}_u (w))^2 d w, \end{align*}
  where we used that $\Theta^{\varepsilon}_u (0) = \Theta^{\varepsilon}_u (1)
  = 0$ because $\Theta^\varepsilon_u \in \SD$. For the remaining term we get
  \[ \int_0^1 p^{\rm Dir}_{\varepsilon} (u, v_1) \Theta^{\varepsilon}_u (v_1) d
     v_1 = \sum_{\ell = 1}^{\infty} (e^{- \varepsilon \pi^2 \ell^2})^2 (- \ell
     \pi)^{- 1} 2 \sin (\ell \pi u) \cos (\ell \pi u) = \Theta^{2
     \varepsilon}_u (u) . \]
  In other words, we obtain $\lim_{\delta \rightarrow 0} E^{- 2}
  K^{\varepsilon, \delta}_u = \| \Theta^{\varepsilon}_u \|_{L^2}^2 - |
  \Theta^{2 \varepsilon}_u (u) |^2$. Now note that $\Theta^{\varepsilon}_u (v)$ also writes as \[\Theta^\varepsilon_u(v)
  = \int_0^1 \Theta_u (w) p^{\rm Dir}_{\varepsilon} (v, w) d w,\] where $\Theta_u
  (w)$ is the unique kernel satisfying 
  $\nabla_u \langle \Theta_u, f \rangle_{L^2} = f (u)$ and  $\int_0^1 \langle
  \Theta_u, f \rangle_{L^2} d u = 0$ for all $f \in \mc C([0,1])$. From the first condition we get $\Theta_u
  (w) =\mathbf{1}_{[0, u]} (w) + g(w)$ for some $g$, the second
  condition gives for almost all $w$
  \[ 0 = \int_0^1 \Theta_u (w) d u = \int_0^1 \mathbf{1}_{[w,1]}(u)\;
     d u + g(w) = 1 - w + g(w), \]
  so $g (w) = w - 1$. Therefore,
  \[ \Theta^{\varepsilon}_u (u) = \int_0^1 \Theta_u (w) p^{\rm Dir}_{\varepsilon} (u,
     w) d w =\mathbb{E}_u^{\rm Dir} \big[\mathbf{1}_{[0, u]} (B_{2 \varepsilon}) +
     B_{2 \varepsilon} - 1\big], \]
  where $\mathbb{E}^{\rm Dir}_u$ is the measure under which $B$ is a Brownian motion
  started at $u$, killed when reaching $0$ or $1$, and consequently $\lim_{\varepsilon \rightarrow 0} \Theta^{\varepsilon}_u (u) = \frac{1}{2} + u - 1 = u - \frac{1}{2}$ for all $u \in (0,1)$.  Also,
  \begin{equation*}
    \lim_{\varepsilon \rightarrow 0} \| \Theta^{\varepsilon}_u \|_{L^2}^2 = \|
    \Theta_u \|_{L^2}^2 = \int_0^1 (\mathbf{1}_{[0, u]} (w) + w - 1)^2 \; dw = u^2 - u + \tfrac{1}{3},
  \end{equation*}
  which leads to
  \[ \lim_{\varepsilon \rightarrow 0} \lim_{\delta \rightarrow 0} E^{-
     2} K^{\varepsilon, \delta}_u = \| \Theta^{\varepsilon}_u\|_{L^2}^2 - |
     \Theta^{2 \varepsilon}_u (u) |^2 = u^2 - u + \tfrac{1}{3} - \left( u -
     \tfrac{1}{2} \right)^2  = \tfrac{1}{12} . \]
  
\end{proof}

\appendix

\section{Auxiliary computations involving the generator}
\label{sec:app}

\subsection{Martingale decomposition for the density fluctuation field $\mc Y_t^n$}\label{app:mart-dens}

In this section we provide  the technical computations that we need to prove \eqref{eq:martdyn}, assuming $\rho=\frac12$. We consider the general case $\gamma \geq \frac12$, and we aim at computing $\int_0^t n^2\mc L_n \mc Y_s^n(\varphi)ds$ when $\varphi \in \SD$.

First, we note that  for any $x\in\Lambda_n$, 
$
\mathcal L_n \eta(x)=j_{x-1,x}(\eta)-j_{x,x+1}(\eta), 
$
where the local function $j_{x,x+1}$ is the microscopic current of the system, which can be decomposed into its symmetric and antisymmetric parts as
$$j_{x,x+1}(\eta)=j^{\rm s}_{x,x+1}(\eta)+j^{\rm a}_{x,x+1}(\eta)$$
with \begin{align*} j^{\rm s}_{x,x+1}(\eta) =\eta(x)-\eta(x+1),  \quad j^{\rm a}_{x,x+1}(\eta)=\frac{E}{n^\gamma}\eta(x)(1-\eta(x+1)).\end{align*}

Recall that above, as well as everywhere it appears, we assume by convention $\eta(0)=\eta(n)=\rho$.

We start by looking at the action of the symmetric part of the current in the density field. In the following, for $\varphi \in \SD$, and $x\in\Lambda_n$, we denote $\varphi_x=\varphi(\frac x n)$.
A simple computation, which makes use of the property of the test function $\varphi$ at the boundary, namely $\varphi_0=\varphi_n=0$,  shows that
\begin{multline}
\frac{n^2}{\sqrt n}\sum_{x=1}^{n-1}\varphi_x \big\{j_{x-1,x}^{\rm s}(\eta)-j_{x,x+1}^{\rm s}(\eta)\big\} \\
 =\frac{1}{\sqrt n}\sum_{x=1}^{n-1}\Delta_n\varphi_x\;\eta(x)+ \frac{n^2}{\sqrt n}\varphi_1\;\eta(0)+\frac{n^2}{\sqrt n}\varphi_{n-1}\;\eta(n).  \label{eq:eq}
\end{multline}
Since we need to close the previous equation in terms of the density field,  we note that the identity $ \sum_{x=1}^{n-1}\Delta_n\varphi_x=-n^2(\varphi_1+\varphi_{n-1})$  implies that: 
\begin{equation}\label{eq:WASEP-IBP}
 \frac{1}{\sqrt n}\sum_{x=1}^{n-1}\Delta_n\varphi_x\;\bar\eta(x)=
 \frac{1}{\sqrt n}\sum_{x=1}^{n-1}\Delta_n\varphi_x(\eta(x)-\rho)=\mathcal{Y}^n
 (\Delta_n\varphi).
\end{equation}
Now we look at the action of the antisymmetric part of the current in the density field. We use the fact that, when $\rho=\frac12$,  for any $x\in\Lambda_n$ we have the identity
\begin{equation} \eta(x-1)\big(1-\eta(x)\big)= - \bar\eta(x)\bar\eta(x-1) + \frac{1}{2}\big(\eta(x-1)-\eta(x)\big) + \frac{1}{4}.\label{eq:ide}\end{equation}
A simple computation shows that 
\begin{align}
&\frac{n^2}{\sqrt n}\sum_{x=1}^{n-1}\varphi_x \big\{j_{x-1,x}^{\rm a}(\eta)-j_{x,x+1}^{\rm a}(\eta)\big\}\label{eqdecompp} \\ & \quad =-\frac{1\sqrt n}{n^{\gamma}}\; \sum_{x=1}^{n-2}\nabla_n^+\varphi_x\; \bar{\eta}(x)\bar{\eta}(x+1) + \frac{E\sqrt n}{2n^{\gamma}}\; \sum_{x=0}^{n-1}\nabla_n^+\varphi_x\;(\eta(x)-\eta(x+1))  + \textrm{BC}_n(\varphi),   \notag \end{align}
where $\textrm{BC}_n(\varphi)$ is a boundary term that is given by
\[
\textrm{BC}_n(\varphi) = \frac{En^{3/2}}{2n^{\gamma}} \Big(-\varphi_n\;\bar\eta(n-1) - \varphi_0\; \bar\eta(1)\Big)
\]
Recall that $\varphi_n=\varphi_0=0$. Therefore, \eqref{eqdecompp} rewrites as 
\begin{multline}
 -E \; \frac{\sqrt n}{n^\gamma}\; \sum_{x=1}^{n-2}\nabla_n^+\varphi_x\; \bar{\eta}(x)\bar{\eta}(x+1) +  \frac{E}{2n^{\gamma+1/2}} \sum_{x=1}^{n-1}\Delta_n\varphi_x\:\bar\eta(x) \\
 = -E \; \frac{\sqrt n}{n^\gamma}\; \sum_{x=1}^{n-2}\nabla_n^+\varphi_x\; \bar{\eta}(x)\bar{\eta}(x+1) + \frac{E}{2n^\gamma}\; \mc Y^n(\Delta_n \varphi) .\label{eq:WASEP2} \end{multline}
 From \eqref{eq:WASEP-IBP} and \eqref{eq:WASEP2} we easily deduce \eqref{eq:martdyn}. 
 
 \subsection{Martingale decomposition for the height fluctuation field $\mc Z_t^n$}
 \label{app:mart_dec_height}

Now let us show \eqref{mart_decomp_height} in a very similar way as we did in Appendix \ref{app:mart-dens}.  Recall that $h(x)=h(1)+\sum_{y=1}^{x-1} \bar\eta(y)$, as explained in Section \ref{ssec:height}. As before we always use the convention $\eta(0)=\eta(n)=\rho$ and we also conveniently set $h(0)=h(1)$ and $h(n+1)=h(n)$. One easily obtains, for any $x\in\{1,\dots,n\}$:
\begin{align} 
\mc L_n^\otimes h(x) & =  - j_{x - 1, x} (\eta) = \eta (x) (1 - \eta (x - 1)) -
  \Big( 1 + \frac{E}{n^{\gamma}} \Big) \eta (x - 1) (1 - \eta (x)) \notag\\
  & = \Delta h(x) - \frac{E}{n^{\gamma}} \eta (x - 1) (1 - \eta (x)), \label{eq:gen-h}
\end{align}
where we wrote $\Delta h (x) = h (x + 1) + h (x - 1) - 2 h (x)$. Note that with our convention $\Delta h(1)=\bar \eta(1)$ and $\Delta h(n)=-\bar\eta(n-1)$. In \eqref{eq:gen-h} we want to rewrite the second term in terms of the height configuration values. This can be easily done for $\rho=\frac12$ using the identity \eqref{eq:ide}, which we rewrite here as:
\[- \eta (x - 1) (1 - \eta (x)) = \nabla^- h (x) \nabla^+ h (x) + \frac{1}{2} \Delta h (x) - \frac{1}{4},\]
where $\nabla^-h(x)=h(x)-h(x-1)$ and $\nabla^+h(x)=h(x+1)-h(x)$. Note that the previous identity does hold for $x=1$ and $x=n$ since with our convention, $\nabla^-h(1)=0$ and $\nabla^+h(n)=0$. This implies
\begin{equation}\mc L_n^\otimes h(x) = \Big(1+\frac{E}{2n^\gamma}\Big)\Delta h(x) + \frac{E}{n^\gamma} \Big( \nabla^- h (x) \nabla^+ h (x) - \frac14\Big), \quad \text{for any } x\in\{1,\dots,n\}.\label{eq:gen-h-2}\end{equation}
Now, let us take $\varphi \in \SN$ and as before denote $\varphi_x=\varphi(\frac x n)$. We start by treating the first term in \eqref{eq:gen-h-2}, and more precisely $\Delta h(x)$. A simple computation shows that
\begin{equation}\label{eq:gen-lap}
\frac{n^2}{n^{3/2}} \sum_{x=1}^n \varphi_x \Delta h(x)  = \frac{1}{n^{3/2}}\sum_{x=1}^n \tilde\Delta_n\varphi_x\; h(x) + \sqrt n \big(\varphi_1-\varphi_0)h(1)+\sqrt n(\varphi_n-\varphi_{n-1})h(n), \end{equation}
where $\tilde\Delta_n\varphi_x$ has been defined in \eqref{eq:lap2} as: 
\[\tilde\Delta_n\varphi_x = \begin{cases}n^2\big(\varphi_{x+1} + \varphi_{x-1} - 2 \varphi_x\big)= \Delta_n\varphi_x  & \text{ if } x \in \{1,\dots,n-1\},\\
2 n^2\big(\varphi_{n-1}-\varphi_n\big) & \text{ if }x=n.\end{cases}\]
Since our goal is to see the height fluctuation field appear, as long as $h(x)$ is replaced with the time dependent configuration $h_{sn^2}^n(x)$, in the first term of \eqref{eq:gen-lap}  we have to recenter the heights  as follows: 
\begin{align}
\frac{n^2}{n^{3/2}} \sum_{x=1}^n &\varphi_x \Delta h_{sn^2}^n(x)   = \frac{1}{n^{3/2}}\sum_{x=1}^n \tilde\Delta_n\varphi_x\; \big(h_{sn^2}^n(x)-c_n s\big) \label{eq:hcontr1}\\
& \quad + \sqrt n\big(\varphi_1-\varphi_0\big)\big(h_{sn^2}^n(1)-c_n s\big) + \sqrt n\big(\varphi_n-\varphi_{n-1}\big)\big(h_{sn^2}^n(n)-c_ns\big) \label{eq:hcontr2}\\
& \quad + \frac{n^2}{n^{3/2}} \frac{E}{n^\gamma} \big(\varphi_0-\varphi_1 + \varphi_{n-1}-\varphi_n\big) c_n s.\label{eq:hcontr3}
\end{align}
Note that the term \eqref{eq:hcontr1} corresponds to $\mc Z_s^n(\tilde\Delta_n\varphi)$,  while the term \eqref{eq:hcontr2}, when integrated in time between $0$ and $t$, will give the contribution $\mc R_t^n(\varphi)$ (see \eqref{eq:height-remainder_new}). It remains to show that \eqref{eq:hcontr3} vanishes in the $n\to\infty$ limit. This is a consequence of the fact that $\varphi \in \SN$ and therefore we have
\[ 
n^2\big(\varphi_0-\varphi_1\big) \xrightarrow[n\to\infty]{} -\Delta\varphi(0), \qquad n^2\big(\varphi_{n-1}-\varphi_n\big) \xrightarrow[n\to\infty]{} - \Delta\varphi(1).
\]
As a consequence, \eqref{eq:hcontr3} is of the same order as 
\[ 
\frac{c_n}{n^{\gamma+3/2}} = - \frac{En^2}{4n^{2\gamma+3/2}},
\] which vanishes for any $\gamma \geq \frac12$. We are left with the second term in \eqref{eq:gen-h-2}, that we put directly into the martingale decomposition as follows: from the computation above, and from Dynkin's formula, the following quantity is a martingale:
\begin{align} 
\mc Z_t^n(\varphi) &  -\mc Z_0^n(\varphi)- \int_0^t \Big(n^2 \mc L_n^\otimes \mc Z_s^n(\varphi) - \frac{c_n}{n^{3/2}}\sum_{x=1}^n\varphi_x\Big) ds \notag \\
\notag
 & = \mc Z_t^n(\varphi) - \mc Z_0^n(\varphi) - \Big(1+\frac{E}{2n^\gamma} \Big)\int_0^t \mc Z_s^n(\tilde\Delta_n\varphi)ds + \mc R_t^n(\varphi) +o^n_t(1)\\
& \quad - \frac{E}{n^\gamma} \frac{n^2}{n^{3/2}} \int_0^t \sum_{x=1}^{n} \Big(\nabla^-h_{sn^2}^n(x)\nabla^+h_{sn^2}^n(x)-\frac14\Big)\varphi_x ds \label{eq:cancel1}\\
& \quad + \frac{n^2c_n t}{n^{3/2}} \sum_{x=1}^n \varphi_x,\label{eq:cancel2}
\end{align}
where the last term \eqref{eq:cancel2} cancels out 
with the term coming with $(-\frac14)$ in \eqref{eq:cancel1} (since $c_n = -En^2/(4n^\gamma)$), and then \eqref{mart_decomp_height} follows.

\section{Proofs of Proposition \ref{prop:boundary} and Proposition \ref{prop:CH-boundary}: boundary behavior}

\label{sec:proof-boundary}

As explained in Section~\ref{sec:energy-pr}, we may assume without loss of generality that $A=1$ and $D=2$, and recall that we denote $\bar E=E$. Therefore, in Proposition \ref{prop:CH-boundary} we have $\frac{D(\bar E)^2}{4A^3}=\frac{E^2}{2}$.

\begin{proof}[Proof of Proposition \ref{prop:boundary}] Let $\mc Y$ and $\rho_\varepsilon$ be as in the assumptions of Proposition \ref{prop:boundary}.
  The map $\mc Y(\rho_\varepsilon) = W_1 (\rho_\varepsilon) (\mc Y)$ is in the first chaos, so Lemma~\ref{lem:ito} and Corollary~\ref{cor:poisson-solution norm} give
  \begin{align*}
  	 \mathbb{E} \bigg[ \sup_{t \in [0,T]}
     \bigg| \int_0^t \mc Y_s (\rho_{\varepsilon}) d s \bigg|^p \bigg]
     & \lesssim  T^{p / 2} \big\| \rho_{\varepsilon}
     \big\|_{\mc H_{\rm Dir}^{- 1}}^p \\
     &  \lesssim T^{p / 2} \bigg( \sup_{\substack{f \in \mc H_{\rm Dir}^1\\ \| f\|_{\mc H_{\rm Dir}^{1}}=1}} \int_0^1 \rho^\varepsilon(u) (f(u) -f(0)) du \bigg)^p\\
     & \le T^{p / 2} \bigg( \sup_{\substack{f \in \mc H_{\rm Dir}^1\\ \| f\|_{\mc H_{\rm Dir}^{1}}=1}} \int_0^1 \rho^\varepsilon(u) \; \sqrt u \; \|f\|_{\mc H_{\rm Dir}^{1}} du \bigg)^p \\
     & =  T^{p / 2} \bigg(\int_0^1 \rho^\varepsilon(u) \sqrt u du\bigg)^p,
  \end{align*}
  where in the second step we used that $f(0)=0$ and in the third step we applied the fundamental theorem of calculus and the Cauchy-Schwarz inequality. By assumption, the right hand side converges to zero as $\varepsilon \to 0$.
\end{proof}
\begin{remark}
In the case $p=2$ and for  $\rho_\varepsilon = \iota_{\varepsilon}(0)$ (resp.~$\rho_\varepsilon = \iota_\varepsilon(1)$), with $\iota_\varepsilon$ as in Definition~\ref{def:approx}, the previous result can also be obtained from the microscopic dynamics by a two--steps procedure. We just sketch the idea and we leave the details to the reader. Note that $\mathcal{Y}_s^n(\iota_{\varepsilon}(0))=\sqrt n \vec\eta_{sn^2}^{\varepsilon n}(0)$. Adapting \cite[Lemma 3]{Franco2016}  we know how to control the variance of $\int_0^t \sqrt n (\eta_{sn^2}(1)-\rho) \, ds$ and adapting \cite[Lemma 7.1]{FRANCO20134156} we can control the variance of $\int_0^t \sqrt n (\eta_{sn^2}(1)-\vec\eta_{sn^2}^{\varepsilon n}(0))\ ds$. This together with the convergence of the fluctuation field is enough to conclude.
\end{remark}

\begin{proof}[Proof of Proposition \ref{prop:CH-boundary}]
   Let us use the same mapping argument as in Section \ref{ssec:map}. Let  $\Psi_s^\delta(u) = e^{E \mc Y_s(\Theta^\delta_u)}$ with $\Theta^\delta_u(v) = \int_0^1 \Theta_u(w) p_\delta^{\rm Dir}(v,w) dw$ and where $p_\delta^{\rm Dir}$ denotes the Dirichlet heat kernel as defined in \eqref{eq:dirker}. Then
   \begin{align}\label{eq:CH-boundary-pr1} \nonumber
      \mathbb{E} \bigg[  \bigg| \int_0^t &\bigg( \frac{\Psi_s (\varepsilon) - \Psi_s(0)}{\varepsilon}  + \frac{E^2}{2} \Psi_s(0) \bigg) d s \bigg|^2 \bigg] \\ \nonumber
      & = \lim_{\delta \to 0} \mathbb{E} \bigg[  \bigg| \int_0^t \bigg( \frac{\Psi^\delta_s (\varepsilon) - \Psi^\delta_s(0)}{\varepsilon}  + \frac{E^2}{2} \Psi_s(0) \bigg) d s \bigg|^2 \bigg] \\ \nonumber
      & \leq 2 \lim_{\delta \to 0}\mathbb{E} \bigg[\bigg| \int_0^t \bigg( \frac{\Psi^\delta_s (\varepsilon) - \Psi^\delta_s(0)}{\varepsilon}  - \frac{E^2}{\varepsilon} \int_0^\varepsilon  \Psi^\delta_s(u) \Theta^{2\delta}_u(u) du \bigg) d s \bigg|^2 \bigg] \\
      & \quad + 2 \lim_{\delta \to 0} \mathbb{E} \bigg[ \bigg| \int_0^t \bigg( \frac{E^2}{\varepsilon} \int_0^\varepsilon  \Psi^\delta_s(u) \Theta^{2\delta}_u(u) du + \frac{E^2}{2} \Psi_s(0) \bigg) d s \bigg|^2 \bigg].
   \end{align}
   The Kipnis-Varadhan inequality (Corollary~\ref{cor:KV}), yields for the first term
   \begin{align*}
       \mathbb{E} \bigg[ \sup_{t \in [0,T]} &\bigg| \int_0^t \bigg(\frac{\Psi^\delta_s (\varepsilon) - \Psi^\delta_s(0)}{\varepsilon}  - \frac{E^2}{\varepsilon} \int_0^\varepsilon  \Psi^\delta_s(u) \Theta^{2\delta}_u(u) du \bigg)d s \bigg|^2 \bigg] \\
      & \lesssim T \bigg\| \frac{\Psi^\delta_0 (\varepsilon) - \Psi^\delta_0(0)}{\varepsilon}  - \frac{E^2}{\varepsilon} \int_0^\varepsilon  \Psi^\delta_0(u) \Theta^{2\delta}_u(u) du \bigg\|_{-1,0}^2 \\
      & \lesssim T \bigg\| \frac{1}{\varepsilon} \int_0^\varepsilon \Psi^\delta_0 (u) E \mc Y_0(p^{\rm Dir}_\delta(u,\cdot)) du - \frac{E^2}{\varepsilon} \int_0^\varepsilon  \Psi^\delta_0(u) \Theta^{2\delta}_u(u) du \bigg\|_{-1,0}^2,
   \end{align*}
   where in the last step we used the fundamental theorem of calculus and the identity $\nabla_u \mc Y_0(\Theta^\delta_u) = \mc Y_0(p^{\rm Dir}_\delta(u,\cdot))$. Let now $F \in \mc C$ be a cylinder function. Then Gaussian integration by parts~\cite[Lemma~1.2.1]{Nualart2006} yields
   \begin{align*}
      \bb E\big[ &e^{E \mc Y_0(\Theta^\delta_u)}  \mc Y_0(p^{\rm Dir}_\delta( u, \cdot)) F(\mc Y_0) \big] \\
      &\quad = \int_0^1 \bb E\Big[  p^{\rm Dir}_\delta(u,v) \big(e^{E \mc Y_0(\Theta^\delta_u)}  \mr D_vF(\mc Y_0) + F(\mc Y_0) \mr D_v e^{E \mc Y_0(\Theta^\delta_u)} \big) \Big] dv \\
      &\quad = \int_0^1 \bb E\Big[  p^{\rm Dir}_\delta(u,v) \big(e^{E \mc Y_0(\Theta^\delta_u)}  \mr D_vF(\mc Y_0) + F(\mc Y_0)  e^{E \mc Y_0(\Theta^\delta_u)} E \Theta^\delta_u(v) \big) \Big] dv \\
      &\quad =  \bb E\bigg[  \int_0^1 p^{\rm Dir}_\delta(u,v) e^{E \mc Y_0(\Theta^\delta_u)}  \mr D_vF(\mc Y_0) dv+ F(\mc Y_0)  e^{E \mc Y_0(\Theta^\delta_u)} E \Theta^{2\delta}_u(u)  \bigg],
   \end{align*}
   and therefore
   \begin{align*}
      &\bb E\bigg[ \bigg(\frac{1}{\varepsilon} \int_0^\varepsilon \Psi^\delta_0 (u) E \mc Y_0(p^{\rm Dir}_\delta(u,\cdot)) du - \frac{E^2}{\varepsilon} \int_0^\varepsilon  \Psi^\delta_0(u) \Theta^{2\delta}_u(u) du\bigg) F(\mc Y_0) \bigg] \\
      &\quad = \bb E\bigg[  \int_0^1  \frac{E}{\varepsilon} \int_0^\varepsilon p^{\rm Dir}_\delta(u,v) e^{E \mc Y_0(\Theta^\delta_u)} du\;  \mr D_vF(\mc Y_0) dv \bigg] \\
      &\quad \overset{\substack{\text{duality}\\+(C-S)}}{\leq} \bb E\bigg[ \bigg\|  \frac{E}{\varepsilon} \int_0^\varepsilon p^{\rm Dir}_\delta(u,\cdot) e^{E \mc Y_0(\Theta^\delta_u)} du \bigg\|_{\mc H_{\rm Dir}^{- 1}}^2 \bigg]^{1/2} \| F\|_{1,0} \\
      &\quad \overset{(C-S)} \le  \bb E\bigg[   \frac{E^2}{\varepsilon} \int_0^\varepsilon \big\| p^{\rm Dir}_\delta(u,\cdot)\big\|_{\mc H_{\rm Dir}^{- 1}}^2 \; e^{2E \mc Y_0(\Theta^\delta_u)} du  \bigg]^{1/2} \| F\|_{1,0} \lesssim \sup_{u \in (0,\varepsilon]} \big\| p^{\rm Dir}_\delta(u,\cdot)\big\|_{\mc H_{\rm Dir}^{- 1}} \| F \|_{1,0},
   \end{align*}
  which leads to the estimate
   \[
      \lim_{\delta \to 0} \bigg\| \frac{\Psi^\delta (\varepsilon) - \Psi^\delta(0)}{\varepsilon}  - \frac{E^2}{\varepsilon} \int_0^\varepsilon  \Psi^\delta(u) \Theta^{2\delta}_u(u) du \bigg\|_{-1,0} \lesssim \limsup_{\delta \to 0} \sup_{u \in (0,\varepsilon]} \big\| p^{\rm Dir}_\delta(u,\cdot)\big\|_{\mc H_{\rm Dir}^{- 1}},
   \]
   and as in the proof of Proposition~\ref{prop:boundary} we see that the right hand side converges to zero for $\varepsilon \to 0$. To treat the second term in~\eqref{eq:CH-boundary-pr1} recall that we showed in the proof of Lemma~\ref{lem:Kdelta} that $\lim_{\delta \rightarrow 0} \Theta^{2\delta}_u (u) = u - \frac{1}{2}$ for all $u \in (0,1)$. Therefore,
   \begin{multline*}
\lim_{\varepsilon \to 0} \lim_{\delta \to 0}\mathbb{E} \bigg[\sup_{t \in [0,T]}  \bigg| \int_0^t \bigg( \frac{E^2}{\varepsilon} \int_0^\varepsilon  \Psi^\delta_s(u) \Theta^{2\delta}_u(u) du + \frac{E^2}{2} \Psi_s(0) \bigg) d s \bigg|^2 \bigg] \\
 \le \lim_{\varepsilon \to 0}  \mathbb{E} \bigg[  \bigg(\int_0^T\bigg|  \frac{E^2}{\varepsilon} \int_0^\varepsilon  \Psi_s(u) \big(u - \tfrac{1}{2}\big) du + \frac{E^2}{2} \Psi_s(0)\bigg|  d s \bigg)^2 \bigg] = 0
   \end{multline*}
   by Lebesgue's differentiation Theorem. The same arguments give the boundary behavior at $u = 1$.
\end{proof}

\section{Heat kernel estimates}\label{app:heat-kernel}

Here we collect basic estimates for the Dirichlet and Neumann heat kernels on $[0,1]$,  namely $p^{\rm Dir}$ and $p^{\rm Neu}$, that we already defined respectively in \eqref{eq:dirker} and \eqref{eq:neuker}.

\begin{lemma}
  \label{lem:Dirichlet-heat-kernel}
For all $\lambda > - 1$ and $t \in (0, 1]$
  \begin{align*}
      &\sup_{u, v \in [0, 1]} | p^{\rm Dir}_t (u, v) | \lesssim t^{- 1 / 2}, \qquad
     \sup_{u \in [0, 1]} \big\| p^{\rm Dir}_t (u, \cdot)\times | u - \cdot |^{\lambda}
     \big\|_{L^1} \lesssim t^{\lambda / 2}, \\
     & \vphantom{\int}\quad \big\| \langle
     p_t^{\rm Dir} (u, \cdot), 1 \rangle - 1 \big\|_{L^2_u} \lesssim t^{1 / 4} .
  \end{align*}
  The same bounds also hold for the Neumann heat kernel $p^{\rm Neu}$.
\end{lemma}

\begin{proof}[Proof of Lemma \ref{lem:Dirichlet-heat-kernel}]
  The Dirichlet heat kernel is the transition density of $\{B_{2 t} \; : \; t
  \geq 0\}$ where $B$ is a Brownian motion that is killed when it reaches
  $0$ or $1$. In particular it is bounded
  from above by the transition density of the Brownian motion, i.e.
  \[ 0 \leq p^{\rm Dir}_t (u, v) \lesssim t^{- 1 / 2} e^{- (u - v)^2 / 4 t}, \]
  from where the first estimate follows. This also gives for $\lambda > - 1$
  \begin{align*} \int_0^1 p^{\rm Dir}_t (u, v) | u - v |^{\lambda} d v &\lesssim t^{\lambda /
     2} \int_0^1 t^{- 1 / 2} \; e^{- \frac{1}{4} \left( \frac{u - v}{\sqrt{t}}
     \right)^2} \bigg| \frac{u - v}{\sqrt{t}} \bigg|^{\lambda} d v \\& =
     t^{\lambda / 2} \int_{\mathbb{R}} e^{- \frac{1}{4} u^2} | u |^{\lambda}
     d u \lesssim t^{\lambda / 2} . \end{align*}
  To estimate the $L^2$--norm, note that if $B$ is a standard Brownian motion
  we have
    \begingroup
\allowdisplaybreaks
  \begin{align*}
    1 & \geq \langle p_t^{\rm Dir} (u, \cdot), 1 \rangle =\mathbb{P} \big(u + B_{2
    s} \in [0, 1], \; \forall \; s \in [0, t]\big) \vphantom{\int} \\
    & \geq \mathbb{P} \big(\big\{ \sup_{s \in
    [0, 2 t]} B_s < 1 - u \big\} \cap \big\{ \inf_{s \in [0, 2t]} B_s > - u \big\}\big)\vphantom{\int}\\
    & \geq \mathbb{P} \big(\sup_{s \in [0, 2 t]} B_s < 1 - u\big) +\mathbb{P}
    \big(\inf_{s \in [0, 2t]} B_s > - u\big) - 1 \geq 2\mathbb{P}\big(\sup_{s \in
    [0, 2 t]} B_s < u \wedge (1 - u)\big) - 1\vphantom{\int}\\
    & = 2\mathbb{P} \big(| B_{2 t} | < u \wedge (1 - u)\big) - 1 \geq 2 \Big(1 -
    2\mathbb{P} \big(B_{2 t} \geq u \wedge (1 - u)\big)\Big) - 1\vphantom{\int}\\
    & = 1 - 4\mathbb{P} \big(B_{2 t} \geq u \wedge (1 - u)\big) \geq 1 - 4
    \frac{e^{- \frac{(u \wedge (1 - u))^2}{4 t}}}{\sqrt{2 \pi} \frac{u \wedge
    (1 - u)}{\sqrt{2 t}}},
  \end{align*}
  \endgroup
  where we used the reflection principle for the Brownian motion and standard
  tail estimates for the normal distribution. From here we get $| \langle
  p_t^{\rm Dir} (u, \cdot), 1 \rangle - 1 | \lesssim \frac{t^{1 / 2}}{u (1 - u)}
  \wedge 1$, which leads to
  \[ \big\| \langle p_t^{\rm Dir} (u, \cdot), 1 \rangle - 1 \big\|_{L^2_u}^2 \lesssim
     \int_0^{t^{1 / 2}} d u + \int_{t^{1 / 2}}^{\infty} \frac{t}{u^2}
     d u \lesssim t^{1 / 2}. \]
  For the Neumann heat kernel we have $\langle p^{\rm Neu}_t(u, \cdot) -1\rangle \equiv 1$, so the last bound is trivial. The remaining bounds for the Neumann kernel follow once we know that $p^{\rm Neu}_t (u, v) \lesssim t^{- 1 / 2} e^{- (u - v)^2 / 4 t}$ uniformly in $t \in (0,1]$, which is basically (3.7) in~\cite{Walsh1986}.
\end{proof}

\begin{lemma}
  \label{lem:Neumann-Dirichlet}The difference between Neumann and Dirichlet
  heat kernel is bounded in $L^1$:
  \[ \big\| p^{\rm Neu}_{\varepsilon} (u, \cdot) - p^{\rm Dir}_{\varepsilon} (u, \cdot)
     \big\|_{L^1} \lesssim \frac{\varepsilon^{1 / 2}}{u (1 - u)} \wedge 1, \]
     uniformly in $\varepsilon \in (0,1]$ and $u \in [0,1]$.
\end{lemma}

\begin{proof}[Proof of Lemma \ref{lem:Neumann-Dirichlet}]
  Note that $\| p^{\rm Neu}_{\varepsilon} (u, \cdot) - p_{\varepsilon}^{\rm Dir} (u,
  \cdot) \|_{L^1} \leq 2 \| p^{\rm Neu}_{\varepsilon} (u,
  \cdot) - p_{\varepsilon}^{\rm Dir} (u, \cdot) \|_{\tmop{TV}}$, where
  \[
     \| \nu \|_{\tmop{TV}} = \sup_{A \in \mc B([0,1])} |\nu(A)|
  \]
   denotes the total variation norm of the signed measure $\nu$ on the Borel sets $\mc B([0,1])$ of $[0,1]$. We know that $p^{\rm Neu}$
  (resp. $p^{\rm Dir}$) is the transition density of a Brownian motion that is
  reflected (resp. killed) in $0$ and $1$ -- both with speed $2$. We write
  $\mathbb{P}^{\rm Neu}_u$ respectively $\mathbb{P}^{\rm Dir}_u$ for the law of the
  reflected respectively killed Brownian motion with speed $2$, both started
  in $u$, while $\mathbb{P}_u$ is the law of the (usual) Brownian motion with
  speed 2, started in $u$. Then we have for all Borel sets $A \in \mathcal{B}
  (\mathbb{R})$
  \[ \mathbb{P}^{\rm Dir}_u \big(\big\{ B_{\varepsilon} \in A \big\} \cap \big\{ B_s \in [0, 1],\; \forall\; s
     \in [0, \varepsilon] \big\}\big) =\mathbb{P}^{\rm Neu}_u \big(\big\{ B_{\varepsilon} \in A \big\}
     \cap \big\{ B_s \in [0, 1], \; \forall \; s \in [0, \varepsilon] \big\}\big), \]
  and therefore
  \[
    \big| \mathbb{P}^{\rm Dir}_u (B_{\varepsilon} \in A) -\mathbb{P}^{\rm Neu}_u
    (B_{\varepsilon} \in A) \big| \leq 2\mathbb{P}_u \big(\big\{ \sup_{s \leq
    \varepsilon} B_s \geq 1 \big\} \cup \big\{ \inf_{s \leq \varepsilon} B_s
    \leq 0 \big\}\big) \lesssim \frac{\varepsilon^{1 / 2}}{u (1 - u)},
  \]
  where the last step follows as in the proof of
  Lemma~\ref{lem:Dirichlet-heat-kernel}. We now take the supremum in $A\in \mc B([0,1])$ and get a
  bound for the total variation norm and thus for the $L^1$--norm of the
  difference of the densities.
\end{proof}

\section{Proof of Proposition \ref{def:she-rob}}\label{app:she-proof}

Let $\Phi$ satisfy the conditions of Proposition~\ref{def:she-rob}. As in~\cite[Chapter~5]{BerG} we see that there exists a standard $\SN'$--valued Brownian motion with covariance~\eqref{eq:Bm-cov} (possibly on an extended probability space) such that for all $\varphi \in \SN$
	\begin{align}   \Phi_t(\varphi) & = \Phi_0(\varphi) + A \int_0^t  \Phi_s(\Delta \varphi) ds  + A \int_0^t (- \alpha \Phi_s(0) \varphi(0)+ \beta \Phi_s(1) \varphi(1)) ds  \nonumber \\
   &\quad  + \sqrt{D} \int_0^1 \int_0^t \Phi_s(u) \varphi(u) d \mc W_s(u) du,\label{eq:she-rob-strong}
\end{align}
and also that it suffices to show the \emph{strong uniqueness} of solutions to the equation driven by this given $\mc W$. The proof for the strong uniqueness is essentially the same as in~\cite[Theorem~3.2]{Walsh1986}, the only difference is that we have to deal with the additional terms coming from the Robin boundary condition. If $\Phi^i$, $i=1,2$ are solutions to~\eqref{eq:she-rob-strong}, then we obtain easily (see also~\cite[Exercise~3.1]{Walsh1986}) that for  all $t \in [0,T]$, $u\in [0,1]$, and $\varepsilon > 0$
   \begin{align*}
      \Phi_t^i(p^{\rm Neu}_\varepsilon(u)) & = A \int_0^1 \Phi_0(v) p^{\rm Neu}_{t+\varepsilon}(u,v)dv \\
      &\quad + A \int_0^t \big(- \alpha \Phi^i_s(0) p_{t-s+\varepsilon}^{\rm Neu}(u,0)+ \beta \Phi^i_s(1) p_{t-s+\varepsilon}^{\rm Neu}(u,1)\big) ds  \\
      &\quad  + \sqrt{D} \int_0^1 \int_0^t p^{\rm Neu}_{t-s+\varepsilon}(u,v) \Phi^i_s(v) d \mc W_s(v) dv,
   \end{align*}
   and with the $L^2$-continuity of $\Phi^i_t$ this extends to $\varepsilon = 0$ (with $p^{\rm Neu}_0(u,\cdot) = \delta_u(\cdot)$). Then the difference $\mc U = \Phi^1 - \Phi^2$ satisfies
   \begin{align*}
      \bb E[\mc U_t(u)^2] & \lesssim \bb E\Big[\Big|\int_0^t \big(- \alpha \mc U_s(0) p_{t-s}^{\rm Neu}(u,0)+ \beta \mc U_s(1) p_{t-s}^{\rm Neu}(u,1)\big) ds \Big|^2 \Big] \\
      &\quad + \int_0^1 \int_0^t p^{\rm Neu}_{t-s}(u,v)^2 \bb E\big[ \mc U_s(v)^2\big] ds dv.
   \end{align*}
   Set now $\mc V_t = \sup_{u \in [0,1]} \bb E[\mc U_t(u)^2]$. Lemma~\ref{lem:Dirichlet-heat-kernel} gives
   \[
      \sup_{u \in [0,1]} \big\| p^{\rm Neu}_t(u,\cdot)\big\|_{L^\infty} \lesssim t^{-1/2},\qquad \sup_{u \in [0,1]} \big\| p^{\rm Neu}_t(u,\cdot)\big\|_{L^1} \lesssim 1,
   \]
    and we apply the Cauchy-Schwarz inequality for the measure $A \mapsto \int_0^t \mb 1_A(s) (t-s)^{-1/2} ds$ to obtain, uniformly in $t \in [0,T]$,
   \[
      \mc V_t \lesssim \int_0^t \mc V_s (t-s)^{-1/2} ds.
   \]
   Now we can simply iterate this inequality to see that $\mc V_t = 0$ for all $t \in [0,T]$, for details see~\cite[Theorem~3.2]{Walsh1986}. This concludes the proof of Proposition~\ref{def:she-rob}. \qed

Having proved uniqueness, let us prove an additional moment bound for the solution to~\eqref{eq:she-rob-strong}, which has been used in Section \ref{ssec:map}.

\begin{lemma}\label{lem:she-moment}
   Assume that $\Phi_0$ satisfies $\sup_{u \in [0,1]} \bb E[|\Phi_0(u)|^p] = M <\infty$ for some $p \in (6,\infty)$ and let $\{\Phi_t \; ; \; t \in [0,T]\}$ solve~\eqref{eq:she-rob-strong} with initial condition $\Phi_0$. Then there exists a constant $C>0$ that only depends on $T$ and $\alpha,\beta,A,D$ but not on $M$, such that
   \[
      \sup_{t\in [0,T],\, u \in [0,1]} \bb E[|\Phi_t(u)|^p] \le C\times  M.
   \]
\end{lemma}

\begin{remark}
   Of course, the same bound holds for $p \in [1,\infty)$. But for $p>6$ the proof slightly simplifies.
\end{remark}

\begin{proof}[Proof of Lemma \ref{lem:she-moment}]
   The following argument is also essentially contained in~\cite[Theorem~3.2]{Walsh1986}. By the Burkholder-Davis-Gundy inequality and Jensen's inequality we have uniformly in $t \in [0,T]$ and locally uniformly in $\alpha,\beta$
   \begin{align*}
       \bb E[|\Phi_t(u)|^p] & \lesssim \int_0^1 \bb E\big[|\Phi_0(v)|^p\big] p^{\rm Neu}_t(u,v) dv \\
       &\quad + \int_0^t \big(p^{\rm Neu}_{t-s}(u,0) \bb E\big[|\Phi_s(0)|^p\big] + p^{\rm Neu}_{t-s}(u,1) \bb E\big[|\Phi_s(1)|^p\big]\big) ds \\
       &\quad + \bb E \bigg[ \Big| \int_0^1 \int_0^t  p^{\rm Neu}_{t-s}(u,v)^2  \Phi_s(v)^2 ds dv \Big|^{p/2} \bigg],
   \end{align*}
   and we bound the integral inside the expectation with H\"older's inequality by
   \begin{multline*}
       \Big| \int_0^1 \int_0^t  p^{\rm Neu}_{t-s}(u,v)^2  \Phi_s(v)^2 ds dv \Big|^{p/2} \\ \le  \Big| \int_0^1 \int_0^t  p^{\rm Neu}_{t-s}(u,v)^{2 q} ds dv \Big|^{p/(2q)} \int_0^1 \int_0^t |\Phi_s(v)|^p ds dv
   \end{multline*}
   for $q = p /(p-2)$, the conjugate exponent of $p/2$. To control the integral over the Neumann heat kernel we first integrate in the space variable and bound
   \[
      \big\| p^{\rm Neu}_{t-s}(u,\cdot)^{2q}\big\|_{L^1} \le \big\| p^{\rm Neu}_{t-s}(u,\cdot)^{2q-1}\big\|_{L^\infty} \big\| p^{\rm Neu}_{t-s}(u,\cdot)\big\|_{L^1} \leq C (t-s)^{1/2-q},
   \]
   where the last step follows from Lemma~\ref{lem:Dirichlet-heat-kernel}. Since $p>6$ we have $q = p/(p-2) < \frac32$, and therefore $(t-s)^{1/2-q}$ is integrable on $[0,t]$ and the integral is uniformly bounded in $t \in [0,T]$. Thus, we have shown that $\mc V_t = \sup_{u \in [0,1]} \bb E[|\Phi_t(u)|^p]$ satisfies on $[0,T]$
   \begin{align*}
      \mc V_t & \lesssim M + \int_0^t \big(1 + (t-s)^{-1/2}\big)\; \mc V_s ds \\
      & \lesssim M + \int_0^t \big(1 + (t-s)^{-1/2}\big) \bigg( M + \int_0^s (1 + (s-r)^{-1/2}) \mc V_r dr \bigg)  ds \\
      & \lesssim M + \int_0^t \int_r^t \big(1 + (t-s)^{-1/2}\big)\big(1 + (s-r)^{-1/2}\big)\; \mc V_r \; ds  dr \\
      & \lesssim M + \int_0^t \mc V_r dr,
   \end{align*}
   where the last step is a simple computation: after expanding the product the most complicated integrand is  $(t-r)^{-1/2}(s-r)^{-1/2}$, for which the integral is computed in~\cite[p.315]{Walsh1986}. Now the claim follows from Gronwall's Lemma.
\end{proof}

\section{Some considerations about the microscopic Cole-Hopf process}\label{app:micro-CH}

Since the major breakthrough made by \cite{BerG}, the KPZ behavior for interacting particle systems has often been investigated through the so-called Cole-Hopf process, and it has been -- until very recently -- the only way to prove the convergence of the microscopic height function. We note however, that this approach has only been possible for very particular microscopic dynamics which allow the Cole-Hopf transformation, such as the one we consider in this paper. Even if here we do not need to use this transformation in order to prove the KPZ--type macroscopic fluctuations, we want to highlight in this section the fact that the specific Robin boundary conditions that we obtain for the stochastic heat equation also emerge from the microscopic Cole-Hopf transformation.

In what follows, we assume $\gamma=\tfrac12$ because we just want to recover the Robin boundary conditions obtained in Proposition \ref{prop:link}. For that purpose, we  define the \emph{microscopic Cole-Hopf transformation} as follows: we set 
\begin{equation}
\label{eq:xitn}
\xi_t^n(x) = \exp\Big( \tfrac{\theta_n}{\sqrt n} \; h_t^n(x)+\lambda_n t\Big), \quad \text{for any } x \in \{1,\dots,n\},
\end{equation}
where $\theta_n$ is such that: 
\begin{equation}\label{eq:gammalambda}
\exp\Big(\tfrac{\theta_n}{\sqrt n}\Big) = 1+ \tfrac{E}{\sqrt n},
\end{equation}
and $\lambda_n$ will be chosen ahead. 
Note that $\theta_n\ge 0$.  We will see below in Section \ref{sec:appCH} the reason for this choice  of $\theta_n$. 
\begin{definition}\label{def:current}
For any $t>0$ let $\mc J_t^n$ be the \emph{current fluctuation field} which acts on functions $\varphi \in \SN$ as: 
\[ \mc J_t^n(\varphi)=\frac{1}{n} \sum_{x=1}^n \varphi\big(\tfrac x n \big) \xi_{tn^2}^n(x). \]
\end{definition}

\begin{remark}
Before proceeding we note that by the conservation law, for any $x\in\Lambda_n$, we have that
\begin{equation}\label{eq:cons_law}
\eta_{t}(x)-\eta_0(x)=J_{x-1,x}(t)-J_{x,x+1}(t)+ \mathbf{1}_{\{1,n-1\}}(x)\;C_t^n(x),
\end{equation}
where $J_{x,x+1}(t)$ is the counting process for the net number of particles at the bond $\{x,x+1\}$ during the time interval $[0,t]$,  with, by convention, $J_{0,1}(t)=0=J_{n-1,n}(t)$. For $x=1$ and $x=n$, the process $C_t^n(x)$  counts the number of particles created at the site $x$ minus the number of particles destroyed at the site $x$ during the time interval $[0,t]$. Also note that $C^n_t(1)=-h_t^n(1)$, where $h_t^n(1)$ was introduced in \eqref{eq:htn}. 
From \eqref{eq:cons_law} we have an equivalent way of defining the microscopic Cole-Hopf transform as
\begin{equation}\label{eq:equixi}\xi_t^n(x) =\exp \bigg( \tfrac{-\theta_n}{\sqrt n} \; \Big(J_{x-1,x}^n(t)-\sum_{y=1}^{x-1} \eta_{0}(y)\Big)\bigg)\exp\Big(-\tfrac{\theta_n(x-1)}{2\sqrt n}+\lambda_n t\Big).\end{equation}
The first exponential on the right hand side of the last equality corresponds to the Cole-Hopf transformation used in \cite{GLM}, {except that our scaling is different, since the strength of asymmetry is taken here with $\gamma=\frac12$, and not $\gamma=1$ as in \cite{GLM} (this explains the factor $\frac{1}{\sqrt n}$ in \eqref{eq:equixi} instead of $\frac1n$).  

Note also that, up to the level of the current fluctuation field $\mc J^n$,  both definitions of the Cole-Hopf process basically coincide: in \cite{GLM} the current field  is defined  on the discrete gradient of the test function $\varphi$ which is assumed to vanish at the boundary,  see (3.1) and the equation above (3.11) in \cite{GLM}. Here, instead we take test functions in $\SN$ (Definition \ref{def:current}). Moreover, the extra factor at the right hand side of \eqref{eq:equixi} corresponds to the average of the microscopic Cole-Hopf variables and also appears in the definition of the current field in \cite{GLM} (formula above (3.11)). }
\end{remark}

\begin{remark}
Note that, when $t=0$, one can compute the average of the current fluctuation field with respect to $\bb E_\rho$ as follows: 
\begin{align*} 
\bb E_\rho\big[\mc J_0^n(\varphi) \big]& = \frac{1}{n} \sum_{x=1}^n \varphi\big(\tfrac x n\big) \bb E\big[\xi_0^n(x)\big] \\
& = \frac{1}{n}\sum_{x=1}^n  \varphi\big(\tfrac x n\big) \Big(\tfrac12\big(e^{\theta_n/(2\sqrt n)}+e^{-\theta_n/(2\sqrt n)}\big)\Big)^{x-1}\\
& \xrightarrow[n\to\infty]{} \int_0^1 \varphi(u) e^{u E^2/2} \; du.
\end{align*}
\end{remark}

Without entering too much into details, let us give in this last section some properties of the microscopic Cole-Hopf transformation that one could prove using some well-known past works.

 \subsection{Martingale decomposition for the current fluctuation field $\mc J_t^n$} \label{sec:appCH}
 
Let us first  write the martingale decomposition that the current field satisfies, in the case $\gamma=\frac12$. For that purpose we take $\varphi$ as a test function, which will be chosen ahead, and denote as before $\varphi_x=\varphi(\frac x n)$. Recall also from \eqref{eq:xitn}
 \[\xi_t^n(x) = \exp\bigg( \frac{\theta_n}{\sqrt n} \Big( h_t^n(1) + \sum_{y=1}^{x-1} \bar\eta_{t}(y)\Big)+\lambda_n t \bigg), \quad \text{for any } x \in \{1,\dots,n\}.\]
First, we are going to explain the choice for $\theta_n$ and $\lambda_n$. By using the usual convention $\eta_{t}(0)=\eta_{t}(n)=\frac12$, since $\xi_t^n$ is a function of $\eta$ and $h_t^n(1)$, recalling \eqref{eq:joint_gen}, we have, for $x \in \{1,\dots,n\}$, that
 \begin{equation*}
 \begin{split}
\xi_{{tn^2}}^n(x) -\xi_{{0}}^n(x)-&\int_0^t n^2\xi_{{sn^2}}^n(x)\Big[ \lambda_n    +  \big(e^{\theta_n/\sqrt n}-1\big)\big(1-\eta_{sn^2}(x-1)\big)\eta_{sn^2}(x) \Big]ds\\
  -&  \int_0^t n^2\xi_{{sn^2}}^n(x)\Big[\big(e^{-\theta_n/\sqrt n}-1\big)\big(1+\tfrac{E}{\sqrt n}\big)\eta_{sn^2}(x-1)\big(1-\eta_{sn^2}(x)\big)\Big]ds,
\end{split} 
  \end{equation*}
  is a martingale.  
From the choice \eqref{eq:gammalambda} we see that the integral term above simplifies to 
 \begin{equation}\label{eq:mart_dec_CH}
\int_0^t n^2\xi_{{tn^2}}^n(x) \lambda_n+En^{3/2}\xi_{{sn^2}}^n(x)\big[\eta_{sn^2}(x)-\eta_{sn^2}(x-1)\big] ds. 
 \end{equation} 
 In order to close the equation above in terms of $\xi_{{sn^2}}^n(x)$ we note that from the trivial identities:
\begin{equation}\label{eq:trivial_id}
\begin{split}
&\xi_{{sn^2}}^n(x)\eta_{sn^2}(x-1)a_n+\xi_{{sn^2}}^n(x){b_n}=\xi_{{sn^2}}^n(x-1)-\xi_{{sn^2}}^n(x), \text{ for } \, x \in \{2,\dots, n\}\\ &\xi_{sn^2}^n(x)\eta_{sn^2}(x)a_n+\xi_{{sn^2}}^n(x){b_ne^{-\theta_n/(2\sqrt n)}}=\xi_{{sn^2}}^n(x)-\xi_{{sn^2}}^n(x+1), \textrm{ for } \, x \in \{1,\dots, n-1\}\
\end{split}
 \end{equation}
 with $a_n=(e^{-\theta_n/(2\sqrt n)}-e^{\theta_n/(2\sqrt n)})$ and ${b_n=e^{\theta_n/(2\sqrt n)}-1}$, the expression \eqref{eq:mart_dec_CH}, for $x\in\{2,\dots, n-1\}$, is reduced to 
  \begin{equation*}
\int_0^t n^2\xi_{{sn^2}}^n(x) \lambda_n- \frac{En^{3/2}}{a_n}\Delta \xi_{{sn^2}}^n(x)+\frac{En^{3/2}}{a_n}\xi_{{sn^2}}^n(x){b_n^2 e^{-\theta_n/(2\sqrt n)}}ds. 
 \end{equation*} 
  Now note that
   for the choice $\sqrt n \lambda_n=-\frac{E}{a_n}{b_n^2 e^{-\theta_n/(2\sqrt n)}}$ the last identity becomes
    \begin{equation*}
 \int_0^t -\frac{En^{3/2}}{a_n}\Delta \xi_{{sn^2}}^n(x) ds.
 \end{equation*} 
 Also note that 
 \begin{equation} \label{eq:solution}
\theta_n \xrightarrow[n\to\infty]{} E\; ; \quad \quad  
a_n   \underset{n\to\infty}{\simeq} \frac{-E}{\sqrt n}\; ; \quad \quad  {\lambda_n\underset{n\to\infty}{\simeq}\frac{E^2}{4n}}.
\end{equation}
To treat  the boundary we do the following. For $x=1$ (resp. $x=n$) we plug the second (resp. first) identity of \eqref{eq:trivial_id} in \eqref{eq:mart_dec_CH} to arrive at \eqref{eq:mart_dec_CH} for $x=1$ (resp. $x=n$) written  as 
 \begin{equation}\label{eq:mart_dec_CH_boundary}\begin{split}
\int_0^t n^2\xi_{{sn^2}}^n(1) \lambda_n-\frac{En^{3/2}}{a_n}\nabla^+\xi_{{sn^2}}^n(1)-En^{3/2}\xi_{{sn^2}}^n(1)\bigg(\frac{{b_ne^{-\theta_n/(2\sqrt n)}}}{a_n}+\frac{1}{2}\bigg)ds, \\
\Big(\textrm{resp.} \int_0^t n^2\xi_{{sn^2}}^n({n}) \lambda_n+\frac{En^{3/2}}{a_n}\nabla^-\xi_{sn^2}^n(n)+En^{3/2}\xi_{{sn^2}}^n(n)\bigg(\frac{{b_n}}{a_n}+\frac{1}{2}\bigg) ds.\Big)  \end{split}
 \end{equation} 
Collecting the previous observations we conclude that for any $x\in\{1,\dots,n\}$, 
\[\xi_{{tn^2}}^n(x)-\xi_0^n(x) - \int_0^t (T_n\xi_{{sn^2}}^n)(x) \; ds\] is a martingale, where $T_n$ is given by
\begin{equation}
\label{eq:omega}
\begin{cases}
T_n\xi^n(1) & =n\alpha_n\xi^n(1){+D_n\; n^2}\nabla^+ \xi^n(1), \\ 
T_n\xi^n(x) & ={D_n\; n^2} \Delta \xi^n(x), \qquad \qquad \qquad \text{if } x \in\{2,\dots,n-1\} \\
T_n\xi^n(n) & = n{\alpha_n} \xi^n(n) {-D_n\; n^2} \nabla^-\xi^n(n),
\end{cases}
\end{equation}
where 
\[D_n = - \frac{E}{a_n \; \sqrt n}, \quad \qquad  \alpha_n=n\lambda_n-{\frac{E\sqrt n}{2}\;\bigg(\frac{2-e^{\theta_n/(2\sqrt n)} - e^{-\theta_n/(2\sqrt n)}}{a_n}\bigg)}\]
and from \eqref{eq:solution} we get that $\alpha_n \to   \frac{E^2}{8}$  {and $D_n\to 1$}, as $n\to\infty$.
Now, from Dynkin's formula, we compute the martingale decomposition  for the current fluctuation field $\mc J_t^n$. From the computations above and by doing a summation by parts, we get that 
\begin{equation}\label{eq:mar_curr}
\mc J_t^n(\varphi)  - \mc J_0^n(\varphi)  -\int_{0}^t 
\frac1n \sum_{x=1}^n  (T_n \xi_{{sn^2}}^n)(x) \varphi_x  ds
\end{equation} is a martingale, where the integral term above can be rewritten as
\begin{equation*}
\begin{split}
\int_0^t \frac{1}{n} \sum_{x=1}^{n-1} {D_n}\; {\Delta}_n\varphi_x \; \xi_{{sn^2}}^n(x) ds &    +\int_0^t\frac 1n \Big({D_n\; n^2}\big( \varphi_1 - \varphi_0\big) + n\alpha_n \varphi_1\Big) \xi_{{sn^2}}^n(1) ds \vphantom{\bigg(} \\
&+\int_0^t\frac 1n \Big( {D_n\; n^2} \big( \varphi_{n-1} - \varphi_{n}\big) + n{\alpha_n} \varphi_n\Big) \xi_{{sn^2}}^n(n)ds. \vphantom{\bigg(}
\end{split}
\end{equation*}
Now to close the field let us sum and subtract, inside the previous sum, the term $\Delta_n \varphi_{n-1} \xi_{sn^2}^n(n)$ so that the last expression becomes:
\begin{equation*}\begin{split}
\int_0^t
\frac1n \sum_{x=1}^n  (T_n \xi_{{sn^2}}^n)(x) \varphi_x ds &=\int_0^t \frac{1}{n} \sum_{x=1}^{n-1} {D_n} \; \tilde{\tilde\Delta}_n\varphi_x \; \xi_{{sn^2}}^n(x) ds \\
& +\int_0^t\frac 1n \Big( {D_n\; n^2}\big( \varphi_1 - \varphi_0\big) + n\alpha_n \varphi_1\Big) \xi_{{sn^2}}^n(1)ds \vphantom{\bigg(} \\
&+\int_0^t\frac 1n \Big(  {D_n\; n^2} \big( \varphi_{n-1} - \varphi_{n}-\tilde{\tilde\Delta}_n\varphi_n \big) +n{\alpha_n} \varphi_n\Big) \xi_{{sn^2}}^n(n) ds. \vphantom{\bigg(} 
\end{split}\end{equation*}
where $\tilde{\tilde\Delta}_n$ is a slightly different approximation of the discrete  Laplacian  defined as 
\begin{equation}
\label{eq:lap3}
\tilde{\tilde\Delta}_n\varphi\big(\tfrac x n\big) = \begin{cases}
\Delta_n\varphi\big(\tfrac x n\big) & \text{ if } x \in\{1,\dots,n-1\},\\
\Delta_n\varphi\big(\tfrac {n-1} n\big) & \text{ if } x = n.
\end{cases}
\end{equation}
Putting everything together, we rewrite \eqref{eq:mar_curr} as 
\begin{align} \mc J_t^n(\varphi)-\mc J_0^n(\varphi) &{- D_n}  \int_0^t \mc J_s^n(\tilde{\tilde\Delta}_n\varphi)ds \label{eq:CH0}\\ 
&- \frac 1n\Big( {D_n\; n^2}\big(\varphi\big(\tfrac 1n\big)-\varphi(0)\big)+ n\alpha_n\varphi\big(\tfrac1n\big)\Big)\int_0^t \xi_{{sn^2}}^n(1) ds \label{eq:CH1}\\
& - \frac 1n\Big({D_n\; n^2} \big(3\varphi\big(\tfrac{n-1}{n}\big)-2\varphi(1)-\varphi\big(\tfrac{n-2}{n}\big)\big)+n\beta_n\varphi(1)\Big) \int_0^t \xi_{{sn^2}}^n(n) ds. \label{eq:CH2}
\end{align}
We note that last expression is a martingale whose  quadratic variation  is given by
\begin{multline}\label{eq:QV-CH}
\int_0^t \frac{E^2}{n} \sum_{x=1}^n\big(\varphi\big(\tfrac x n\big)\big)^2\big(\xi_{{sn^2}}^n(x)\big)^2 \Big[\eta_{sn^2}(x)\big(1-\eta_{sn^2}(x-1)\big) \\ + \eta_{sn^2}(x-1)\big(1-\eta_{sn^2}(x)\big) e^{-\theta_n/\sqrt n}\Big]ds.
\end{multline}
Note that in \eqref{eq:CH2} $$\big(3\varphi\big(\tfrac{n-1}{n}\big)-2\varphi(1)-\varphi\big(\tfrac{n-2}{n}\big)\big)=-\tfrac{1}{n}\nabla \varphi(1)+ O(n^{-2}).$$

\subsection{Asymptotic limit}

From \eqref{eq:CH0}--\eqref{eq:CH2}, one sees that there are two natural ways to close the martingale problem. First, note that if $\varphi \in \mc C^\infty([0,1])$, then $\tilde{\tilde\Delta}_n\varphi$ is an approximation of $\Delta\varphi$. Moreover,
\begin{itemize}\item if $\varphi$ satisfies the Robin boundary condition 

\begin{enumerate}[(i)]
\item $\nabla\varphi(0) = -\frac{E^2}{8}\varphi(0)$, then \eqref{eq:CH1} is of order $1/n$ in $L^2(\bb P_\rho)$ and therefore vanishes as $n\to\infty$;
\medskip

\item $\nabla\varphi(1)= \frac{E^2}{8}\varphi(1)$, then \eqref{eq:CH2} is also of order $1/ n$ and vanishes. \end{enumerate}
\medskip

\item If, however, $\varphi \in \SN$, then $n(\varphi_1-\varphi_0) \to 0$ as $n\to\infty$, and therefore only one term remains in \eqref{eq:CH1}, which reads
\[  \alpha_n  \varphi\big(\tfrac 1 n\big)\int_0^t \xi_s^n(1)ds  \underset{n\to\infty}{\simeq} \tfrac{E^2}{8} \varphi\big(\tfrac 1 n\big)\int_0^t\xi_s^n(1)ds. \]
In the macroscopic limit, the last term will correspond to 
\[ \tfrac{E^2}{8} \int_0^t  \Phi_s(0)\varphi(0)ds\] in the definition \eqref{eq:she-rob-weak} of the solution to the SHE with Robin boundary condition, as soon as one proves convergence of $\xi_{{tn^2}}^n$ in ${\mc D}([0,T],\mc C([0,1]))$. \end{itemize}

\subsection{Exponential moments and quadratic variation} 

Finally, one might check that the quadratic variation \eqref{eq:QV-CH} converges to
\[\frac{E^2 t}{2} \int_0^1 |\Phi_s(u)|^2\varphi^2(u)du,\] where $\Phi_\cdot$ is the limit of the current field $\mc J_\cdot^n$ in $\mc D([0,T],\mc C([0,1])$. Heuristically, this is indeed the case if one is able to replace $\xi_{sn^2}^n(x)^2$ in \eqref{eq:QV-CH} with $(\bb E[\xi_{sn^2}^n(x)])^2$. This could be proved by using some ideas taken from \cite[Lemma 4.3]{GLM} which permit to control all the exponential moments
\[\sup_{x\in\{1,\dots,n\}} \bb E_\rho \big[ \big(\xi_{sn^2}^n(x)-\bb E[\xi_{sn^2}^n(x)]\big)^k\big],\]
with $k\in\mathbb N$. 

\section*{Acknowledgements}

We would like to thank Ivan Corwin and Hao Shen for sending and discussing with us their work~\cite{Corwin2016} when it was still a draft.

This work has been supported by the French Ministry of Education through the grant  EDNHS ANR-14-CE25-0011.  PG thanks FCT/Portugal for support through the project UID/MAT/ 04459/2013. NP thanks the DFG for financial support via Research Unit FOR 2402. MS thanks Labex CEMPI (ANR-11-LABX-0007-01) for its support.

This project has received funding from the European Research Council (ERC) under  the European Union's Horizon 2020 research and innovative programme (grant agreement  No 715734). 

This work was finished during the stay of PG and MS at Institut Henri Poincar\'e - Centre Emile Borel during the trimester "Stochastic Dynamics Out of Equilibrium". PG and MS thank this institution for hospitality and support.

\bibliography{bibliography}
\bibliographystyle{plain}

\end{document}